\documentclass{amsart}
\usepackage[normalem]{ulem}
\usepackage{amsmath,amsthm,amsfonts,amscd,amssymb} 
\usepackage{verbatim}      	% Allows quoting source with commands.
\usepackage{mathtools}
\usepackage{epsfig}         	% Allows inclusion of eps files.
\usepackage{url}		% Allows good typesetting of web URLs.
\usepackage{epic,eepic,latexsym}
\usepackage{graphicx}
\usepackage{color}
\usepackage{lcd}
\usepackage{mathrsfs}
\usepackage[centering,text={15.5cm,22cm},
        marginparwidth=20mm,dvips]{geometry}
\usepackage[linktocpage]{hyperref}
\usepackage{lscape}
\usepackage{tabularx}
\usepackage{marginnote}
\usepackage[all]{xy}
\usepackage{enumitem}
\usepackage{stmaryrd}
\usepackage{tipa}
\usepackage{upgreek, comment}
\usepackage{theoremref}

\newcommand{\QQ}{\mathbb{Q}}

\DeclareMathOperator{\Hom}{Hom}

\DeclareMathOperator{\Pic}{Pic}

\DeclareMathOperator{\Proj}{Proj}

\DeclareMathOperator{\im}{im}
\DeclareMathOperator{\rank}{rank}
\DeclareMathOperator{\Spec}{Spec}

\DeclareMathOperator{\Id}{Id}

\DeclareMathOperator{\val}{val}

\DeclareMathOperator{\sign}{sgn}
\DeclareMathOperator{\SL}{SL}
\DeclareMathOperator{\PGL}{PGL}

\DeclareMathOperator{\trop}{trop}

\DeclareMathOperator{\prin}{prin}

\DeclareMathOperator{\In}{in}

\DeclareMathOperator{\Ends}{Ends}

\DeclareMathOperator{\Joints}{Joints}

\DeclareMathOperator{\Newt}{Newt}

\DeclareMathOperator{\Supp}{Supp}

\DeclareMathOperator{\midd}{mid}

\DeclareMathOperator{\can}{can}

\DeclareMathOperator{\Scat}{Scat}

\DeclareMathOperator{\Cox}{Cox}

\DeclareMathOperator{\gp}{gp}
\DeclareMathOperator{\diag}{diag}

\let\?=\overline
\let\:= \colon
\let\llb=\llbracket
\let\rrb=\rrbracket
\let\bb=\mathbb

\let\rar=\rightarrow
\let\f=\mathfrak
\let\s=\mathcal

\let\wh=\widehat
\let\wt=\widetilde

\def\risom{\buildrel\sim\over{\smashedlongrightarrow}}
 \def\smashedlongrightarrow{\setbox0=\hbox{$\longrightarrow$}\ht0=1.25pt\box0}

\newcommand {\kk} {\Bbbk}

\newcommand {\sQ} {\mathcal{C}}

\theoremstyle{plain}% default
 \newtheorem{thm}{Theorem}[section]
 
 \newtheorem{rem}[thm]{Remark}
 \newtheorem{lem}[thm]{Lemma}
  \newtheorem{lemdfn}[thm]{Lemma/Definition}
  \newtheorem{prop}[thm]{Proposition}
  
   \newtheorem{cor}[thm]{Corollary}

 \newtheorem{claim}[thm]{Claim}
 \newtheorem{warn}[thm]{Warning}

\newcommand{\C}{{\mathbb C}}

\newcommand{\Q}{{\mathbb Q}}
\newcommand{\R}{{\mathbb R}}
\newcommand{\Z}{{\mathbb Z}}
\newcommand{\N}{{\mathbb N}}

\newcommand{\eq}[2]{\begin{equation}\label{#2} \begin{split} #1  \end{split} \end{equation}}
\newcommand{\eqn}[1]{\begin{equation*} \begin{split} #1 \end{split} \end{equation*}}

\theoremstyle{definition}

 \newtheorem{eg}[thm]{Example}

\theoremstyle{remark}

  \newtheorem{ass}[thm]{Assumption}

\newcommand{\lrp}[1]{\left(#1\right)}

\newcommand{\lrc}[1]{\left\{#1\right\}}

\usepackage{tikz}
\usetikzlibrary{arrows,decorations.pathmorphing,backgrounds,positioning,fit,matrix}

\tikzstyle{mutable}=[inner sep=0.5mm,circle,draw,minimum size=2mm]
\tikzstyle{frozen}=[inner sep=0.5mm,draw,minimum size=2mm]
\tikzstyle{dot} = [inner sep=0.5mm,circle,draw,minimum size=1mm]
\tikzstyle{coeff}=[inner sep=0.5mm,circle,fill=white]
\tikzstyle{wall}=[draw=blue!50,thick]

\tikzstyle{dot} = [fill=black!25,inner sep=0.5mm,circle,draw,minimum size=1mm]
\usetikzlibrary{patterns}
\newenvironment{exam}{\refstepcounter{thm}\begin{proof}[Example \emph{\thethm}]}{\end{proof}}

\title{Valuative independence and cluster theta reciprocity}
\author{Man-Wai Cheung}
\email{mandywai24{\char'100}gmail.com}
\author{Timothy Magee}
\address{Mathematics, Statistics, and Computer Science Dept \\ Hollins University \\ Roanoke, VA 24020 \\ USA}
\email{mageetd{\char'100}hollins.edu}
\author{Travis Mandel}
\address{Department of Mathematics \\ University of Oklahoma \\ Norman, OK 73019 \\ USA}
\email{tmandel{\char'100}ou.edu}
\author{Greg Muller}
\address{Department of Mathematics \\ University of Oklahoma \\ Norman, OK 73019 \\ USA}
\email{gmuller{\char'100}ou.edu}

\begin{document}

\begin{abstract}
We prove that theta functions constructed from positive scattering diagrams satisfy valuative independence.  That is, for certain valuations $\val_{v}$, we have $\val_v(\sum_u c_u \vartheta_u)=\min_{c_u\neq 0} \val_v(\vartheta_u)$.  As applications, we prove linear independence of theta functions with specialized coefficients and characterize when theta functions for cluster varieties are unchanged by the unfreezing of an index.  This yields a general gluing result for theta functions from moduli of local systems on marked surfaces.  We then prove that theta functions for cluster varieties satisfy a symmetry property called theta reciprocity: briefly, $\val_v(\vartheta_u)=\val_u(\vartheta_v)$.  For this we utilize a new framework called a ``seed datum'' for understanding cluster-type varieties.  One may apply valuative independence and theta reciprocity together to identify theta function bases for global sections of line bundles on partial compactifications of cluster varieties.
\end{abstract}

\maketitle

\setcounter{tocdepth}{1}
\tableofcontents  

\section{Introduction}

\subsection{The Valuative Independence Theorem and Theta Reciprocity}\label{sec1.1} A variety $U$ is \emph{log Calabi-Yau (log CY)} if it is isomorphic to the complement of an anticanonical divisor $D$ with normal crossings in a smooth projective variety $Y$.  The pair $(Y,D)$ is called a minimal model for $U$.  One says $U$ has maximal boundary if $D$ contains a $0$-stratum.  A fundamental expectation of the Gross-Siebert program is that the coordinate ring of a sufficiently nice affine log CY variety $U$ with maximal boundary should have a distinguished basis of \emph{theta functions} \cite{CPS,GHK1,GHKK,GSTheta,GHS,GSInt2,KY,KY2}. If $U^\vee$ is a mirror dual affine log CY variety, then the theta functions on $U$ should be naturally parametrized by the \emph{integral tropical points} of $U^\vee$, denoted $U^{\vee}(\Z^{\min})$.  Here, one may define $U^{\vee}(\Z^{\min})$ as the \emph{boundary divisorial valuations} on the field of rational functions of $U^{\vee}$---i.e., $\bb{Z}_{\geq 0}$-multiples of valuations along boundary components of minimal models for $U^{\vee}$.

Similarly, an integral tropical point $v$ of $U$ determines two objects:
\begin{itemize}
    \item a boundary divisorial valuation $\val_v$ on the coordinate ring of $U$;
    \item a theta function $\vartheta_v$ in the coordinate ring of $U^\vee$.
\end{itemize}
Given an integral tropical point $u$ of the mirror dual $U^\vee$, the \textbf{theta pairing} of $v$ and $u$ is the integer obtained by plugging the theta function of $u$ into the boundary valuation of $v$:
\[
\val_v(\vartheta_u).
\]

In this paper, we consider this pairing when $U$ is a \textbf{cluster variety}---i.e., a log CY variety with maximal boundary constructed (up to codimension $2$) by gluing together algebraic tori via certain birational maps called mutations.   We do not assume $U$ or $U^{\vee}$ are affine, so the theta functions may be formal Laurent series (in which case valuations may take the value $-\infty$).

\begin{thm}\label{thm:main}
Let $U$ be a cluster variety, and let $U^\vee$ be a mirror dual\footnote{Theorem \ref{thm:taut-pairing} proves theta reciprocity for ``chiral dual'' pairs.  Proposition \ref{lem:D-move} and Corollary \ref{cor:Langlands} extend this to ``chiral Langlands dual'' and ``Langlands dual'' pairs under an integrality assumption on the exchange matrix.  In \cite{GHKK}, ``mirror'' refers to the Langlands dual.  See \S \ref{subsub:mirror} for more details.\label{foot:mirror}} cluster variety.
\begin{enumerate}
    \item (Valuative Independence) For all (formal) linear combinations $\sum_{u\in U^{\vee}(\Z^{\min})} c_{u} \vartheta_{u}$ and all $v\in U(\Z^{\min})$, we have
    \[
    \val_{v}\left(\sum_{u} c_u\vartheta_u\right)=\inf_{u:\,c_u\neq 0} \val_{v}(\vartheta_u)
    \]
    assuming the value on the right-hand side is finite.
    \item (Theta Reciprocity) For all $u\in U(\mathbb{Z}^{\min})$ and $v\in U^\vee(\mathbb{Z}^{\min})$, 
    \[
    \val_v(\vartheta_u) = \val_u(\vartheta_v)
    \]
    whenever both sides are finite.\footnote{We expect theta reciprocity to hold without the finiteness assumption.  This is discussed in \S \ref{subsub: infinite}.}
\end{enumerate}
\end{thm}

In particular, valuative independence implies \cite[Conj. 9.8]{GHKK}, while theta reciprocity is conjectured in \cite[Rmk. 9.11]{GHKK}.   In fact, Theorem \ref{thm:main} applies to a somewhat generalized version of cluster varieties constructed from ``seed data'' as defined in \S \ref{sec:cluster}.  It follows from \cite[\S 5]{GHK3} that this generality includes all log Calabi-Yau surfaces with maximal boundary (called Looijenga pairs in \cite{GHK1}).  Additionally, many of our results, including valuative independence, apply more generally to any theta functions constructed from a positive scattering diagram, \textbf{not just the cluster setting}.  In particular, the scattering diagrams associated to smooth affine log CY's in \cite{KY,KY2} are positive, so our arguments apply to this setting as well.  See \S \ref{subsub:general-VIT} for further discussion.

A simple but useful example to keep in mind is when $U$ is an algebraic torus (which may be viewed as a cluster variety with no mutable variables).  Let $M$ denote the character lattice and $N=\Hom(M,\Z)$ the cocharacter lattice, so $U=\bb{G}_m(N)$ has coordinate ring $\Z[x^M]$.  The mirror is the dual algebraic torus $U^{\vee}=\bb{G}_m(M)$ with coordinate ring $\Z[x^N]$.  The integral tropical points are $U(\Z^{\min})=N$ and $U^{\vee}(\Z^{\min})=M$, and the theta functions are just the monomials $\vartheta_m=x^m\in \Z[x^M]$ and $\vartheta_n=x^n\in \Z[x^N]$.  The theta pairing is simply the dual pairing between the lattices:
\[
\val_n(x^m)=m\cdot n=\val_m(x^n).
\]
More generally, the boundary divisorial valuation $\val_n$ on $\Z[x^M]$ associated to $n\in N$ is given by 
\[
\val_{n}\left(\sum_m c_m x^m\right)=\min_{m:\,c_m\neq 0} (m\cdot n).
\]

Since every cluster variety $U$ contains an open algebraic torus $T$, several parts of this story generalize. Let $M$ and $N$ denote the character and cocharacter lattices of $T$.
The coordinate ring of $U$ includes into the Laurent polynomial ring $\Z[x^M]$, and the theta functions are sent to positive integral Laurent polynomials (or formal Laurent series).  The integral tropical points of $U$ may still be identified with $N$, and the boundary valuation $\val_n$ associated to $n\in N$ may again be characterized by its values on monomials:
    \begin{align}\label{eq:valn}
    \val_{n}\left(\sum_m c_m x^m\right):=\inf_{m:\,c_m\neq 0} (m\cdot n).
    \end{align}
Therefore, the theta pairing can be identified with 
\begin{align}\label{eq:valnm}
    \val_{n}(\vartheta_m) = \inf (m'\cdot n)
\end{align}
where the infimum runs over all $m'$ in the support of the Laurent expansion of $\vartheta_m$.

Note that \eqref{eq:valn} makes sense for all $n\in N_{\R}\coloneqq N\otimes \R$.  Define $\vartheta_n^{\trop}:M_{\R}\rar \R\cup \{\pm \infty\}$, $m\mapsto \val_m(\vartheta_n)$.  Similarly, $\val_n:M\rar \Z\cup \{\pm \infty\}$, $m\mapsto \val_n(\vartheta_m)$.  By theta reciprocity, we have $\val_n=\vartheta_n^{\trop}|_M$.  Thus, all properties of tropical theta functions are automatically found to hold for the corresponding valuation functions.  In particular, \eqref{eq:valnm} implies that $\vartheta_n^{\trop}$, hence also $\val_n$ (extended to $M_{\R}$---we shall henceforth understand $\val_n$ to mean this extension), are convex and integral piecewise-linear.

We shall now discuss several geometric and algebraic consequences of the Valuative Independence Theorem (VIT) and theta reciprocity.

\subsection{Theta bases for line bundles}\label{sub:line}

\subsubsection{Line bundles supported on the boundary}
Consider a variety $Y$ with Zariski open subset $U$, and let $D=\sum_i D_i = Y\setminus U$, each $D_i$ denoting an irreducible component.  Suppose $\s{B}=\{f_j\}$ is an additive basis for $\Gamma(U,\s{O}_U)$.  E.g., $U$ may be a cluster variety with each $D_i$ being a toric boundary divisor for the cluster tori and with theta function basis $\s{B}$.  We are interested in identifying when some subset of $\s{B}$ forms a basis for $\Gamma(Y,\s{O}_Y)$, or more generally, for global sections of line bundles $\s{O}_Y(\sum_i a_i D_i)$.

Consider the setting where $Y$ is a toric variety, $U=\bb{G}_m(N)$ is the big torus orbit with coordinate ring $\Z[x^M]$, and $\s{B}=\{x^m: m\in M\}$ is the theta function basis.  By standard toric geometry, each boundary component $D_i$ corresponds to a primitive element $n_i\in N$ such that $\val_{D_i}=\val_{n_i}$.  Here $\val_{D_i}$ is the discrete valuation mapping a function $f$ to the order of vanishing/pole of $f$ along $D_i$.  It follows that $\Gamma(Y,\s{O}_Y(\sum_i a_i D_i))$ admits a basis of monomials consisting of those $x^m$ with $m\in M$ contained in the polytope $\bigcap_i \{m\cdot n_i \geq -a_i\}$.

In the completely general setup, a basis $\s{B}$ will typically \textit{not} yield bases for these line bundles on $Y$.  The issue is as follows.  In general, given a linear combination of basis elements $\sum_j b_j f_j$ ($b_j\neq 0$), one has
\begin{align*}
    \val_{D_i}(\sum_j b_j f_j) \geq \min_j \val_{D_i} (f_j).
\end{align*}
The possibility of strict inequality here (which can happen if the highest-order poles or lowest-order zeroes cancel with each other) means that functions $f_j$ which are \textit{not} sections of a line bundle $\s{L}$ might combine to yield a function which \textit{is} a section of $\s{L}$.  VIT says that this issue does not happen for theta functions on cluster varieties.

\begin{cor}\label{intro-cor-2}
    Let $U$ be a cluster variety with partial compactification $Y$ and $D=Y\setminus U$ an anticanonical divisor such that the distinguished volume form of $U$ has a simple pole along each component of $D$.\footnote{The distinguished volume form is given (up to sign) on each cluster $\Spec \kk[M]$ by $d\log x^{e_1} \wedge \ldots \wedge d\log x^{e_d}$ for any basis $e_1,\ldots,e_d$ of $M$.  The components of $D$ correspond to elements of $U^{\trop}(\Z)$ as in \cite[Def. 1.7]{GHK3}.  Equivalently, valuations along components of $D$ correspond to valuations along toric boundary divisors of the clusters.}  Whenever the theta functions form an additive basis\footnote{The theta functions at least yield an additive basis for $\Gamma(U,\s{O}_U)$ whenever the full Fock-Goncharov conjecture holds.  Some conditions ensuring this are given in \cite[Prop. 0.14]{GHKK}.} for $\Gamma(U,\s{O}_U)$, some subset will form an additive basis for $\Gamma(Y,\s{O}_Y)$.  Moreover, for any line bundle $\s{L}$ supported on $D$, there exists a convex polyhedron $P\subset M_{\R}$ such that $\{\vartheta_m\,|\,m\in P\cap M\}$ forms an additive basis for $\Gamma(Y,\s{L})$.
\end{cor}

For the final claim, consider $\s{L}=\s{O}_Y(\sum_i a_i D_i)$ with $D_i$ corresponding to primitive $n_i\in U(\Z^{\min})$.  The polytope $P$ is given by 
\begin{align*}
    \bigcap_i \left\{m\in M \,\left| \,\val_{n_i}(\vartheta_m)\geq -a_i \right.\right\}.
\end{align*}
As previously noted, theta reciprocity says that $\val_{n_i}=\vartheta_{n_i}^{\trop}$, and this is convex integral piecewise-linear.  So $P$ is indeed a convex polyhedron.

\begin{rem}[Quantum analogs]\label{rem:q}
    By the positivity of broken lines for quantum theta functions \cite{DM}, the exponents contributing to quantum theta functions are the same as those for their classical counterparts.  Since our definitions and arguments are based on these exponents, our results immediately carry over to the quantum setting.  In particular, if a quantum upper cluster algebra $\s{U}$ admits a basis of quantum theta functions, and if $\?{\s{U}}$ is a partially compactified version of $\s{U}$ as in \cite[Def. 2.5]{QY}, then there is a convex polyhedron $P\subset M_{\R}$ such that the quantum theta functions $\{\vartheta_m\,|\,m\in P\cap M\}$ form an additive basis for $\?{\s{U}}$.
\end{rem}

\subsubsection{Theta bases for Cox rings}\label{subsub:Cox}

Corollary \ref{intro-cor-2} applies to line bundles supported on $D$.  Here we discuss how one may find theta bases for arbitrary line bundles on $Y$ (up to isomorphism).

Let $Y$ be a partially compactified cluster variety as in Corollary \ref{intro-cor-2}.  Consider the Cox ring of $Y$:\footnote{Cox rings were defined in \cite{HK}.  When $\Pic(Y)$ is not torsion-free, defining $\Cox(Y)$ so that it admits a ring structure is a bit more complicated; the construction is due to \cite[\S 3]{BH}.  See \cite[Construction 4.3]{GHK3} or \cite[\S 3]{ManCox} for details in the present context.}
\begin{align*}
    \Cox(Y)\coloneqq\bigoplus_{\s{L}\in \Pic(Y)} \Gamma(Y,\s{L}).
\end{align*}
This ring $\Cox(Y)$ is itself a ring of functions on a cluster variety $Y'$ partially compactifying an open cluster variety $U'$---see \cite[\S 4]{GHK3} for the case $Y=U$ (in which case $Y'=U'$) and \cite[Thm. 3.6, Cor. 3.7]{ManCox} for general $Y$.  The space $Y'$ is called the universal torsor over $Y$.  The theta functions on $Y'$ are homogeneous with respect to $\Pic(Y)$-grading, so if the theta functions form a basis for $\Gamma(Y',\s{O}_{Y'})$, then they yield bases for the global sections of every line bundle on $Y$ (up to isomorphism), cf. \cite[Thm. 4.2]{ManCox}.   VIT lets us replace the condition of $Y'$ having a theta basis with the better-understood condition of $U'$ having a theta basis.

More precisely, $Y'$ can be realized as a cluster $\s{A}$-variety with special coefficients, together with the boundary divisors associated to the frozen vectors $\{e_i:i\in F\}$.  The frozen vectors here are chosen to map to primitive generators for the rays of the fan for $Y$.  We want to identify those theta functions on $U'$ which extend to regular functions on $Y'$.  I.e., we want those $\vartheta_m$ such that 
    \begin{align}\label{eq:regular}
        \text{$\val_{e_i}(\vartheta_m)\geq 0$ for each $i\in F$.}
    \end{align}
    As we first learned from Sean Keel, theta reciprocity gives us a nice way to do this.  Namely, let $W=\sum_{i\in F}\vartheta_{e_i}$ be the mirror Landau-Ginzburg potential associated to $Y$, and consider the rational polyhedral cone $$\Xi=\{m\in M_{\R}: W^{\trop}(m)\geq 0\}.$$  Since $W^{\trop}=\min_{i\in F} \vartheta_{e_i}^{\trop}$, theta reciprocity implies that the set of values $m$ satisfying \eqref{eq:regular} is precisely $\Xi\cap M$.

\subsection{Independence after specializing coefficients}

When proving VIT, the valuations are interpreted as infima of valuations of the monomials attached to broken lines.  As a key step in our proof, we show (Theorem \ref{thm:min-taut}) that the only broken lines which could possibly yield a minimum valuation are those which are ``taut'' as defined in \S \ref{sub:taut}. Our proof of VIT has several further consequences which we shall summarize here and in the subsections that follow. 

First, recall that theta functions are constructed over a ring of coefficients like $\Z[y^{\sQ}]=\sum_i a_i y^q$ ($a_i \in \Z$, $q\in \sQ$) for some monoid $\sQ$, or some completion $\Z\llb y^{\sQ}\rrb$.  The theta functions in general are formal Laurent series, but when they are actually finite Laurent polynomials, one can specialize the coefficients $y^{\sQ}$.  It is not obvious that the theta functions remain linearly independent under this coefficient-specialization.  E.g., in \cite[Thm. 0.3]{GHKK}, this linear independence corresponds to the injectivity of their map $\nu$, a property previously only known under additional assumptions (cf. Footnote \ref{foot:inj-nu}).  We obtain the following as a consequence of our proof of VIT.

\begin{thm}[Theorem \ref{thm:nu}]\label{thm:nu-intro}
    Theta functions which are finite Laurent polynomials maintain their valuative independence, hence also their linear independence, after specialization of coefficients.
\end{thm}

The corresponding injectivity of the map $\nu$ from \cite{GHKK} is given in Corollary \ref{cor:nu}.

In \S \ref{sec:sing}, we show as a corollary of Theorem \ref{thm:nu} that the condition of having enough global functions (EGF) is preserved under specialization of coefficients.  We then present an argument of Sean Keel showing that log Calabi–Yau varieties with maximal boundary and EGF have canonical singularities (Corollary \ref{cor:sing}).

\subsection{The Theta Function Extension Theorem and its applications}

Suppose we extend our scattering diagram $\f{D}_1$ (the combinatorial object from which our theta functions are constructed) to a new scattering diagram $\f{D}_2$, e.g., by adding new initial walls.  In general, this will change the associated theta functions---their Laurent expansions $\vartheta^{\f{D}_i}_{u,p}$ will pick up new terms coming from the new walls.  Here, $\f{D}_i$ indicates the scattering diagram, $u$ indexes the theta function, and $p$ corresponds to a point identifying a chamber of the scattering diagram, hence the local coordinate system used for taking the Laurent expansion.  The following condition ensures that in fact there are no new terms:

\begin{thm}[Theorems \ref{thm:theta-extension-general} and \ref{thm:theta-ext-seed}]\label{thm:ext-intro}
    If $\vartheta^{\f{D}_1}_{u,p}$ is a finite linear combination of theta functions $\vartheta^{\f{D}_2}_{u_j,p}$, then in fact
$\vartheta^{\f{D}_1}_{u,p}=\vartheta^{\f{D}_2}_{u,p}$.
\end{thm}

\begin{rem}
    As indicated in Remark \ref{rem:q}, Theorem \ref{thm:ext-intro} also applies to quantum theta functions for skew-symmetric quantum cluster algebras.  Indeed, the classical case above is equivalent to saying that the new walls of $\f{D}_2$ do not result in any new broken lines with ends $(u,p)$.  Then the positivity of the monomials attached to quantum broken lines \cite{DM} implies that there cannot be any new broken lines in the quantum setting either.
\end{rem}

The simplest examples of Theorem \ref{thm:ext-intro} are the global monomials: if one removes all scattering walls, then the theta functions are precisely the monomials.  Theorem \ref{thm:ext-intro} applies here to recover the result from \cite{GHKK} that global monomials are always theta functions.  

We next discuss two applications of the Theta Function Extension Theorem.  In \S \ref{subsub:moduli-local} we briefly consider cluster algebras associated to moduli of local systems on a marked surface $S$.  Theorem \ref{thm:ext-intro} applies to show that the theta functions are well-behaved under the operation of gluing together boundary arcs of $S$.  Then in \S \ref{sub:loop}, we apply Theorem \ref{thm:ext-intro} to describe a simple class of theta functions called ``loop elements.''  Briefly, if a quiver $Q$ contains the Kronecker quiver $K$ as a full subquiver, then we show, under certain assumptions, that the interesting theta functions for $K$ extend to also give theta functions for $Q$.

\subsubsection{Modularity of theta functions for moduli of local systems on marked surfaces}\label{subsub:moduli-local}

In \cite{FG0}, given a marked compact triangulable surface $S$ (possibly with boundary) and a semisimple group $G$ with Langlands dual group $G^{\vee}$, Fock and Goncharov define the moduli space $\s{A}_{G,S}$ of decorated twisted $G$-local systems on $S$, and also the moduli space $\s{X}_{G^{\vee},S}$ of framed $G^{\vee}$-local systems on $S$.  In the case where $G=\SL_n$ and $G^{\vee}=\PGL_n$, they define cluster algebra structures on the moduli spaces.  These cluster structures were extended to include other simply connected classical groups $G$ in \cite{Le} and then to more general semi-simple groups $G$ in \cite{GS:QGM}.\footnote{\cite{GS:QGM} uses ``pinnings'' to incorporate frozen variables in their $\s{X}$-spaces, calling the resulting moduli spaces $\mathscr{P}_{G^{\vee},S}$.  We continue referring to these spaces as $\s{X}_{G^{\vee},S}$ for simplicity.}  

Here and below, we implicitly assume that $S$ satisfies whatever admissibility conditions are needed to ensure nice moduli and cluster structures.  E.g., every component of $S$ and of $\partial S$ should contain at least one marked point, and more marked points in $S$ may be required for low-genus cases.  In particular, \cite[Condition 1 in \S 2.1.10]{GS:QGM} is satisfied.

The most well-understood cases are those with $G=\SL_2$, $G^{\vee}=\PGL_2$.  In this setting, Fock and Goncharov defined canonical functions \cite[Def. 12.4]{FG0}, reinterpreted combinatorially as the ``bracelets basis'' in \cite{MSW2}.  Recently, \cite{ManQin} proved that the bracelets basis coincides with the theta basis of \cite{GHKK}.  A key step in proving this result was the Gluing Lemma \cite[Lem. 8.2]{ManQin} which says (for unpunctured surfaces) that if a bracelet basis element $\beta$ on $\s{A}_{G,S}$ is a theta function, and we glue two boundary edges of $S$ together to produce a new surface $S'$, then the induced bracelet element $\beta'$ on $\s{A}_{G,S'}$ is automatically a theta function as well.

The proof of the Gluing Lemma in loc. cit. was specific to the setting there, relying on previously known properties of the bracelets basis.  However, the Theta Function Extension Theorem readily yields the following much more general result:

\begin{thm}
    Let $f$ be a function on one of the moduli spaces $\s{A}_{G,S}$ or $\s{X}_{G^{\vee},S}$ discussed above, and let $f'$ be the induced function on $\s{A}_{G,S'}$ or $\s{X}_{G^{\vee},S'}$, respectively, where $S'$ is obtained from $S$ by gluing two boundary arcs together.  If $f$ is a theta function, then $f'$ is as well.
\end{thm}

\begin{proof}
    In terms of the cluster structures, gluing together boundary arcs corresponds to first identifying pairs of frozen indices with each other and then unfreezing the resulting glued indices.  This process is called amalgamation in \cite{FG-amalg,GS:QGM}.  The fact that theta functions remain theta functions when gluing the frozen indices is a general phenomenon, cf. \cite[\S 6.3]{ManQin}.  So it remains to show that the theta functions remain theta functions after unfreezing.  By Theorem \ref{thm:ext-intro}, it suffices to show that $f'$ is a finite linear combination of theta functions.  By construction, $f'$ is a regular function on $\s{A}_{G,S'}$ or $\s{X}_{G^{\vee},S'}$, hence lives in the respective upper cluster algebra.  Finally, by \cite[Thm. 2.14]{GS:QGM}, other than the case of $G=\SL_2$ on a once-punctured closed surface (which follows from the results of \cite{ManQin}), these moduli of local systems always satisfy the full Fock-Goncharov conjecture, so being in the upper cluster algebra implies that $f'$ is indeed a finite linear combination of theta functions, as desired.
\end{proof}

\subsubsection{Extension theorem example: loop elements}\label{sub:loop}
We briefly assume the reader has some familiarity with cluster algebras. 
 Consider the cluster algebra associated to the Kronecker quiver $\begin{tikzpicture}[baseline=(a.base)] \node[mutable] (a) at (0,0) {$\scriptstyle x_1$}; \node[mutable] (b) at (1,0) {$\scriptstyle x_2$}; \draw[double,->] (a) to (b);\end{tikzpicture}$.  Here, $x_1$ and $x_2$ indicate the cluster variables associated to the two vertices.  Most theta functions in this case are global monomials (i.e., their restriction to some cluster is a monomial), but it is well-known there is a ray in $\R_{\geq 0}(1,-1)\subset M_{\R}$ whose corresponding theta functions are different; cf. \cite[Ex. 3.8 and 3.11]{CGMMRSW} and \cite[Ex. 3.10]{GHKK}.  In particular, we have $\ell\coloneqq \vartheta_{(1,-1)}= \frac{x_1^2+x_2^2+1}{x_1x_2}$.  More generally, for $k\geq 1$, $\vartheta_{(k,-k)}=T_k(\ell)$ where $T_k$ is the $k$th Chebyshev polynomial of the first kind, given recursively by $T_0(\ell)=2$, $T_1(\ell)=\ell$, and $T_{k+1}(\ell)=\ell T_k(\ell) -T_{k-1}(\ell)$; cf. \cite[Ex. 5.8]{ManQin}.

    Now consider a more general quiver $Q$---without loops or oriented $2$-cycles---which contains the Kronecker quiver $K$ from above as a full subquiver.  If all vertices of $Q\setminus K$ are frozen (but not $K$ itself), then $\vartheta_{(k,-k,0,0,\ldots,0)}$ is given by $T_k(\ell)$ as above, except that $\ell$ now has some extra frozen variables as coefficients.  Specifically, \begin{align}\label{eq:ell}
    \ell=x_1x_2^{-1}(1+x^{p^*(e_2)}+x^{p^*(e_1)+p^*(e_2)}),
    \end{align}
    where $p^*(e_i)=\sum_j b_{ij}e_j^*$ for $b_{ij}$ the number of arrows from $i$ to $j$ (negative if the arrows actually go from $j$ to $i$).  Here, $x^{e_j^*}$ indicates the cluster variable $x_j$ associated to vertex $j$.
    
    Now suppose some vertex $j$ of $Q\setminus K$ is actually not frozen.  Then $\ell$ (hence also each $T_k(\ell)$) remains a regular function after mutation by $j$ as long as each exponent of $\ell$ has non-negative $e_j^*$-coefficients.  By \eqref{eq:ell}, this is equivalent to having $b_{2j}\geq 0$ and $b_{1j}+b_{2j}\geq 0$.

    Assuming that the seed associated to $Q$ is ``totally coprime'' (which holds whenever $p^*$ is injective, e.g., whenever we work with principal coefficients), containment of $\ell$ in the upper cluster algebra is implied by containment in the initial and adjacent clusters; cf. \cite[Cor. 1.7]{BFZ} or \cite[Thm. 3.9(2)]{GHK3}.  Theorem \ref{thm:ext-intro} applies to give the following.

    \begin{prop}
        As above, consider a quiver $Q$ containing the Kronecker quiver $K$ as a full subquiver.  Suppose that for each non-frozen vertex $j\in Q\setminus K$, we have that the number of arrows from $2$ to $j$ is non-negative and is greater than or equal to the number of arrows from $j$ to $1$.  Assume furthermore that the associated cluster algebra $\s{A}(Q)$ (or $\s{A}^{\prin}(Q)$ if total coprimality fails) satisfies the full Fock-Goncharov conjecture.  Then the elements $T_k(\ell)$, $k\geq 1$, are theta functions in $\s{A}(Q)$.
    \end{prop}
    
    We refer to these elements $T_k(\ell)$ as loop elements since they include the bracelets basis elements associated to (annular) weighted loops in cluster algebras from surfaces.

\subsection{Representation-theoretic applications}\label{subsub:rep}

\cite[\S 9]{GHKK} discusses applications of their theta basis constructions to representation theory. The idea is as follows.
Many representation-theoretic concepts can be expressed in geometric terms.
For instance, the irreducible representations of a reductive group $G$ can be viewed as the spaces of global sections of line bundles over the flag variety $G/B$ for $B \subset G$ a Borel subgroup, 
and the structure constants for decomposing tensor products of these irreducible representations as direct sums can be computed as the dimensions of spaces of sections of line bundles $\s{L}$ over $G \backslash\lrp{G/B}^{\times 3} $.
Both of these varieties (as well as many others of interest) may be realized as minimal models $(Y,D)$ for cluster varieties $U$.
Moreover, the full Fock-Goncharov conjecture is known to hold for many such cluster varieties.
So, the integral tropical points of the mirror $U^{\vee}$ index a canonical basis $ {\bf{B}}_{U} $ of theta functions for $\s{O}(U)$, and one may hope that some subset $ {\bf{B}}_{\s{L}\to Y} $ of $ {\bf{B}}_{U} $ will be a canonical basis for the space of sections of the line bundle $\s{L}\to Y$ under consideration.  Indeed, the approach discussed \S \ref{sub:line} shows how to identify these canonical bases.
The particular example mentioned above of structure constants for decomposing tensor products of irreducible representations is studied in \cite{GonSh}, \cite{MageeGHK}, and \cite{Fei-mult}.

Valuative independence and theta reciprocity were conjectured by Gross-Hacking-Keel-Kontsevich \cite[Conjecture~9.8 and Remark 9.11, respectively]{GHKK}.  In order to carry out their representation-theoretic applications, \cite{GHKK} proved that these conjectures hold for a given valuation $\val_n$ if one assumes the existence of an \textit{optimized seed}; that is, if $\vartheta_n$ is a global monomial (i.e., its restriction to some cluster is a Laurent monomial).  See \cite[Proposition~9.7]{GHKK} and \cite[Lem.~9.10(3)]{GHKK}, respectively.  Our results establish their conjectures in general, thus allowing one to apply these ideas without verifying the optimized seed condition.

\subsection{Some further details and discussion}

\subsubsection{Mirror dual seeds}\label{subsub:mirror}  Here we briefly explain what ``mirror dual'' means in Theorem \ref{thm:main}, expanding on Footnote \ref{foot:mirror}.

In the construction of a cluster algebra, one starts with a seed, or more generally (as we define in \S \ref{sec:cluster}), a seed datum.  This includes the data of a pair of dual lattices $M$ and $N$, along with elements $u_1,\ldots,u_s\in M$ and $v_1,\ldots,v_s\in N$ such that the matrix $B=(u_j\cdot v_i)_{ij}$ is skew-symmetrizable, i.e., $B=DB^{\bullet}$ for some skew-symmetric matrix $B^{\bullet}$ and a diagonal matrix $D = \diag(d_1,\dots,d_r)$, $d_i\in \Q_{>0}$.  From this one can construct a consistent scattering diagram $\f{D}$ in $M_{\R}$ whose incoming walls are
$$\{(v_i^{\perp},1+x^{u_i}y^{e_i},v_i):i=1,\ldots,s\}$$
for $y^{e_i}$ a coefficient-variable---a scattering wall for us is a triple $(\f{d},f,n)$ where $n\in N$, $\f{d}\subset n^{\perp}$, and $f\in \Z[x^u]\llb y^{\sQ}\rrb$ with $u\in n^{\perp}$.  The associated wall-crossing automorphism is $E_{n,f}:x^my^q\mapsto x^my^qf^{m\cdot n}$.

By switching the roles of $N$ and $M$, one constructs the ``chiral dual'' scattering diagram $\f{D}^{\vee}$ in $N_{\R}$ whose incoming walls are
\begin{align}\label{eq:dual-scattering}
    \{(u_i^{\perp},1+x^{d_i^{-1}v_i}y^{e_i},d_iu_i):i=1,\ldots,s\}.
\end{align}
Theta reciprocity is proved for chiral dual pairs; cf. Claim \ref{conj: theta1} and Theorem \ref{thm:taut-pairing}.

If $D$ and $B^{\bullet}$ are both integral, Proposition \ref{lem:D-move} shows (using our coefficient specialization results from \S \ref{sub:lin-indep}) that we can leave out the $d_i^{-1}$ and $d_i$ factors in \eqref{eq:dual-scattering} and still have theta reciprocity.  This corresponds to taking the chiral-Langlands dual seed data rather than just the chiral dual seed data as defined in \S \ref{sub:lin-morph}.  Then Corollary \ref{cor:Langlands} states theta reciprocity for Langlands dual seed data, which has initial walls $\{(u_i^{\perp},1+x^{-v_i}y^{e_i},u_i):i=1,\ldots,s\}$.  We note that theta reciprocity in the Langlands dual setup requires us to expand our mirror theta functions in the negative chamber of the scattering diagram instead of the positive chamber.

\subsubsection{Valuative independence for more general mirrors}\label{subsub:general-VIT}
    Our proof of VIT applies to theta functions constructed from \textit{any} positive scattering diagram (i.e., a scattering diagram for which the scattering functions all have positive coefficients).  This positivity property is known to hold in the cluster setting \cite[Thm. 1.13]{GHKK}.  Additionally, for any smooth affine log CY, the associated scattering diagram of Keel and Yu \cite{KY,KY2} is known to be positive, so our arguments point towards\footnote{It remains to check that the valuations of loc. cit. agree with ours.  Additionally, we assume a convexity condition on scattering diagrams (cf. Footnote \ref{foot:N+}), and it may be necessary to develop an analog of this condition in the Keel-Yu setting.} a proof of \cite[Conjecture 13.6]{KY2}.  In cases containing a Zariski-dense torus, Johnston \cite{Johnston-comparison} shows that the Keel-Yu theta functions agree with those of Gross-Siebert \cite{GSInt2}.
   
    In fact, the need for positivity is subtle and can perhaps be overcome; cf. Theorem \ref{thm:min-taut} and Remark \ref{rem-not-positive}.  In particular, the condition $\vartheta_{ku,p}^{\trop}=k\vartheta_{u,p}^{\trop}$ for $k\in \Z_{\geq 0}$ is sufficient (this is Lemma \ref{lem:val-ku} under the positivity assumption).  Additionally, Theorem \ref{thm:nu-intro} above says that VIT applies to non-positive scattering diagrams which come from specializing the coefficients of positive ones.  

    Several more general or modified definitions of scattering diagrams are outlined in \S \ref{sub:gen}.  We believe our valuative independence arguments apply to all of these setups (assuming positivity).

\subsection{Additional related work}

Jiarui Fei has studied a similar dual pairing using generic bases. \cite[Thm. 1.15]{Fei-tropF} proves a generic basis version of theta reciprocity for acyclic cluster algebras and for cases where one of the basis elements is a cluster variable.  The general symmetry conjecture for generic bases is \cite[Conj. 6.8]{Fei-tropF}.  See \S \ref{sub:NP} for connections to other conjectures of Fei \cite{Fei-CombF} (proved by \cite{LP}) regarding the coefficients of extremal monomials in the Laurent expansions of cluster variables.

Tropical theta functions for log Calabi-Yau surfaces were previously studied by the third author. 
 In particular, \cite[Prop. 4.15]{Man1} is essentially the two-dimensional case of our Theorem \ref{thm:trop-determines-u} (from which VIT follows), while \cite[Thm. 4.13]{Man1} is essentially the two-dimensional case of theta reciprocity (reinterpreted as in our Claim \ref{conj: theta2}).

It is shown in \cite{ManAtomic} that the theta function basis is atomic, meaning that it consists precisely of the extremal universally positive elements of the cluster algebra; cf. \S \ref{sub:BL-and-Theta}.  In \S \ref{sub:atomic} we observe that VIT implies that the tropical theta functions are, in a certain sense, atomic in the semiring of tropicalized functions.

These tropicalized functions $f^{\trop}:M_{\R}\rar \R\cup \{\pm \infty\}$ can be used to construct polytopes $\Xi_{f,r}\coloneqq \{m\in M_{\R}|f^{\trop}(m)\geq r\}$ for $r\in \R$.  The BL-convexity of $f^{\trop}$ implies that $\Xi_{f,r}$ is broken line convex \cite[Prop. 41]{F-MM}, meaning that any broken line segment with endpoints in $\Xi_{f,r}$ is entirely contained in $\Xi_{f,r}$.  The geometry of such polytopes is explored in \cite{Man1} for the case of positive log Calabi-Yau surfaces, and for more general cluster varieties in \cite{F-MM} using our valuative independence and theta reciprocity results.

By \cite[Thm. 1.2]{CMN}, $\Xi_{f,r}$ being broken line convex implies that it is also positive in the sense of \cite[Def. 8.6]{GHKK}.  It follows that one can construct a graded algebra
\begin{align*}
    V=\bigoplus_{k\in \Z_{\geq 0}} \bigoplus_{m\in M\cap \Xi_{f,-k}} \Z[y^{\sQ}]\cdot \vartheta_m T^k \subset \Z[x^M]\llb y^{{\sQ}^{\gp}} \rrb[T]
\end{align*}
with grading given by the powers of $T$.  Assume that for each $k\in \Z_{\geq 0}$ and each $m\in \Xi_{f,-k}$, $\vartheta_m$ is a finite Laurent polynomial in each cluster---such theta functions form a $\Z[y^{\sQ^{\gp}}]$-basis for the middle cluster algebra $A^{\midd}$. Then $\s{V}\coloneqq \Proj V$ defines a partial compactification of the cluster variety $\s{A}^{\midd}\coloneqq \Spec A^{\midd}$.

To show that $\s{V}$ is a projective variety over $\Spec \Z[y^{\sQ}]$, one should show that $V$ is finitely generated over $\Z[y^{\sQ}]$.  This is done in \cite[Thm. 8.19]{GHKK} under the assumption of ``enough global monomials.''  Alternatively, when $\s{V}$ is the mirror to a log Calabi-Yau surface, an argument for this finite generation is given in \cite[Prop. 4.19]{LaiZhou}.  The argument of loc.~cit.~ can be  generalized to mirrors of higher-dimensional fibers $Y$ of cluster Poisson manifolds:~one uses \cite[\S 2.2]{ManFrob} to identify the monoid $\sQ$ with a cone of curve classes in $A_1(Y,\Q)$, then uses the techniques of \cite[\S 5]{AB} (generalizing techniques of \cite{GHK1}) to show that coefficients appearing in the structure constants correspond to effective curve classes.  With this setup, the proof of \cite[Prop. 4.19]{LaiZhou} applies with no significant changes.

In \S \ref{sub:Looijenga-Pair}, we focus on the log Calabi-Yau surfaces $(Y,D)$ considered in \cite{GHK1,Man1,LaiZhou}.  We give an explicit description of tropical theta functions and note that this description agrees with the one given in \cite[\S 4]{Man1} (thus confirming that our valuations are the same as those considered in loc.~cit.).  

Applications of valuative independence and theta reciprocity to mirror symmetry for log Calabi-Yau varieties are discussed in the recent work of Keel-Yu \cite[\S 13]{KY2}.  In particular, Conjectures 13.5 and 13.6 of loc.~cit.~ predict theta reciprocity and valuative independence for all affine log Calabi-Yau varieties with maximal boundary.  Other recent work of Keel-White \cite{KW} applies these ideas, together with our results, to achieve canonical theta function bases for Cox rings of positive log Calabi-Yau surfaces as discussed in \S \ref{subsub:Cox}. 

Since the initial posting of this article, Yang Li has used valuatively independent bases to prove results on degenerations of Calabi-Yau metrics \cite{Li-Degen} and the metric SYZ conjecture \cite{Li-SYZ}.  The latter uses recent work of Blum-Liu \cite{Blum-Liu} which proves the existence of (non-canonical) valuatively independent bases of functions on log Calabi-Yau varieties very generally.

\subsection{Outline of the paper}

In \S \ref{sec:Background} we review the constructions and basic properties of scattering diagrams and theta functions.  We then introduce the valuations $\val_v$ and prove some basic properties.

In \S \ref{sec:Taut-VIT} we introduce the notion of a rational broken line (allowing rational bends).  A valuation $\val_v$ is expressed by an infimum over rational broken lines, and in Theorem \ref{thm:min-taut} we show that rational broken lines cannot be locally minimizing unless they are $v$-taut as defined in \S \ref{sub:taut}.  This leads to Theorem \ref{thm:trop-determines-u}, which says that if $v$ is generic, then $\val_v$ will take different values for each theta function.  This implies valuative independence (Theorem \ref{thm:indep}) for generic $v$, and then continuity lets us extend to non-generic $v$. We next prove valuative independence---hence linear independence---for specialized coefficients (Theorem \ref{thm:nu}), followed by our first version of the theta function extension theorem (Theorem \ref{thm:theta-extension-general}).  Then in \S \ref{sub:NP} we discuss an application to Newton polytopes, finding that the coefficients of extremal monomials of theta functions must be $1$.  Atomicity of tropical theta functions is discussed in \S \ref{sub:atomic}, and upper semicontinuity of valuations $\val_v$ is shown in \S \ref{sub:up}.

In \S \ref{sec:cluster} we focus on theta functions for cluster varieties.  We introduce the notion of a ``seed datum'' which generalizes the usual concept of a seed, allowing us to consider a wider class of cluster-type varieties that includes all cluster varieties of $\s{A}$-type and $\s{X}$-type along with quotients of $\s{A}$-types and subfamilies of $\s{X}$-types.  Theorem \ref{thm: linearmorphism} describes the behavior of theta functions under morphisms of seed data, called linear morphisms, generalizing the cluster ensemble map.  Theorem \ref{thm:theta-ext-seed} applies the theta function extension theorem to the seed datum setting.  Claim \ref{conj: theta1} states our first version of theta reciprocity for chiral dual seed data.  Proposition \ref{lem:D-move} and Corollary \ref{cor:Langlands} give corresponding statements for chiral-Langlands dual and Langlands dual seed data.  

In \S \ref{sec:Lambda} we introduce the idea of a $\Lambda$-structure, motivated by the compatible pairs of \cite{BZ}.  These $\Lambda$-structures let us relate a seed datum to its chiral dual, allowing us to state a version of theta reciprocity between a seed datum and itself (Claim \ref{conj: theta2}).  We then spend \S \ref{sub:Ltaut}-\S \ref{sub:proof} proving theta reciprocity by first expressing our tropical theta functions in terms of infima over certain collections of ``$\Lambda$-taut'' broken lines.  These constructions are applied to log Calabi-Yau surfaces in \S \ref{sub:Looijenga-Pair}.

\subsection*{Acknowledgments}

We would like to thank Sean Keel for many years of helpful conversations about tropical theta functions and their applications, and for feedback about early drafts of this paper.  We also thank Luca Francone for pointing out subtle issues with how some claims in \S \ref{sub:line} were originally stated.  T. Magee gratefully acknowledges partial support of the Engineering and Physical Sciences Research Council through the grant EP/V002546/1.  T. Mandel was partially supported by a grant from the Research Council of the University of Oklahoma Norman Campus.

\section{Scattering diagrams, theta functions, and valuations}\label{sec:Background}

Roughly speaking, a scattering diagram is a real vector space endowed with a collection of \emph{walls}, which are hyperplanar regions decorated with \emph{elementary transformations}: automorphisms of a formal Laurent series ring which act by $x^m\mapsto x^mf^{m\cdot n}$ for a polynomial or series $f$ and a vector $n$ normal to the wall. A \emph{consistent} scattering diagram is one in which the composition of the elementary transformations along any path is invariant under homotopies of the path fixing its endpoints. By enumerating certain piecewise linear paths (\emph{broken lines}) in a consistent scattering diagram, one may define \emph{theta functions}.  In this section, we review these constructions before defining a certain class of valuations on the ring of theta functions.

\subsection{Scattering diagrams}\label{sub:scattering}

We fix a commutative monoid $\sQ$ and a finite-rank lattice $N$ with dual lattice $M$.  Denote the units in $\sQ$ by $\sQ^{\times}$.   We write $\sQ^{\gp}$ for the Grothendieck group of $\sQ$, and $\check{\sQ}^{\gp}$ for $\Hom(\sQ^{\gp},\Z)$. Denote $N_{\R} \coloneqq N\otimes \R$, $M_{\R} \coloneqq M\otimes \R$, $\sQ^{\gp}_{\R}=\sQ^{\gp}\otimes \R$, and $\sQ_{\R}=\sQ\otimes \R_{\geq 0}$, and similarly with $\Q$ in place of $\R$.  We use $\cdot$ to indicate the dual pairing between $M$ and $N$ (similarly for $M_\R$ and $N_\R$).  Let $A\coloneqq \Z[x^M][y^{\sQ}]$, and let $\s{I}$ be the ideal $A\langle y^q|q\in \sQ\setminus \sQ^{\times}\rangle$.  Define $A_k\coloneqq A/\s{I}^k$ and $\wh{A}\coloneqq \varprojlim A_k = \Z[x^M]\llb y^{\sQ}\rrb$, and let $\wh{\s{I}}=\varprojlim A_k \s{I}$ be the closure of $\s{I}$ in $\wh{A}$.  We also denote $\wh{A}^+ \coloneqq \N[x^M]\llb y^{\sQ}\rrb\subset \wh{A}$, and for nonzero $v\in M$, $$
A_v^{\parallel}\coloneqq \Z[x^v]\llb y^{\sQ}\rrb\cap (1+\wh{\s{I}})\subset \wh{A}.
$$  For convenience, we will sometimes write $x^m y^q$ as simply $z^{u}$ for $u=(m,q)\in M\oplus \sQ$.  Additionally, given $u\in M\oplus \sQ$, we may write the components of $u$ as $u_M$ and $u_{\sQ}$.

Given $n\in N$, nonzero $v\in n^{\perp}\subset M$, and $f\in A_v^{\parallel}$, we define the (formal) \textbf{elementary transformation} $E_{n,f}:\wh{A}\rar \wh{A}$ to be the automorphism given by
\begin{align*}
    E_{n,f}(x^my^q)=x^my^qf^{m\cdot n}
\end{align*}
extended linearly (and formally).  Note that $E_{n,f}^{-1}=E_{-n,f}=E_{n,f^{-1}}$.

A \textbf{wall} is a triple $(W,f,n)$ consisting of 
\begin{itemize}
    \item a nonzero element $n\in N$,
    \item a convex rational-polyhedral cone $W\subset n^{\perp}$ with codimension one in $M_{\R}$,
    \item an element $f\in A_v^{\parallel}$ for some nonzero $v\in n^{\perp}\cap M$.
\end{itemize}
For $v$ as above, the vector $-v$ is called the \textbf{direction} of the wall, and $f$ is called the \textbf{scattering function}.  The cone $W$ is called the \textbf{support} of the wall.  The wall is called \textbf{incoming} if its support contains $v$.  Otherwise it is called \textbf{outgoing}.  We say a wall is \textbf{positive} if the scattering function $f$ lies in $\wh{A}^+$, i.e., if the coefficients are always non-negative.

A \textbf{scattering diagram} $\f{D}$ is a set of walls such that, for each $k\in \N$, all but finitely many walls of $\f{D}$ are trivial in the projection to $A_k$ (i.e., the scattering functions are equivalent to $1$ modulo $\s{I}^k$).  We write $\f{D}^k$ for this finite set of walls in $\f{D}$ which are non-trivial in $A_k$.  We say $\f{D}$ is \textbf{positive} if every wall of $\f{D}$ is positive.

We shall require our scattering diagrams to satisfy one more technical condition:  let $\s{M}^{\circ}\subset M\oplus \sQ$ be the set of pairs $(m,q)$ which appear as exponents in the scattering functions for $\f{D}$.  Let $\s{M}^+$ be the closure of the $\R_{\geq 0}$-span of $\s{M}^{\circ}$.  Then we require that $\s{M}^+$ is strictly convex (i.e., contains no line through the origin).  Let $\s{N}^+=\Hom_{\text{monoid}}(\s{M}^+,\R_{\geq 0})\subset N_{\R}\oplus \check{\sQ}^{\gp}_{\R}$.  The strict convexity of $\s{M}^+$ is equivalent to $\s{N}^+$ having nonempty interior as a subset of $N_{\R}\oplus \check{\sQ}^{\gp}_{\R}$.\footnote{Explicitly imposing this condition is non-standard, but we do not expect this to impact any cases of interest.  E.g., the condition holds for all scattering diagrams constructed from a finite set of initial walls as in Theorem \ref{thm:ScatD}.  The point of the condition is to ensure that tropicalized functions are finite on some top-dimensional cone in $N_{\R}\oplus \check{\sQ}^{\gp}_{\R}$, e.g., what we call $\s{N}^{+}$.  This is used when proving Theorem \ref{thm:indep} for non-generic $v$.\label{foot:N+}}

The \textbf{support} of $\f{D}$, denoted $\Supp(\f{D})$, is the union of the supports of all of its non-trivial walls.  The \textbf{joints} of $\f{D}$, denoted $\Joints(\f{D})$, are the codimension-two intersections and boundaries of the walls of $\f{D}$.  A connected component of $M_{\R}\setminus \Supp(\f{D})$ is called a \textbf{chamber} of $\f{D}$.  

A continuous path $\gamma:I\rar M_{\R}$ is \textbf{generic} (with respect to $\f{D}$) if
\begin{itemize}
    \item the endpoints do not lie in $\Supp(\f{D})$,
    \item the path does not intersect $\Joints(\f{D})$, and
    \item if $\gamma(t)$ is in the support $W$ of a wall, then $\gamma(t-\epsilon)$ and $\gamma(t+\epsilon)$ lie on opposite sides of $W$ for all sufficiently small $\epsilon>0$. 
\end{itemize}
Given a generic path $\gamma$ crossing a set of walls $(W_i,f_i,n_i)$ at time $t$, define 
\begin{align}\label{eq:Egamma}
    E_{\gamma,t}\coloneqq \prod_i E_{n_i,f_i}^{s_i} \quad \text{where} \quad s_i\coloneqq \sign(\gamma(t-\epsilon)\cdot n_i)\in \{\pm 1\}.
\end{align}  
Genericness of $\gamma$ implies that these walls have parallel support, hence parallel $n_i$'s.  It follows that these automorphisms commute, so the order of the product (given by composition) does not matter.  One then defines the \textbf{path-ordered product} $E_{\gamma}:\wh{A}\rar \wh{A}$ as $\prod_t E_{\gamma,t}$ where the product is ordered right-to-left by increasing times $t\in I$.  We may write $E_{\gamma}^{\f{D}}$ if $\f{D}$ is not clear from context.

Two scattering diagrams $\f{D}_1,\f{D}_2$ in $M_\R$ are called \textbf{equivalent} if $E_{\gamma}^{\f{D}_1}=E_{\gamma}^{\f{D}_2}$ for all paths $\gamma$ which are generic with respect to both $\f{D}_1$ and $\f{D}_2$.  A scattering diagram is called \textbf{consistent} if the path-ordered products depend only on the endpoints of the paths.

The following is a fundamental feature of scattering diagrams which has been proved in various related contexts by \cite{GS11,WCS,GHKK}.

\begin{thm}\label{thm:ScatD}
    Let $\f{D}_{\In}$ be a scattering diagram such that the support of each wall of $\f{D}_{\In}$ is an entire hyperplane\footnote{\cite[Thm. 5.6]{AG} gives a version of Theorem \ref{thm:ScatD} which allows the initial walls to take more general forms---``widgets'' rather than just hyperplanes.}  in $M_{\bb{R}}$.  Then there exists a consistent scattering diagram $\f{D}=\Scat(\f{D}_{\In})$ containing $\f{D}_{\In}$ such that every wall of $\f{D}\setminus \f{D}_{\In}$ is outgoing.  Furthermore, $\f{D}$ is unique up to equivalence.
\end{thm}

We may refer to $\Scat(\f{D}_{\In})$ as the \textbf{consistent completion} of $\f{D}_{\In}$.

\subsection{Broken lines and theta functions}\label{sub:BL-and-Theta}

Let $\f{D}$ be a consistent scattering diagram in $M_{\bb{R}}$ (we shall implicitly assume consistency whenever we discuss broken lines or theta functions).  A \textbf{broken line} $\Gamma$ for $\f{D}$ with initial exponent $u=(u_M,u_\sQ)\in M\oplus \sQ$ and generic endpoint $p\in M_{\bb{R}}$, sometimes written $\Ends(\Gamma)=(u,p)$, is a continuous piecewise affine-linear map $\Gamma:\R_{\leq 0}\rar M_{\R}\setminus \Joints(\f{D})$, together with a monomial $cx^my^q\in \wh{A}$ attached to each straight segment, such that:
\begin{itemize}
    \item $\Gamma(0)=p$;
    \item the monomial on the initial (unbounded) straight segment is $x^{u_M}y^{u_\sQ}$.  In particular, an initial straight segment exists, i.e., $\Gamma$ has finitely many bends;
    \item if $\Gamma$ is straight for $t\in (a,b)$ with attached monomial $cx^my^q$, then $\Gamma'(t)=-m$ for all $t\in (a,b)$;
    \item if the attached monomial changes from $cx^my^q$ to $c'x^{m'}y^{q'}$ when crossing $t=t_0$, then the latter is a term in $E_{\Gamma,t}(cx^my^q)$.  By construction,  the exponents $\sign(\gamma(t-\epsilon)\cdot n) (m\cdot n)$ appearing in the application of the elementary transformations here will always be positive.
\end{itemize}
Note that if $\f{D}$ is positive, then the attached monomials of broken lines will always be positive as well.

Given a broken line $\Gamma$, the attached monomial of the final segment is denoted $c_{\Gamma}x^{m_{\Gamma}}y^{q_{\Gamma}}$ or $c_{\Gamma} z^{u_{\Gamma}}$.  For any $u=(u_M,u_\sQ)\in M\oplus \sQ$ and generic $p\in M_{\bb{R}}$, we define a \textbf{theta function}
\begin{align*}
    \vartheta_{u,p}=\sum_{\Gamma} c_{\Gamma}x^{m_\Gamma}y^{q_{\Gamma}}\in \wh{A}
\end{align*}
where the sum is over all broken lines $\Gamma$ with initial exponent $u$ and endpoint $p$. We may write $\vartheta_{u,p}^{\f{D}}$ if the scattering diagram $\f{D}$ is not clear from context.

The following consequence of the consistency of $\f{D}$ is a fundamental property of theta functions.

\begin{lem}[\cite{CPS}]\label{lem:CPS}
    Assume $\f{D}$ is consistent. 
 Let $\gamma$ be a generic path from $p_1$ to $p_2$ for $p_1,p_2\in M_{\R}$ generic.  Then for any $u\in M\oplus \sQ$, we have $\vartheta_{u,p_2}=E_{\gamma}(\vartheta_{u,p_1})$.
\end{lem}

It is straightforward to show that $\vartheta_{(u_M,u_\sQ),p}=y^{u_\sQ}\vartheta_{(u_M,0),p}$ and that $\vartheta_{u,p}\in x^{u_M}y^{u_\sQ}(1+\wh{\s{I}})$.  As a consequence of the latter property, one can show that, for any fixed generic $p\in M_{\R}$, the theta functions $\vartheta_{u,p}$ for $u\in M\oplus \sQ$ form a topological\footnote{Topological here essentially means that we allow linear combinations to be infinite.  I.e., $\wh{A}$ is not technically the span of the theta functions, but rather the closure of the span in the $\s{I}$-adic topology.   See \cite[\S 2.2.2]{DM} for details.} $\Z$-module basis for $\wh{A}$.

For each generic $p\in M_{\R}$, we associate a copy $\wh{A}_p$ of $\wh{A}$.  By Lemma \ref{lem:CPS}, path-ordered products canonically identify the algebras $\wh{A}_p$ while preserving the theta functions.  We therefore view these algebras $\wh{A}_p$ as copies of a single algebra\footnote{The hat over the direct sum here indicates that we must take the $(\sQ\setminus \sQ^{\times})$-adic completion.} $$\wh{A}^{\can}=\wh{\bigoplus}_{u\in M\oplus \sQ} \Z \cdot \vartheta_{u}$$ just presented in different coordinate systems via the identifications $\iota_p:\wh{A}^{\can}\risom \wh{A}_p$, $\vartheta_{u}\mapsto \vartheta_{u,p}$.

Any $f\in \wh{A}^{\can}$ may be written as a series in the theta functions
\[f=\sum_{u\in M\oplus \sQ} \alpha_{u}(f)\vartheta_{u}
\]
for some coefficients $\alpha_u(f)\in \Z$.  Similarly, for any generic basepoint $p\in M_{\bb{R}}$, $\iota_p(f)$ may be written as a Laurent series
\[ \iota_p(f)=\sum_{u\in M\oplus \sQ} \beta_{u,p}(f)x^{u_M}y^{u_\sQ}.
\]
The following says that the coefficients indexed by $u$ agree when the basepoint $p$ is close to $u_M$.
\begin{lem}\label{lem:coeff}
Let $u$, $\alpha_{u}$, and $\beta_{u,p}$ be as above. Then, for any $f\in \wh{A}^{\can}$ and for all generic $p$ sufficiently close to $u_M$, we have
$ \alpha_{u}(f) = \beta_{u,p}(f)$.
\end{lem}

The proof is essentially contained in the proof of \cite[Prop. 6.4(3)]{GHKK}; also cf. \cite[Lem. 5.11]{ManQin} (the present setting is more general but the proof is the same).

The \textbf{structure constants} of the theta function basis are the constants $\alpha(u_1,\ldots,u_s;u)\in \Z$ defined by
\begin{align*}
        \vartheta_{u_1}\cdots \vartheta_{u_s}=\sum_{u\in M\oplus \sQ} \alpha(u_1,\ldots,u_s;u) \vartheta_{u}.
\end{align*}
As in \cite[Prop. 6.4(3)]{GHKK} (also cf. \cite[Prop. 5.9]{ManQin}), Lemma \ref{lem:coeff} implies the following:
\begin{cor}\label{cor:str-const}
    For any $u_1,\ldots,u_s,u\in M\oplus \sQ$, the corresponding structure constant is given by
    \begin{align*}
        \alpha(u_1,\ldots,u_s;u) = \sum_{\substack{\Gamma_1,\ldots,\Gamma_s \\ \Ends(\Gamma_i)=(u_i,p) \\ \sum_{i=1}^s u_{\Gamma_i}=u}} c_{\Gamma_1} \cdots c_{\Gamma_s} \in \Z
    \end{align*}
    for any $p$ sufficiently close to $u$.
\end{cor}

The following immediate consequence of Corollary \ref{cor:str-const} is also well-known.

\begin{cor}[Strong positivity]\label{cor:strong-pos}
    If $\f{D}$ is positive, then all of the structure constants are non-negative integers.
\end{cor}

\begin{exam}\label{eg:str}
    Using Corollary \ref{cor:str-const}, one finds the following:
    \begin{enumerate}
        \item $\alpha(u_1,\ldots,u_s;u)$ equals $0$ whenever $u_{\sQ}\notin (u_1)_{\sQ}+\ldots+(u_s)_{\sQ}+\sQ$.
        \item If $u_{\sQ}=(u_1)_{\sQ}+\ldots+(u_s)_{\sQ}$, then $\alpha(u_1,\ldots,u_s;u)$ equals $1$ if $u_M=(u_1)_M+\ldots+(u_s)_M$, and it equals $0$ otherwise.
    \end{enumerate}
\end{exam}

We say that $f\in \wh{A}^{\can}$ is \textbf{universally positive} if $\iota_p(f)\in \wh{A}_p^+$ for each generic $p\in M_{\R}$.  Note that if $\f{D}$ is positive, then $\vartheta_{u}$ is universally positive for every $u\in M\oplus \sQ$.  Furthermore, using Lemma \ref{lem:coeff}, one can show that $f\in \wh{A}^{\can}$ is universally positive if and only if $f$ is a positive linear combination of theta functions---i.e., the basis of theta functions is \textbf{atomic} (cf. \cite{ManAtomic}).  We write $\wh{A}^{\can,+}$ for the semiring of universally positive elements, or equivalently, for the semiring $\wh{\bigoplus}_{u\in M\oplus \sQ} \N\cdot \vartheta_{u}$ of positive (formal) linear combinations of theta functions.

\subsection{Valuations}\label{sub:val}

Let $\mathbb{Z}^{\min}\coloneqq \mathbb{Z}\cup\{\pm \infty\}$.  If we define $(-\infty)+\infty=\infty$, then $\min$ and $+$ make $\Z^{\min}$ into a complete semiring with $\min$-identity $\infty$ and $+$-identity $0$.  Similarly for $\R^{\min}=\R\cup \{\pm \infty\}$.

For any $v=(w,s)\in N_\mathbb{R}\oplus \check{\sQ}^{\gp}_\R$, we associate a \textbf{valuation}
$\mathrm{val}_{v}:\wh{A} \rightarrow \mathbb{R}^{\min}$ defined by
\begin{align}\label{eq:valv-def} \mathrm{val}_{v} \left( \sum_{\substack{m\in M \\ q\in \sQ}} c_{m,q} x^m y^q\right) \coloneqq  \inf_{m,q\mid c_{m,q}\neq 0} [(m,q)\cdot (w,s)].
\end{align}
Here, $\cdot$ denotes the dual pairing between $M_\R \oplus Q^{\gp}_\R$ and $N_\R \oplus \check{Q}^{\gp}_\R$.  In particular, we define $\val_v(0)=\infty$.  If $(w,s)\in N\oplus \check{\sQ}^{\gp}$, then $\mathrm{val}_{(w,s)}$ will take values in $\Z^{\min}$.

\begin{warn}
As a function on formal series, $\mathrm{val}_{(w,s)}$ need not be a valuation in the traditional sense! Specifically, it may not send products of formal series to sums of their values. As an example, assume $m\cdot w =-1$, and consider
\begin{align*}
\val_{(w,0)} (1-x^m)
&= \inf(0,-1) = -1\\
\val_{(w,0)} (1+x^m+x^{2m}+x^{3m}+\cdots)
&= \inf(0,-1,-2,...) = -\infty \\
\val_{(w,0)} \left(
(1-x^m) (1+x^m+x^{2m}+x^{3m}+\cdots)
\right)
&= \val_{(w,0)}(1) = 0
\neq (-1) + (-\infty).
\end{align*}
Note that this can be avoided by working in the semiring $\wh{A}^+$ of positive Laurent series, in which case each $\val_{(w,s)}:\wh{A}^+ \rightarrow \mathbb{R}^{\min}$ is a genuine morphism of semirings.  Similarly, if $\f{D}$ is positive, then the restrictions $\val_{(w,s)}:\wh{A}^{\can,+} \rightarrow \mathbb{R}^{\min}$ are morphisms of semirings.
\end{warn}

For any $v\in N_\R \oplus \check{\sQ}^{\gp}_\R$ and generic $p\in M_{\R}$, we shall write $\val_{v,p}$ for the valuation $\val_{v}$ as above applied to $\wh{A}_p$, or for $\val_{v,p}\circ\iota_p:\wh{A}^{\can}\rar \R^{\min}$.  In particular, we will understand $\val_{v,p}(\vartheta_u)$ to mean $\val_{v}(\vartheta_{u,p})$.

Now suppose that $\f{D}$ is positive, so the final monomial $c_{\Gamma}z^{u_{\Gamma}}$ of every broken line is positive.  Then for each $u\in M\oplus \sQ$ and generic $p\in M_{\R}$, we can write
\begin{align}\label{eq:val-theta}
    \val_{v,p}(\vartheta_{u})=\inf_{\Gamma:\Ends(\Gamma)=(u,p)} (u_{\Gamma}\cdot v).
\end{align}

$$\text{\textbf{Unless otherwise stated, we shall assume from now on that $\f{D}$ is positive, so \eqref{eq:val-theta} holds.}}$$

In the next section, we will consider cases where the $\inf$ in \eqref{eq:val-theta} is actually a $\min$, i.e., is attained for a broken line $\Gamma$ which we call a $\val_{v,p}$-minimizing broken line.  We will show that such minimizing $\Gamma$ must satisfy a ``tautness'' condition.

\begin{exam}\label{ex:Nplus}
    Recall that $\s{M}^+\subset M_{\R}\oplus \sQ_{\R}$ denotes the $\R_{\geq 0}$-span of the exponents appearing in $\f{D}$, and $\s{N}^+$ denotes the dual monoid.  For any $u\in M\oplus \sQ$, generic $p\in M_{\R}$, and $v\in \s{N}^+$, we have $\val_v(\vartheta_{u,p})=u\cdot v$.  This is because when broken lines bend, their attached exponents can only change by adding an element of $\s{M}^+$, and this cannot decrease the dual pairing with $\s{N}^+$, so the minimal dual pairing is attained by the straight broken line.
\end{exam}

We will need the following lemma:

\begin{lem}\label{lem:val-ku}
    Assume $\f{D}$ is positive. 
 Let $v\in N_\R \oplus \check{\sQ}^{\gp}_\R$, $u\in M\oplus \sQ$, $p\in M_{\R}$ generic, and $k\in \Z_{\geq 0}$.  Then $\val_{v,p}(\vartheta_{ku})=k\val_{v,p}(\vartheta_{u})$.
\end{lem}
\begin{proof}
By Example \ref{eg:str}, $\vartheta_{u,p}^k=\vartheta_{ku,p}+\text{[higher $\sQ$-degree terms]}$.  By strong and universal positivity, all terms will have positive coefficients and therefore cannot cancel, so
\begin{align*}
    k\val_{v,p}(\vartheta_{u})  &= \val_{v,p}(\vartheta_{u}^k)\\
                            &=\val_{v,p}(\vartheta_{ku}+\text{[universally positive terms]}) \\
                            &\leq \val_{v,p}(\vartheta_{ku}).
\end{align*}

On the other hand, for each broken line $\Gamma$ contributing to $\vartheta_{u,p}$, there is a corresponding broken line $\Gamma'$ contributing to $\vartheta_{ku,p}$ with the same support but whose attached exponents are $k$ times the corresponding attached exponents of $\Gamma$.  Indeed, this is clear for the initial straight segment, and it follows inductively for later segments: if $c_{i+1}z^{u_{i+1}}$ is a term in $c_i z^{u_i} f^{u_i \cdot n_i}$, then $c'_i z^{ku_i}f^{ku_i\cdot n_i}$ includes a term of the form $c'_{i+1}z^{ku_{i+1}}$.  Thus, the exponents contributing to $\vartheta_{ku,p}$ include $kw$ for each exponent $w$ contributing to $\vartheta_{u,p}$, so $\val_{v,p}(\vartheta_{ku})\leq k\val_{v,p}(\vartheta_u)$.  The claim follows.
\end{proof}

\subsection{Tropical functions}\label{sub:trop-fun}
  Given any $f\in \wh{A}$, define the \textbf{tropicalization} of $f$ to be 
\begin{align*}
    f^{\trop}:N_{\R}\oplus \check{\sQ}^{\gp}_{\R}\rar \R^{\min}, \qquad v\mapsto \val_v(f).
\end{align*}
Note that $f^{\trop}$ is an infimum of linear functions, hence is convex and piecewise-linear on $N_{\R}\oplus \check{\sQ}^{\gp}_{\R}$.  Furthermore, we see that $f^{\trop}:N_{\R}\oplus \check{\sQ}^{\gp}_{\R}\rar \R^{\min}$ is upper semicontinuous, and it is continuous except possibly on the boundary between where it is finite and where it is $-\infty$.

\begin{lem}\label{lem:inf=min}
    Let $v\in N_{\R}\oplus \check{\sQ}^{\gp}_{\R}$ be a rational point where $f^{\trop}$ is finite or a (possibly irrational) point in the interior of the locus where $f^{\trop}$ is finite.  Then the infimum defining $f^{\trop}(v)$ is a minimum---that is, $f^{\trop}(v)=u\cdot v$ for some $u\in M\oplus \sQ$ appearing as an exponent of a term in $f$.
\end{lem}
\begin{proof}
Let $c=f^{\trop}(v)$ and let $F\subset M\oplus \sQ$ be the set of exponents of terms in $f$.  If $kv\in N\oplus \check{\sQ}^{\gp}$ for some $k\in \Z_{>0}$ (i.e., $v$ is rational) then the values $u\cdot v$ for $u\in F$ must lie in $\frac{1}{k}\Z$.  The claim for rational $v$ follows from the discreteness of this set.

Now suppose $v$ is possibly irrational and is in the interior of the set where $f^{\trop}$ is finite. 
 Define $H_{\epsilon}$ to be the set $\{u\in M_{\R}\oplus \sQ_{\R}|c<u\cdot v < c+\epsilon\}$.  If the infimum is not attained, then $H_{\epsilon}\cap F$ is nonempty for all $\epsilon>0$.  Let $B\subset N_{\R}\oplus \check{\sQ}^{\gp}_{\R}$ be a small neighborhood of $v$.  For any $K\in \R_{>0}$, the set $$B_{c-K}^{\vee} \coloneqq \{u\in M_{\R}\oplus \sQ_{\R}|u\cdot w\geq c-K \text{ for all } w\in B\}$$ has bounded intersection with $H_{\epsilon}$.  Since $F$ is discrete, $H_{\epsilon}\cap B_{c-K}^{\vee}\cap F$ must be empty for sufficiently small $\epsilon>0$.  But then the assumption that $H_{\epsilon}\cap F\neq \emptyset$ implies that there must be $u\in H_{\epsilon}\cap F$ that take values smaller than $c-K$ at points of $B$.  Since $B$ and $K$ are arbitrary, it follows that $f^{\trop}=-\infty$ at points arbitrarily close to $v$, so $v$ is not in the interior of the locus where $f^{\trop}$ is finite.
\end{proof}

For $f\in \wh{A}^{\can}$ and generic $p\in M_{\R}$, we define $f^{\trop}_p\coloneqq \iota_p(f)^{\trop}$.  We will be particularly interested in the tropical theta functions $\vartheta_{u,p}^{\trop}$.  Lemma \ref{lem:val-ku} can be reinterpreted as saying that
\begin{align*}
    \vartheta_{ku,p}^{\trop}=k\vartheta_{u,p}^{\trop}.
\end{align*}

Note that Example \ref{ex:Nplus} and our strict convexity assumption on $\s{M}^+$ imply that the locus where $\vartheta_{u,p}^{\trop}$ is finite always has nonempty interior (at least including the interior of $\s{N}^+$) where we can apply Lemma \ref{lem:inf=min}.

\subsection{Limiting representations of functions}\label{sub:lim}

Recall the canonical identifications $\iota_p:\wh{A}^{\can}\risom \wh{A}_p=\wh{A}$ as in \S \ref{sub:BL-and-Theta}.  These were defined for all generic $p\in M_{\R}$.  When we prove theta reciprocity in \S \ref{sec:Lambda}, it will be helpful to have the following construction that works even for non-generic $p$.

Let $p\in M_{\R}$ (not necessarily generic) and fix a generic direction $\mu\in M_{\R}$.  For each generic $\epsilon >0$, the point $p+\epsilon \mu$ will be generic, so we can define $\iota_{p+\epsilon \mu}$.  For any fixed $k>0$, there will be a chamber of $\f{D}^k$ which contains $p+\epsilon \mu$ for all sufficiently small $\epsilon > 0$.  Thus, as we decrease $\epsilon$, the morphism $\iota_{p+\epsilon \mu}$ stabilizes to order $k$ when $\epsilon$ gets sufficiently small---that is, the composition $\iota_{p+\epsilon \mu}:A^{\can} \risom \wh{A}_{p+\epsilon \mu} \rar A_k$ does not depend on the choice of sufficiently small $\epsilon>0$.  We thus obtain a well-defined isomorphism
\begin{align*}
    \lim_{\epsilon \rar 0^+} \iota_{p+\epsilon \mu}: \wh{A}^{\can} \risom \wh{A}.
\end{align*}
We may denote this copy of $\wh{A}$ by $\lim_{\epsilon\rar 0^+} \wh{A}_{p+\epsilon \mu}$.

Recall that for each $f\in \wh{A}^{\can}$, we write $f_p$ to denote $\iota_p(f)$.  We similarly write $$\lim_{\epsilon \rar 0^+} f_{p+\epsilon \mu}\coloneqq \lim_{\epsilon \rar 0^+} \iota_{p+\epsilon \mu} (f).$$  Note that $f_p=\lim_{\epsilon \rar 0^+} f_{p+\epsilon \mu}$ whenever $p$ is generic.  In particular, we can define $\lim_{\epsilon \rar 0^+} \vartheta_{u,p+\epsilon \mu}$ for each $u\in M\oplus \sQ$.\footnote{The exponents appearing in $\lim_{\epsilon \rar 0^+} \vartheta_{u,p+\epsilon \mu}$ can alternatively be understood as coming from a ``non-generic broken line'' in the sense of \cite[Def. 3.3]{CMN}.}

\begin{lem}\label{lem:epsilon}
    Let $\f{D}_p$ be the scattering diagram $\{(W+\R p, f,n) \mid (W,f,n)\in \f{D}, p\in W\}$.  That is, for each wall $(W,f,n)\in \f{D}$ with $p\in W$, there is a corresponding wall $(W+\R p,f,n) \in \f{D}_p$.  Let $\mu,\mu'\in M_{\R}$ be generic, and let $\gamma$ be a path from the chamber of $\f{D}_p$ containing $\mu$ to the chamber containing $\mu'$.  Then for any $f\in \wh{A}^{\can}$,
    \begin{align*}
        \lim_{\epsilon \rar 0^+} f_{p+\epsilon \mu'} = E_{\gamma}^{\f{D}_p}\left(\lim_{\epsilon \rar 0^+} f_{p+\epsilon \mu}\right).
    \end{align*}
\end{lem}
\begin{proof}
    Fix $k>0$, and choose $\epsilon$ small enough so that $p+\epsilon \mu$ and $p+\epsilon \mu'$ each share a (possibly different) chamber of $\f{D}^k$ with $p$.  Then we can choose $\gamma$ from $p+\epsilon \mu$ to $p+\epsilon \mu'$ which crosses no walls of $\f{D}^k$ other than those which contain $p$, so $E_{\gamma}^{\f{D}^k}=E_{\gamma}^{\f{D}_p^k}$.  Lemma \ref{lem:CPS} thus implies the claim to order $k$.  Since $k$ was arbitrary, the claim holds to arbitrary finite order, hence also in the desired limit.
\end{proof}

\subsection{Generalizations}\label{sub:gen}
There are several other flavors and generalizations of scattering diagrams to which our results also apply.  For the sake of readability, do not work in the most general possible setting, but instead briefly outline some generalizations here.

\begin{enumerate}
    \item In place of $\wh{A}=\Z[x^M]\llb y^{\sQ}\rrb$, one might use the formal Laurent series ring $$\wh{A}^{\circ}\coloneqq \Z[x^M]\llb y^{\sQ^{\gp}}\rrb\coloneqq \Z[x^M][y^{\sQ^{\gp}}]\otimes_{\Z[x^M][y^\sQ]} \Z[x^M]\llb y^\sQ\rrb.$$
The scattering functions are still required to lie in $\wh{A}$, but the indices for the theta functions may be any $u\in M\oplus \sQ^{\gp}$, and the resulting theta functions yield a topological basis for $\wh{A}^{\circ}$.
    \item Rather than requiring the support of a wall to be a cone, one can more generally allow affine cones; i.e., not requiring the apex to be at the origin, and allowing $W$ parallel to $n^{\perp}$ rather than contained in $n^{\perp}$.  Incoming walls are then those whose support is closed under addition by $\R_{\geq 0} v$ for $-v$ the direction of the wall.  This generality is used, for example, when applying the ``perturbation trick'' in \cite[\S C.3]{GHKK} and throughout \cite{Man3}.
    \item Sometimes, especially in the cluster setting, one might view the scattering diagram as living not in $M_{\R}$ but in $M_{\R}\oplus \sQ^{\gp}_{\R}$ (often without specifying a splitting).  Then, rather than lying in $n^{\perp}$ for $n\in N$, the walls will lie in $(n')^{\perp}$ for $n'\in N\oplus \check{\sQ}^{\gp}$.  This can be somewhat more restrictive for the scattering functions since now $f\in 1+z^{m'}\Z\llb z^{m'}\rrb$ for $m'\in (n')^{\perp}$ rather than the less restrictive condition $f\in A_m^{\parallel}$ for $m\in n^{\perp}$.  On the other hand, this approach has the advantage of allowing for many more coordinate charts $\iota_p:\wh{A}^{\can}\rar \wh{A}_p$ because broken lines can now end at any generic $p$ in $M_{\R}\oplus \sQ^{\gp}_{\R}$ rather than just $p\in M_{\R}$.  Results like universal positivity still hold by the same arguments.    
    \item In the Gross-Siebert program, one typically chooses a $\sQ$-convex multivalued integral piecewise-linear function $\varphi$ on $M_{\R}$ \cite{GS11}.  This is essentially the data of an element $\kappa_{\rho}\in \sQ$ for a number of codimension-one polyhedral cones $\rho\in M_{\R}$ such that there locally exists a $\sQ_{\R}$-valued piecewise-linear function $\varphi$ on $M_{\R}$ with kinks given by the elements $\kappa_{\rho}$ (i.e., $\varphi$ on one side $\sigma^+$ of a wall $\rho$ is equal to the linear extension of $\varphi$ from the opposite side $\sigma^-$ plus $n_{\rho}\otimes \kappa_{\rho}$ for $n_{\rho}\in N$ a primitive vector vanishing along $\rho$ and positive on $\sigma^-$).  One then modifies the scattering diagram by inserting the additional ``slabs'' $(\rho,y^{\kappa_{\rho}},n_{\rho})$ for each of these cones $\rho$.  These slabs are treated like walls even though they are equivalent to $0$ modulo $\wh{\s{I}}$ instead of $1$.
    \item In \cite{GHS}, one considers scattering diagrams not in a vector space $M_{\R}$, but on a real polyhedral affine pseudomanifold $B$, possibly with a codimension-two singular locus $\Delta$.  The lattices $M$ and $N$ are replaced with sheaves of integral tangent/cotangent vectors on $B\setminus \Delta$.  By locally identifying $B\setminus \Delta$ with a neighborhood of the origin in $M_{\R}$, we may define scattering diagrams on $B$.  Here one incorporates a $\sQ$-convex multivalued integral piecewise-linear function $\varphi$ on $B$ as above, so the scattering diagram may include slabs.  Theta functions then are associated to the ``asymptotic monomials'' of $B$ and are again constructed via broken lines.
\end{enumerate}

We believe that the techniques of the next section---particularly the proof of VIT---apply to the above setups with almost no significant changes (except that the convexity condition on $\s{M}^+$ seems to require some nontrivial modification for the settings described in (5)).

\section{Tautness and the Valuative Independence Theorem}\label{sec:Taut-VIT}

\subsection{Rational broken lines}\label{sub:rational-bl}

In this subsection and the next, the positivity assumption on $\f{D}$ is only used for Lemma \ref{lem:Theta} and Corollaries \ref{cor:Theta-theta} and \ref{cor:Q-CPS}.

We wish to characterize the broken lines $\Gamma$ which are $\val_{v,p}$-minimizing.  To do this, we first define ``rational broken lines'' which will allow us to utilize the fact that a global minimum must also be a local minimum.

Consider $u\in M\oplus \sQ$ and a generic path $\gamma$ for a scattering diagram $\f{D}$ such that $u$ and $\gamma(t-\epsilon)$ lie on the same side of the walls crossed at time $t$.  Let $S^{\circ}_{u,\gamma,t}$ be the set of exponents of terms appearing in $E_{\gamma,t}(z^u)$, and let $\?{S}_{u,\gamma,t}$ be the convex hull of $S^{\circ}_{u,\gamma,t}$ in $M_{\bb{Q}}\oplus \sQ^{\gp}_{\Q}$.

Equivalently, let $(W_i,f_i,n_i)$ be the walls containing $\gamma(t)$.  Up to equivalence, we can assume each $n_i$ equals some common element $n\in N$ such that $u\cdot n>0$.  Denote $k=u\cdot n$ and $f\coloneqq\prod_i f_i$.  Then $\?{S}_{u,\gamma,t}$ is the convex hull in $M_{\bb{Q}}\oplus \sQ^{\gp}_{\Q}$ of the set $S_{u,\gamma,t}$ of elements of the form $u+km$ for $m\in M\oplus \sQ$ an exponent of a term in $f$.  We use this characterization to define $S_{u,\gamma,t}$ and $\?{S}_{u,\gamma,t}$ for all $u\in  M_{\Q}\oplus \sQ^{\gp}_{\Q}$, not just the lattice points.\footnote{One can even extend the definitions of $S_{u,\gamma,t}$ and $\?{S}_{u,\gamma,t}$ to all $u\in  M_{\R}\oplus \sQ^{\gp}_{\R}$, not just the rational points.  In this case, the convex hull is computed not in $M_{\bb{Q}}\oplus \sQ^{\gp}_{\Q}$ but in the $\Q$-span of $S_{u,\gamma,t}$ (which we note is still countable).  With this we can define rational broken lines and the sets $\Theta_{u,p}$ even for irrational $u$, and the results of this section still hold via the same arguments.\label{foot:irr}} 

Now define a \textbf{rational broken line} with initial exponent $u\in M_{\Q}\oplus \sQ_{\Q}$ and generic endpoint $p\in M_{\R}$ to be a continuous piecewise affine-linear path $\Gamma:\R_{\leq 0}\rar M_{\R}\setminus \Joints(\f{D})$, with elements of $M_{\Q}\oplus \sQ^{\gp}_{\Q}$---referred to as \textbf{exponents}---associated to each straight segment, such that
\begin{itemize}
    \item $\Gamma(0)=p$;
    \item the exponent associated to the initial straight segment is $u$.  In particular, an initial straight segment exists, i.e., $\Gamma$ has finitely many bends;
    \item if $(m,q)$ is the exponent associated to a straight segment of $\Gamma$, then $-\Gamma'=m$ everywhere on this segment;
    \item if the associated exponent changes from $u_i$ to $u_{i+1}$ when crossing $t=t_i$, then $u_{i+1}\in \?{S}_{u_i,\Gamma,t_i}$.
\end{itemize}
We write $u_{\Gamma}=(m_{\Gamma},q_{\Gamma})$ for the exponent associated to the final straight segment.

Note that if $u\in M\oplus \sQ$, and if in the last condition we instead required $u_{i+1}\in S^{\circ}_{u_i,\Gamma,t_i}$ (as opposed to the convex hull $\?{S}_{u_i,\Gamma,t_i}$), then the resulting data would be the support and attached exponents for an ordinary broken line.  We may abuse terminology and refer to such rational broken lines as ordinary broken lines.

For each $u\in M_{\Q}\oplus \sQ^{\gp}_{\Q}$ and each generic\footnote{For $u\in M_{\R}\oplus \sQ^{\gp}_{\R}$ (possibly irrational, cf. Footnote \ref{foot:irr}), the set $S_u$ of elements of $M_{\R}\oplus \sQ^{\gp}_{\R}$ which could be attached to a segment of a rational broken line with initial exponent $u$ is countable.  Now, given $p\in M_{\R}\oplus \sQ^{\gp}_{\R}$, if we work backwards from $p$ to construct a rational broken line whose final exponent lies in $S_u$, there will only be countably many possibilities, so these rational broken lines will not run into joints so long as $p$ is sufficiently generic relative to $u$.} $p\in M_{\R}$, define
\begin{align}\label{eq:ThetaQ}
    \Theta_{u,p}=\{u_{\Gamma}\in M_{\Q}\oplus \sQ^{\gp}_{\Q}| \Gamma \text{ a rational broken line with } \Ends(\Gamma)=(u,p)\}.
\end{align}
Then for each $v\in N_{\R}\oplus \check{\sQ}^{\gp}_{\R}$, define 
$$\val_v(\Theta_{u,p})=\inf_{w\in \Theta_{u,p}} w\cdot v.$$

We make the following observation:
\begin{lem}\label{lem:Theta}
    For $u\in M\oplus \sQ$ and $p$ generic, $\Theta_{u,p}$ is the same as the set of elements in $M_{\Q}\oplus \sQ^{\gp}_{\Q}$ of the form $\frac{1}{k}u_{\Gamma}$ for $k\in \Z_{\geq 1}$ and $\Gamma$ an ordinary broken line with endpoint $p$ and initial exponent $ku$.\footnote{Without assuming positivity of $\f{D}$, we still see that $\Theta_{u,p}$ includes all elements of the form $\frac{1}{k}u_{\Gamma}$ as described in the lemma, but the reverse containment seems unclear.\label{foot:no-pos}}
\end{lem}
Lemmas \ref{lem:Theta} and \ref{lem:val-ku} together imply the following:
\begin{cor}\label{cor:Theta-theta}
    For $u\in M\oplus \sQ$ and $p$ generic, $\val_v(\Theta_{u,p})=\val_v(\vartheta_{u,p})$.
\end{cor}

\begin{cor}\label{cor:Q-CPS}
Let $u\in M_{\Q}\oplus \sQ_{\Q}$.  If $p_1,p_2\in M_{\R}$ are generic points in the same chamber of $\f{D}$, then $\Theta_{u,p_1}=\Theta_{u,p_2}$.
\end{cor}
\begin{proof}
    This is immediate from Lemma \ref{lem:Theta} combined with Lemma \ref{lem:CPS}.
\end{proof}

Recall the limiting theta functions of \S \ref{sub:lim}.  In light of Lemma \ref{lem:Theta}, we define 
\begin{align}\label{eq:limT}
\lim_{\epsilon \rar 0^+} \Theta_{u,p+\epsilon \mu}\coloneqq \left\{\left.\frac{1}{k} w\in M_{\Q}\oplus \sQ_{\Q}^{\gp}\right|k\in \Z_{\geq 1}, w \text{ an exponent of a term in }\lim_{\epsilon\rar 0^+} \vartheta_{ku,p+\epsilon \mu}\right\}
\end{align}
for possibly non-generic $p$ and generic $\mu$.

\subsection{Taut broken lines}\label{sub:taut}

Let $\Gamma$ be a rational broken line in $M_{\R}$ with endpoint $p$.  Suppose $\Gamma$ crosses a wall $\f{d}\in \f{D}$ at time $t$. Let $u_s$ denote the element of $M_{\Q}\oplus \sQ^{\gp}_{\Q}$ attached to $\Gamma$ at generic time $s$, and let $T_t:M_{\R}\oplus \sQ^{\gp}_{\R}\rar M_{\R}\oplus \sQ^{\gp}_{\R}$ be the unique linear map which fixes $\f{d}\oplus \sQ^{\gp}_{\R}$ and maps $u_{t-\epsilon}$ to $u_{t+\epsilon}$ for small $\epsilon >0$.  Then define $T^{\vee}_t:N_{\R}\oplus \check{\sQ}^{\gp}_{\R}\rar N_{\R}\oplus \check{\sQ}^{\gp}_{\R}$ to be the adjoint map, so $T_t(u)\cdot v=u\cdot T_t^{\vee}(v)$ for all $u\in M_{\R}\oplus \sQ^{\gp}_{\R}$ and $v\in N_{\R}\oplus \check{\sQ}^{\gp}_{\R}$.

Now fix $v\in N_{\R}\oplus \check{\sQ}^{\gp}_{\R}$.  We can transport $v$ backwards along $\Gamma$, applying $T_t^{\vee}$ whenever crossing backwards over a point $\Gamma(t)$ where $\Gamma$ bends.  Let $v_t$ be the resulting element of $N_{\R}\oplus \check{\sQ}^{\gp}_{\R}$ associated to $\Gamma$ at generic time $t\in \R_{\leq 0}$.

Recall from our definition of rational broken lines that for each $t\in \R_{<0}$ and small $\epsilon>0$, $u_{t+\epsilon}\in  \?{S}_{u_{t-\epsilon},\Gamma,t}$.  We say that $\Gamma$ is $v$-\textbf{taut} if for each $t\in \R_{<0}$, we have \begin{align}\label{eq:taut-ineq}
u_{t+\epsilon} \cdot v_{t+\epsilon} \leq u'_t\cdot v_{t+\epsilon}
\end{align}
for all $u'_t\in \?{S}_{u_{t-\epsilon},\Gamma,t}$.  In other words, being $v$-taut means that $\Gamma$ always takes a bend which minimizes the pairing with $v_t$ immediately on the new side of the wall.

\begin{rem}\label{rem:taut-integral}
    We note that for $u\in M\oplus \sQ$ and $v$ generic, $v$-taut broken lines are in fact ordinary broken lines.  This is because each minimizing bend must correspond to a vertex of $\?{S}_{u_{t-\epsilon},\Gamma,t}$, hence an element of $S^{\circ}_{u_{t-\epsilon},\Gamma,t}$.
\end{rem}

\subsection{Minimizing broken lines are taut}

\begin{thm}\label{thm:min-taut}
Without assuming that $\f{D}$ is positive, let $\Gamma$ be a rational broken line with initial exponent $u\in M_{\Q}\oplus \sQ_{\Q}$ and generic endpoint $p$ such that $\Gamma$ is $\val_{v,p}$-minimizing amongst such rational broken lines; i.e., $u_{\Gamma} \cdot v=\val_{v,p}(\Theta_u)$.  Then $\Gamma$ is $v$-taut.  If $u\in M\oplus \sQ$ and $v$ is generic, then $\Gamma$ is an integral broken line and $\val_{v,p}(\Theta_u)=\val_{v,p}(\vartheta_u)$.

If $\f{D}$ is positive and $u\in M\oplus \sQ$, then we have $\val_v(\Theta_{u,p})=\val_v(\vartheta_{u,p})$ for all $v$, and the condition for $\Gamma$ to be $\val_{v,p}$-minimizing can be checked amongst ordinary broken lines; i.e., $\Gamma$ with ends $(u,p)$ must be $v$-taut whenever $u_{\Gamma} \cdot v=\val_{v,p}(\vartheta_u)$.
\end{thm}

\begin{proof}
    Suppose $\Gamma$ is a rational broken line which has initial exponent $u$ and endpoint $p$ but which is not $v$-taut.  Fix $k\in \Z_{\geq 1}$ large enough so that all non-trivial bends and at least one non $v$-taut bend of $\Gamma$ are along walls of $\f{D}^k$.
    Choose $t_0\in \R_{<0}$ where the $v$-tautness condition fails along a wall of $\f{D}^k$.  While keeping $\Gamma|_{(-\infty,t_0]}$ fixed and all bends after $t_0$ unchanged---i.e., the linear maps $T_t$ for $t>t_0$ are unchanged (except that the values of $t$ where the wall-crossings occur might shift), so $v_{t_0+\epsilon}$ is unchanged---we can slightly deform the bend of $\Gamma$ at $t_0$ to get a new broken line $\wt{\Gamma}$ which has a smaller value for $-\wt{\Gamma}'(t_0+\epsilon)\cdot v_{t_0+\epsilon}$.  This deformation can be taken to be small enough to avoid deforming through $\Joints(\f{D}^k)$, so the walls of $\f{D}^k$ crossed by $\wt{\Gamma}$ are unchanged.  Furthermore, with some additional deformation that avoids $\Joints(\f{D}^k)$---translating $\wt{\Gamma}$ and then possibly adjusting the lengths of the straight segments without adjusting the bends---we can make it so that $\wt{\Gamma}$ still ends at the same generic point $p$.  Thus, $u_{\wt{\Gamma}}=-\wt{\Gamma}'(0)\in \Theta_{u,p}$.  Since
    \begin{align*}
     -\wt{\Gamma}'(0)\cdot v = -\wt{\Gamma}'(t_0+\epsilon)\cdot v_{t_0+\epsilon} < -\Gamma'(t_0+\epsilon)\cdot v_{t_0 +\epsilon}= -\Gamma'(0)\cdot v,
    \end{align*}
    we see that $\Gamma$ could not have been $\val_{v,p}$-minimizing.  The first claim follows.

    The claim that $v$-taut $\Gamma$ is an ordinary broken line whenever $u\in M\oplus \sQ$ and $v$ is generic follows from Remark \ref{rem:taut-integral}.  Furthermore, this $\Gamma$ is the unique broken line with ends $(u,p)$ and final exponent $u_{\Gamma}$ because any other such broken line $\wt{\Gamma}$ would also be $\val_{v,p}$-minimizing, hence $v$-taut, hence  uniquely determined by starting at $\wt{\Gamma}(0)=p$ with exponent $u_{\Gamma}$, then tracing backwards in the $(u_{\Gamma})_M$-direction, taking $v$-taut bends at each wall-crossing (and these bends are uniquely determined since $v$ is generic).  It follows that $m_{\Gamma}$ appears as an exponent in $\vartheta_{u,p}$ (it will not be canceled by other broken lines, even without assuming positivity). The equality $\val_v(\vartheta_{u,p})\leq \val_v(\Theta_{u,p})$ follows immediately, while the reverse inequality follows from Footnote \ref{foot:no-pos} (the rational broken lines include all the ordinary broken lines).
    
    With the positivity assumption, the equality $\val_v(\Theta_{u,p})=\val_v(\vartheta_{u,p})$ is just Corollary \ref{cor:Theta-theta}.  The fact that it suffices to consider ordinary broken lines follows immediately from this equality.
\end{proof}

\begin{rem}\label{rem-not-positive}
    We note that the positivity assumption is important because we do not generally know how to determine the existence of a $\val_{v,p}$-minimizing rational broken line, considered amongst all rational broken lines. 
 Lemma \ref{lem:inf=min} shows that $\val_{v,p}$-minimizing ordinary broken lines almost always exist if $\val_{v,p}$ is finite, but extending this to the rational setting relies on Corollary \ref{cor:Theta-theta}, which relies on Lemma \ref{lem:val-ku}, which utilizes positivity.
\end{rem}

\subsection{Tropicalizations of distinct theta functions are distinct}

We return to always assuming that $\f{D}$ is positive.

\begin{thm}\label{thm:trop-determines-u}
    Fix generic elements $v\in N_{\R}\oplus \check{\sQ}_{\R}^{\gp}$ and $p\in M_{\R}$.  If $\val_{v,p}(\vartheta_{u_1})$ and $\val_{v,p}(\vartheta_{u_2})$ are finite and equal, then $u_1=u_2$.
\end{thm}
\begin{proof}
 Since $v$ is generic, there can only be one element $u'\in M\oplus \sQ$ for which $\val_{v,p}(\vartheta_u)=u'\cdot v$.  This $u'$ must be the final exponent $u_{\Gamma}$ for a $\val_{v,p}$-minimizing broken line $\Gamma$ with ends $(u,p)$---here we use Lemma \ref{lem:inf=min} and the genericness of $v$ to ensure such a broken line exists.  By Theorem \ref{thm:min-taut}, $\Gamma$ must be $v$-taut.  Furthermore, as we saw in the proof of Theorem \ref{thm:min-taut}, a broken line $\Gamma$ with endpoint $p$ and given $u_{\Gamma}$ is uniquely determined by the $v$-tautness condition---one simply starts tracing $\Gamma$ backwards from $p$ in the direction $-\Gamma'(0)=(u_{\Gamma})_M$, bending as necessary to ensure the $v$-tautness (the genericness of $v$ ensures the uniqueness of the $v$-taut bends).  In particular, the initial exponent $u$ is uniquely determined, so the claim follows.
\end{proof}

In the language of \cite[Def. 3.1]{JP}, Theorem \ref{thm:trop-determines-u} essentially says that tropical theta functions are tropically independent.

\subsection{The Valuative Independence Theorem}

\begin{thm}[Valuative independence]\label{thm:indep}
For any $v\in N_{\R}\oplus \sQ^{\gp}_{\R}$ and generic $p\in M_{\R}$,
    \begin{align}\label{eq:VIT}
        \val_{v,p}\left(\sum_{u\in M\oplus \sQ} c_u\vartheta_u\right)=\inf_{u:c_u\neq 0} \val_{v,p}(\vartheta_u)
    \end{align}
assuming that the value on the right-hand side is finite. 
\end{thm}
\begin{proof}
    Note that it is always true, for any (possibly infinite) collection of functions $f_i\in \wh{A}^{\can}$ indexed by $i$ such that $\sum_i f_i$ is well-defined in $\wh{A}^{\can}$, that $$\val_{v,p}\left(\sum_i f_i\right)\geq \inf_i\left(\val_{v,p}(f_i)\right).$$ 
    If $v$ is generic and the right-hand side is finite (hence a $\min$ by Lemma \ref{lem:inf=min}), then the inequality can only fail to be an equality if two or more elements $f_i$ have the same valuation---this is necessary for the $\val_{v,p}$-minimizing terms to be able to cancel.  We know from Theorem \ref{thm:trop-determines-u} that we will never have $\val_{v,p}(c_{u_1}\vartheta_{u_1})=\val_{v,p}(c_{u_2}\vartheta_{u_2})$ for generic $v$ and distinct $u_1,u_2$, so the claim follows for generic $v$.  
    
    To extend to arbitrary (not necessarily generic) $v$, first recall from Example \ref{ex:Nplus} that every $\vartheta_{u,p}^{\trop}$ is finite on $\s{N}^+$, which we've assumed is top-dimensional in $N_{\R}\oplus \sQ^{\gp}_{\R}$.  By convexity of $\vartheta_{u,p}^{\trop}$, the set where the right-hand side of \eqref{eq:VIT} is finite is a single convex cone, necessarily a top-dimensional cone since it contains $\s{N}^+$.  So every $v$ for which the right-hand side of \eqref{eq:VIT} is finite must be arbitrarily close to some \textit{generic} values $v'$ where the right-hand side is finite and a min, and we know from above that the equality holds for such $v'$.  The claim for $v$ now follows from the fact that tropicalized functions (hence both sides of \eqref{eq:VIT} when viewed as functions of $v$) are continuous on the loci where they are finite.
\end{proof}

\subsection{Valuative independence of theta functions with specialized coefficients}\label{sub:lin-indep}

Here we restrict to theta functions $\vartheta_m\coloneqq \vartheta_{(m,0)}$ with $m\in M$ and, similarly, valuations $\val_{n,p}\coloneqq \val_{(n,0),p}$ with $n\in N_{\R}$.  All uses of the dual pairing between $M\oplus \sQ^{\gp}$ and $N\oplus \check{\sQ}^{\gp}$ are restricted to the dual pairing between $M$ and $N$ (and similarly when tensored with $\R$).  E.g.,  $\vartheta_{m,p}^{\trop}\coloneqq \vartheta_{(m,0),p}^{\trop}|_{(N_\R,0)}$.  Moreover, the definition of $v$-taut from \S \ref{sub:taut} is adapted to define $n$-taut via the following modifications:
\begin{itemize}
    \item when defining rational broken lines, we ignore (i.e., mod out) the $\sQ^{\gp}_{\Q}$-directions to define $\?{S}_{m_i,\Gamma,t_0}$ and $\Theta_{m,p}$ in $M_{\Q}$;
    \item the linear maps $T_t$ are replaced with their restrictions/projections to $M_{\R}\rar M_{\R}$, and the adjoint maps $T_t^{\vee}$ are the resulting adjoint maps $N_{\R}\rar N_{\R}$.  We thus define elements $n_t$ similarly to the definition of $v_t$ in \S \ref{sub:taut};
    \item the tautness condition \eqref{eq:taut-ineq} is modified by restricting the dual pairing and using $n_{t+\epsilon}$ in place of $v_{t+\epsilon}$.
\end{itemize}
    With these changes, the proof of Theorem \ref{thm:min-taut} is easily adapted to show that $\val_n$-minimizing broken lines are $n$-taut for this modified version of tautness.

  We will need the following additional restriction on our scattering diagrams.

  \begin{ass}\label{assum:A-parallel}
      Recall that each scattering function $f$ lies in $A_m^{\parallel}\coloneqq \Z[x^m]\llb y^{\sQ}\rrb\cap (1+\wh{\s{I}})\subset \wh{A}$ for some $m\in M$.  We assume that $1$ is the only term of $f$ whose $x$-exponent is $0$.  Additionally, let $x^{km}$ be the highest power of $x^m$ appearing in $f$ (assuming such a highest power exists).  Then there is only one term in $f$ whose $x$-exponent is this $km$.

      In particular, it suffices to assume that each scattering function $f$ is contained in $1+z^u\Z\llb z^u\rrb$ for some $u=(m,q)$.  This is always the case in the cluster setting (cf. Lemma \ref{lem:D-skew}) and in the setup of \S \ref{sub:gen}(3).
  \end{ass}
  
  Now let $A^{\midd}_p\subset \wh{A}_p$ be the subalgebra generated over $\Z[y^{\sQ}]$ by only those theta functions $\vartheta_{m,p}$, $m\in M$, which are \textit{finite} Laurent polynomials.  Consider any ring $R$ and any ring homomorphism $\nu:\Z[y^{\sQ}]\rar R$ which does not map monomials to $0$ or to zero divisors---i.e., for all $q\in \sQ$ and $b\in \Z\setminus \{0\}$, $\nu(by^q)$ is nonzero and is not a zero divisor.  Extend $\nu$ to $A^{\midd}_p\rar R[x^M]$ in the natural way.  Given $$f=\sum_{m\in M} \left(\sum_{q\in \sQ} b_{m,q} y^q\right) x^m\in A_p^{\midd}$$ let us denote $a_m\coloneqq \sum_q \nu(b_{m,q}y^{q})\in R$ for each $m\in M$, so $\nu(f)=\sum_m a_m x^m$.  Then for each $n\in N_{\R}$ we define
  \begin{align}\label{eq:valnuf}
 \val_n(\nu(f))\coloneqq \nu(f)^{\trop}(n)\coloneqq \min_{m:a_m\neq 0}  m\cdot n
  \end{align}
where $\min(\emptyset)\coloneqq \infty$.  In this way we define $\nu(f)^{\trop}:N_{\R}\rar \R\cup \{\infty\}$.
\begin{thm}\label{thm:nu}
    Under Assumption \ref{assum:A-parallel}, for $\nu$ as above and any $\Z$-linear combination of theta functions $f=\sum_m b_m \vartheta_{m,p}\in A_p^{\midd}$, $b_m\in \Z$, we have $f^{\trop}=\nu(f)^{\trop}$.  Furthermore, the theta functions $\nu(\vartheta_{m,p})$ for $\vartheta_{m,p}\in A_p^{\midd}$ satisfy valuative independence:
    \begin{align}\label{eq:val-ind-specialized}
        \val_{n,p}\left(\sum_{m\in M} c_m \nu(\vartheta_m) \right) = \min_{m:c_m\neq 0} \val_{n,p}\nu(\vartheta_m)
    \end{align} for any coefficients $c_m\in R$ with only finitely many being nonzero.  In particular, the theta functions $\nu(\vartheta_{m,p})$ for $\vartheta_{m,p}\in A_p^{\midd}$ are linearly independent.
\end{thm}
\begin{proof}
     As in Theorem \ref{thm:trop-determines-u} but restricted to $N_{\R}\oplus \{0\}$, for each $\vartheta_{m,p}\in A^{\midd}_p$ and any generic $n\in N_{\R}$, 
     $\val_{n,p}(\vartheta_m)$ equals $m_{\Gamma}\cdot n$ where $m_{\Gamma}$ is the final $x$-exponent for some $n$-taut broken line $\Gamma$ with ends $(m,p)$.    Furthermore, $\Gamma$ is the unique $n$-taut broken line ending at $p$ with final exponent $m_{\Gamma}$---the genericness of $n$ forces the $M$-component of each $n$-taut bend to be extremal, and then Assumption \ref{assum:A-parallel} ensures that the extremal bend can only be attained for one term.  By assumption, $\nu(c_{\Gamma}y^{q_{\Gamma}})$ is neither $0$ nor a zero divisor.
    
    Now consider $f=\sum b_m \vartheta_{m,p}\in A^{\midd}_p$.  By our above observations, the monomials determining the linear parts of $f^{\trop}$ all have coefficients which do not vanish or cancel when applying $\nu$.  Thus, the tropicalization is unchanged after applying $\nu$.  
    
    The valuative independence claim \eqref{eq:val-ind-specialized} is more general since we allow coefficients $c_m\in R$ which might not lie in the image of $\nu$.  But we know that for generic $n$, the minimum on the right-hand side of \eqref{eq:val-ind-specialized} is attained for a unique monomial $\nu(b_{\Gamma}x^{m_{\Gamma}}y^{q_{\Gamma}})$ appearing as a term in a unique $\vartheta_m$ with $c_m\neq 0$.  Then \eqref{eq:val-ind-specialized} follows since $c_m\nu(b_{\Gamma}y^{q_{\Gamma}})$ is nonzero by assumption on $\nu$. Continuity lets us extend from generic $n$ to all $n$.
    
    In particular, if $g\coloneqq \sum_m c_m \nu(\vartheta_m)=0$, then $\val_n(g)=\infty$, hence every coefficient $c_m$ must have been $0$ to make the right-hand side of \eqref{eq:val-ind-specialized} equal to infinity.  This proves the linear independence. 
\end{proof}

\cite{GHKK} considers the case where $\nu$ is the map specializing from a cluster algebra with principal coefficients to a coefficient-free cluster algebra or to more general cluster algebras with specialized coefficients.  The injectivity of this $\nu$, and thus the linear independence of the theta functions with specialized coefficients, was conjectured to hold in general but has previously only been proved in special cases.\footnote{\cite[Thm. 0.3(7)]{GHKK} shows injectivity of $\nu$ under a certain convexity assumption on the seed data, while \cite[Thm. 7.16(7)]{GHKK} shows injectivity of $\nu$ for \textit{generic} specialized coefficients.  Injectivity of $\nu$ on the cluster complex is shown in \cite[Thm. 7.20]{GHKK}.  This is extended to the closure of the cluster complex in \cite[Prop. 5.19]{ManQin}.\label{foot:inj-nu}}  We have now proved the following:
\begin{cor}\label{cor:nu}
    The map $\nu$ of \cite[Thm. 0.3]{GHKK} is always injective.
\end{cor}

    \subsubsection{Enough global functions and canonical singularities}\label{sec:sing}  

    Continuing with the notation from above, consider a subalgebra $A\subset A_p^{\midd}$.  We say that $A$ (respectively, $\nu(A)$) has enough global functions (EGF) if, for every nonzero $n\in N$, there exists an $f\in A$ (respectively, $f\in \nu(A)$) such that $\val_n(f)<0$.  In particular, if $\vartheta_{m,p}\in A_p^{\midd}$ for all $m\in M$, then $A_p^{\midd}$ has EGF.  Indeed, given nonzero $n\in N$, we can always find $m\in M$ such that $m\cdot n<0$, and then, since $z^m$ is a term in $\vartheta_{m,p}$, we have $\val_n(\vartheta_{m,p})\leq \val_n(z^m)<0$.
    
    Theorem \ref{thm:nu} immediately implies the following.
    \begin{cor}
        Under Assumption \ref{assum:A-parallel}, $A$ has enough global functions if and only if $\nu(A)$ has enough global functions.
    \end{cor}

    Much of \cite[\S8-\S9]{GHKK} makes the assumption of enough global monomials (EGM), but it is suggested just before their Prop.~ 8.13 that many of these results might extend to the setting of enough global functions (particularly now that the min-convexity result of loc.~cit.~ has been proved very generally; cf. Prop. \ref{prop: BL} and the subsequent references to \cite[Lem. 15.6]{KY} and \cite[\S 10.7]{KY2}).  Here we present one geometric consequence of EGF that was shown to us by Sean Keel.

    Let us briefly recall the notion of canonical singularities \cite{Reid}.  Let $U$ be a normal variety over a field with $\Q$-Cartier canonical class $K_U$, and let $f:U'\rar U$ be a resolution of singularities of $U$ with exceptional divisors $E_1,\ldots,E_k$.  Then $$K_{U'}=f^* K_U + \sum_{i=1}^k a_i E_i$$ for some $a_i\in \Q$ called the discrepancies.  One says that $U$ has (at worst) \textbf{canonical singularities} if $a_i \geq 0$ for all $i$.

    Now, under Assumption \ref{assum:A-parallel}, consider $\nu:A_p^{\midd}\rar R[x^M]$ as above with $R$ a field.  Let $A\subset A^{\midd}_p$ be a subalgebra such that $U\coloneqq \Spec \nu(A)$ is a (possibly singular) affine log Calabi-Yau variety with $U(\Z^{\min})=N$.

    \begin{cor}\label{cor:sing}
        If $\nu(A)$ has EGF (equivalently, if $A$ has EGF), then $U$ as above has at worst canonical singularities.
    \end{cor}

    \begin{proof}
        Since $K_U$ is assumed to be trivial, $U$ having canonical singularities means that $K_{U'}$ is effective; i.e., the (pullback of the) log volume form $\Omega$ on $U$ does not have a pole along any $E_i$.  Recall that $U(\Z^{\min})$ can be understood as the set of divisorial discrete valuations on the function field $K(U)$ along which $\Omega$ has a pole (together with the $0$-valuation).   So showing that $U$ has canonical singularities is equivalent to showing that none of the exceptional divisors $E_i$ correspond to points of $U(\Z^{\min})=N$.

        Since each exceptional divisor $E_i$ has center in $U=\Spec \nu(A)$, we have $\val_{E_i}(f)\geq 0$ for each $i$ and every $f\in \nu(A)$.  So, to show that $U$ has canonical singularities, it suffices to show that every $n\in U(\Z^{\min})\setminus \{0\}$ satisfies $\val_n(f)<0$ for some $f\in \nu(A)$ (because then $n$ cannot correspond to an $E_i$).  This is exactly the EGF condition.
    \end{proof}

    In fact, the above argument shows that EGF implies canonical singularities for any log CY variety $U$ with maximal boundary.  In this context, EGF means that, for every nonzero $v\in U(\Z^{\min})$, there exists a global function $f\in \Gamma(U,\s{O}_U)$ such that $\val_v(f)<0$.

    We note that \cite[Thm. 8.32]{GHKK} shows that general fibers of mirrors to cluster varieties with enough global monomials are log canonical (i.e., discrepancies are $\geq -1$).  Corollary \ref{cor:sing} suggests that we should have \textit{canonical} singularities on all fibers, not just general fibers.  In particular, we expect one could apply this approach to recover \cite[Thm. 4.8]{BMRS} (locally acyclic cluster algebras have at worst canonical singularities).

    The program of Keel-Yu \cite{KY,KY2} applies to smooth affine log Calabi-Yau varieties with maximal boundary, but it is expected that the smoothness condition can be weakened to allow for canonical singularities.  Corollary \ref{cor:sing} provides evidence for the conjecture in \cite[\S 13]{KY2} that fibers of the mirror family have canonical singularities, which in turn is useful for stating the double mirror conjecture \cite[Conj. 13.3]{KY2}.

\subsection{The Theta Function Extension Theorem}\label{sub:Theta-ext}

Let $\sQ_1=\sQ$.  By an extension of $\sQ_1$, we shall mean a monoid $\sQ_2$ which contains $\sQ_1$ as a face; that is, $\sQ_2\supset \sQ_1$, and there exists some $w\in \check{\sQ}^{\gp}_2$ with $w|_{\sQ_1}=0$ and $w|_{\sQ_2\setminus (\sQ_2\cap \sQ_{1,\R})} < 0$.  Note that scattering diagrams over $\sQ_1$ can naturally be viewed as scattering diagrams over $\sQ_2$, and similarly, for $\wh{A}_i$ denoting $\Z[M]\llb \sQ_i\rrb$, we naturally have $\wh{A}_1\subset \wh{A}_2$.

\begin{thm}[Theta function extension theorem]\label{thm:theta-extension-general}
    Let $\sQ_2$ be an extension of $\sQ_1$ as above.  Let $\f{D}_1$ and $\f{D}_2$ be positive scattering diagrams in $M_{\R}$ for the monoids $\sQ_1$ and $\sQ_2$, respectively, such that $\f{D}_1 \subset \f{D}_2$.  Furthermore, suppose that for every monomial $cx^my^q\neq 1$ of every scattering function of $\f{D}_2\setminus \f{D}_1$, we have $q\notin \sQ_{1,\R}$.  If $\vartheta_{u,p}^{\f{D}_1}$ is a finite linear combination of theta functions $\vartheta_{u_i,p}^{\f{D}_2}$ for $\f{D}_2$, then $\vartheta^{\f{D}_1}_{u,p}=\vartheta^{\f{D}_2}_{u,p}$.
\end{thm}

A typical example (as we will demonstrate in Theorem \ref{thm:theta-ext-seed}) is where $\f{D}_1$ and $\f{D}_2$ are determined as in Theorem \ref{thm:ScatD} by some initial walls $\f{D}_{1,\In}$ and $\f{D}_{2,\In}$, respectively, where $\f{D}_{2,\In}$ consists of all walls of $\f{D}_{1,\In}$ plus some additional walls whose coefficients are independent of those from $\f{D}_1$.

\begin{proof}
    Pick $w\in \check{\sQ}_{\R}^{\gp}$ with $w|_{\sQ_1}= 0$ and $w|_{\sQ_2\setminus (\sQ_2\cap \sQ_{1,\R})} < 0$.
    
    If $\vartheta_{u,p}^{\f{D}_1}\neq \vartheta_{u,p}^{\f{D}_2}$, then there must be some broken line contributing to $\vartheta_{u,p}^{\f{D}_2}$ which bends at a wall of $\f{D}_2\setminus \f{D}_1$.  By our assumptions, the $y$-exponent of the corresponding monomial must lie in $\sQ_2\setminus (\sQ_2\cap \sQ_{1,\R})$.  Hence, $\val_w(\vartheta_{u,p}^{\f{D}_2})<0$ even though $\val_w(\vartheta_{u,p}^{\f{D}_1})=0$.  It therefore suffices to show that in fact $\val_w(\vartheta_{u,p}^{\f{D}_1}) \leq \val_w(\vartheta_{u,p}^{\f{D}_2})$.
    
    For any $u\in M\oplus \sQ$ and either $j=1$ or $2$, we have $\vartheta_{u,p}^{\f{D}_j}\in z^u\left(1+\wh{\s{I}}_j\right)$.  Here, $\wh{\s{I}}_j$ is the ideal $\wh{\s{I}}$ in $\wh{A}_j$ defined as in \S \ref{sub:scattering} for $\sQ=\sQ_j$. 
 In particular, $\wh{\s{I}}_1\subset \wh{\s{I}}_2$.  It follows that in the expansion $\vartheta_{u,p}^{\f{D}_1}=\sum_v c_v \vartheta_{v,p}^{\f{D}_2}$, we must have $c_u=1$.  So by the Theorem \ref{thm:indep}, we have $\val_w(\vartheta_{u,p}^{\f{D}_1}) \leq \val_w(\vartheta_{u,p}^{\f{D}_2})$, as desired.
\end{proof}

\subsection{Vertices of Newton polytopes}\label{sub:NP}

Let $f=\sum a_u z^u\in \wh{A}$.  The \textbf{Newton polytope} of $f$, denoted $\Newt(f)$, is the convex hull in $M_{\R}\oplus \sQ^{\gp}_{\R}$ of the elements $u\in M\oplus \sQ$ for which $a_u\neq 0$.  We say that $f$ is \textbf{monic} if $a_u=1$ whenever $u$ is a vertex of $\Newt(f)$.  We call $\f{D}$ monic if all of its scattering functions are monic.  E.g., cluster scattering diagrams are monic by \cite[Thm. 1.13]{GHKK}.

\begin{thm}\label{thm:Newt}
    Assume $\f{D}$ is monic.  Let $$f=\sum_{u\in M\oplus \sQ} a_u z^u = \sum_{w\in M\oplus \sQ} b_w \vartheta_{w,p}\in \wh{A}_p.$$  If $u$ is a vertex of $\Newt(f)$, then $a_u=b_w$ for some $w$.  In particular, if $f$ is a theta function or a finite sum of distinct theta functions (each with coefficient $1$), then $f$ is monic; that is, $a_u=1$ for each vertex $u$ of $\Newt(f)$.
\end{thm}
\begin{proof}
    If $u$ is a vertex of $\Newt(f)$, then there is some open cone $\sigma\subset N_{\R}\oplus \check{\sQ}^{\gp}_{\R}$ such that $f^{\trop}(v)=u\cdot v$ for all $v\in \sigma$.  By Theorem \ref{thm:min-taut},  the $\val_{v,p}$-minimizing broken lines for $f^{\trop}(v)$ must be $v$-taut.  As in the proof of Theorem \ref{thm:trop-determines-u}, for each generic $p\in M_{\R}$ and generic $v\in \sigma$, there is a unique $v$-taut broken line $\Gamma$ with endpoint $p$ and final exponent $u$.  Let $w$ be the initial exponent of $\Gamma$.  Since every bend of $\Gamma$ is extremal (by $v$-tautness and genericness of $v$), the monic condition on $\f{D}$ ensures that the coefficient of the final monomial for $\Gamma$ is $1$.  Thus, the coefficient $a_u$ of $z^u$ in $f$ must equal $b_w$.
\end{proof}

Note that one could similarly define Newton polytopes in $M_{\R}$ rather than $M_{\R}\oplus \sQ^{\gp}_{\R}$ after specializing each $y^q$ to, say, $\nu(y^q)=1$.  The monic assumption on $\f{D}$ is modified to say that, for each scattering function $f$, if $m$ is a vertex of the projection of $\Newt(f)$ to $M_{\R}$, then the coefficient of $x^m$ in $\nu(f)$ equals $1$.  An analog of Theorem \ref{thm:Newt} holds for this viewpoint via a similar argument using the modified version of tautness from \S \ref{sub:lin-indep}.

Since cluster variables are special cases of $\vartheta_{m,p}$ (for $m$ and $p$ in the ``cluster complex''), we obtain the following:

\begin{cor}
     Let $f=\sum a_m x^m$ be the Laurent expansion of a cluster variable in some cluster.  Let $u$ be a vertex of $\Newt(f)\subset M_{\R}$.  Then $a_u=1$.
\end{cor}

This reproves one direction of \cite[Conj. 4.15]{Fei-CombF}.   For another proof---along with a proof of the converse ($a_u=1$ implies $u$ is a vertex of $\Newt(f)$) and a related saturation result, cf. \cite[\S 4.2]{LP}.

\subsection{Tropical theta bases are atomic}\label{sub:atomic}  Recall from the end of \S \ref{sub:BL-and-Theta} that theta bases are atomic.  Let $A^{\can}\subset \wh{A}^{\can}$ be the subalgebra generated over $\Z$ by the theta functions, and let $A^{\can,+}=A^{\can}\cap \wh{A}^{\can,+}$. If the theta function multiplication is polynomial, then the theta functions form a $\Z$-basis for $A^{\can}$ (generally they may just be a topological basis).  Atomicity then says the theta functions are a minimal set of generators (under $+$) for the semiring $A^{\can,+}$.

Let $A^{\trop}_p$ be the $(\min,+)$-semiring of tropicalized functions $f^{\trop}:N_{\R}\oplus \check{\sQ}^{\gp}_{\R}\rar \R^{\min}$ for $f\in A_p^{\can}$.  Theorems \ref{thm:trop-determines-u} and \ref{thm:indep} imply the following:

\begin{cor}\label{cor:atomic}
    Assume theta function multiplication is polynomial, so $A^{\can}$ consists precisely of finite $\Z$-linear combinations of theta functions.  Then the tropical theta functions $\vartheta_{(m,q),p}^{\trop}$ are the atomic basis for the semiring $A_p^{\trop}$ in the sense that they are the unique minimal set of elements which generate $A_p^{\trop}$ under the operation $\min$.
\end{cor}

 Let $A^{\trop}_{p,N}$ denote the semiring obtained by restricting the above tropicalized functions from $N_{\R}\oplus \check{\sQ}^{\gp}_{\R}$ to $N_{\R}$ (we shall abuse notation and denote $f^{\trop}|_{N_{\R}}$ as simply $f^{\trop}$).  The minimality condition in Corollary \ref{cor:atomic} may be lost in this setting since there may not be an open cone (like $\s{N}^+$) where each restricted tropical theta function is finite (hence generically distinct by Theorem \ref{thm:trop-determines-u}).  E.g., for negative definite Looijenga pairs as in \cite{GHK1}, all tropical theta functions other than $\vartheta_{0,p}^{\trop}=1^{\trop}\equiv 0$ are equal to $-\infty$ on all of $N_{\R}\setminus \{0\}$.  Excluding such possibilities by assuming $A_{p}^{\midd}=A_p^{\can}$ (so the tropicalizations never equal $-\infty$), we may apply Theorem \ref{thm:nu} and its proof  to recover the following.

\begin{cor}
    Suppose Assumption \ref{assum:A-parallel} holds, and assume all theta functions $\vartheta_{m,p}$ for a given $p$ are polynomial, i.e., $A_{p}^{\midd}=A_p^{\can}$.  Then the restricted tropical theta functions $\vartheta_{m,p}^{\trop}$ are the atomic basis (i.e., the minimal set of generators under $\min$) for the semiring $A^{\trop}_{p,N}$.
\end{cor}

It seems natural to ask for a characterization of which functions can appear in $A^{\trop}_{p,N}$. As in \cite[Def. 4.18]{Man1}, let us define a \textbf{tropical function} $N_{\R}\rar \R\cup \{\infty\}$ to be an integral piecewise-linear function which is convex along broken lines (cf. \S \ref{sub:BL-convex}).  Here we include $0^{\trop}\equiv \infty$.  For positive Looijenga pairs, \cite[Thm. 4.20]{Man1} shows that $A^{\trop}_{p,N}$ consists precisely of the tropical functions.  We expect this holds more generally for all cluster algebras satisfying $A^{\can}=A^{\midd}$.

\begin{eg}
    In cases with $A^{\can}\neq A^{\midd}$, there may be tropical functions which are not tropicalizations of regular functions or of elements from $A^{\can}$.  For example, in the language of \cite{GHK1}, consider the case of a Looijenga pair $(\bb{P}^2,\?{D})$ for $\?{D}$ an irreducible nodal cubic.  We obtain a new Looijenga pair $(Y,D)$ by blowing up $9$ smooth points of $\?{D}$. (For another example, one may take $\?{D}$ to be a union of three generic lines in $\bb{P}^2$ and then blow up $3$ points on each component).  If these $9$ points lie on a second cubic, then there is a pencil map $f:Y\rar \bb{P}^1$ with the cubics as fibers, and $\varphi\coloneqq f^{\trop}$ is a tropical function.  However, if the $9$ points are chosen generically, then the tropical function $\varphi$ is not the tropicalization of any regular function on $Y$ nor of any finite combination of theta functions.
\end{eg}

\subsection{Upper semicontinuity for valuations}\label{sub:up}

We noted in \S \ref{sub:trop-fun} that tropicalized functions like $\vartheta_{u,p}^{\trop}$ are upper semicontinuous.  Here we prove the analogous claim for valuations.

\begin{lem}\label{lem:v-upper-semi}
    For $v\in N_{\R}\oplus \check{\sQ}^{\gp}_{\R}$ and generic $p\in M_{\R}$, the function $\val_{v,p}:M_{\Q}\oplus \sQ_{\Q} \rar \R^{\min}$, $u\mapsto \val_v(\Theta_{u,p})$ is upper semicontinuous.  Similarly for the analogous functions $\val_{n,p}:M_{\Q}\rar \R^{\min}$, $m\mapsto \val_n(\Theta_{m,p})$ with $n\in N_{\R}$ and $m\in M_{\Q}$ as in \S \ref{sub:lin-indep}.
\end{lem}
\begin{proof}
    Given a broken line $\Gamma$, let $t_1<t_2<\ldots <t_k$ be the times when $\Gamma$ bends, and let $T_{\Gamma}=T_{t_k}\circ \ldots \circ T_{t_2} \circ T_{t_1}$ be the composition of the corresponding linear transformations $T_{t_i}$ from \S \ref{sub:taut}.  Note that $T_{\Gamma}$ is an invertible linear transformation.

    Now fix $L>\val_{v,p}(u)$.  Let $\Gamma$ be a broken line with ends $(u,p)$ such that $u_{\Gamma}\cdot v<L$.  We can find some open cone $C_{L}\subset M_{\Q}\oplus \check{\sQ}^{\gp}_{\Q}$ around $u_{\Gamma}$ such that $w\cdot v< L$ for all $w\in C_{L}$.  Furthermore, if $C_{L}$ is small enough, then given any $w\in C_{L}$, we can produce a broken line $\beta$ with endpoint $p$ and final exponent $u_{\beta}=w$ which takes the same bends as $\Gamma$, so $T_{\beta}=T_{\Gamma}$.  Then the initial exponent of $\beta$ is $T_{\Gamma}^{-1}(w)$.  Thus, for any $u'\in T_{\Gamma}^{-1}(C_L)$ (which is an open cone containing $u$), there will be a broken line $\beta$ with ends $(u',p)$ such that $u_{\beta}\cdot v<L$, hence $\val_{v,p}(u')<L$.  Since $L$ and $u$ were arbitrary, the upper semicontinuity follows.

    The argument for $n\in N_{\R}$ and $m\in M_{\Q}$ is the same, mutatis mutandis.
\end{proof}

\section{Seed data and linear morphisms}\label{sec:cluster}

For simplicity, throughout this section and the next, we restrict to $\vartheta_m$ and $\val_{n,p}$ with $m\in M$ and $n\in N_{\R}$ as defined at the start of \S \ref{sub:lin-indep}.  The arguments extend to the setting using the full lattices $M\oplus \sQ^{\gp}$ and $N\oplus \check{\sQ}^{\gp}$, but this extension can also be obtained from our restricted setting using a principal coefficients setup.

We restrict now to scattering diagrams of the form $\f{D}=\Scat(\f{D}_{\In})$ (as in Theorem \ref{thm:ScatD}) for 
\begin{align}\label{eq:cluster-Din}
    \f{D}_{\In} = \{(n_i^{\perp},f_i=1+x^{m_i}y^{e_i},n_i)|i\in I\}
\end{align} where $I$ is a finite index-set, $n_i\in N$, $m_i \in n_i^{\perp} \subset M$, and $\sQ=\N^I$.  Furthermore, we will require that the matrix $\left(m_j \cdot n_i\right)_{ij}$ is skew-symmetrizable (see \S \ref{sec:sd}).

The vectors $m_i$ and $n_i$ determining $\f{D}_{\In}$ may be presented differently in the literature depending on the context.  E.g., when constructing theta functions on a cluster $\s{A}$-variety as in \cite{GHKK}, the vectors $n_i$ form part of a basis---they're the seed vectors of \cite{FG1}.  \cite{GHKK} also constructs theta functions for cluster Poisson varieties $\s{X}$ by taking a certain subalgebra of $\s{A}^{\prin}$, but one can instead construct a scattering diagram for $\s{X}$ directly with the $m_i$'s being the seed vectors; cf. \cite[\S 4.3]{DM}.  Similarly, one may construct scattering diagrams for certain torus-quotients of $\s{A}$ and fibers or subfamilies of $\s{X}$ as considered in \cite{CMN}.  The initial scattering diagrams for all these cases have the form of \eqref{eq:cluster-Din}.  We introduce a new framework, called a seed datum, which allows us to conveniently address all these cases at once.

\subsection{Seed datum}\label{sec:sd}
Let $\lrc{e_1,\dots, e_r}$ be the standard basis for the rank $r$ lattice $\Z^r$.
In this paper, a skew-symmetric \textbf{seed datum} of rank $r$ consists of a pair of group homomorphisms
\[ P:\mathbb{Z}^r\longrightarrow M \text{ and }Q:\mathbb{Z}^r\longrightarrow N, \]
where $M$ and $N$ are a pair of mutually dual finite-rank lattices, such that the \emph{exchange matrix} $B$ defined by
$ B_{i,j} \coloneqq  (Pe_j)\cdot (Qe_i) $ is skew-symmetric.  We additionally assume that each $Qe_i$ is nonzero.

More generally, we consider \textbf{skew-symmetrizable} seed data.  Here we additionally have:
\begin{itemize}
\item a lattice $N^\bullet$ such that both $N^\bullet$ and $N$ are finite-index sublattices of a common refinement,
\item a homomorphism $Q^{\bullet}:\Z^r \to N^{\bullet}$,
\item and a diagonal matrix $D = \diag(d_1,\dots,d_r)$ with each $d_i \in \Q_{>0}$ such that $Q= Q^{\bullet}\circ D$.\footnote{We define $Q^{\bullet}\circ D$ on $\Z^r$ by extending $Q^{\bullet}$ linearly from $\Z^r$ to $\Q^r\supset D\Z^r$.  Similarly for the other pre-compositions with $D$ and $D^{-1}$ which we shall consider.}
\end{itemize}
Rather than requiring $B$ to be skew-symmetric, we require that $B^{\bullet}_{i,j}\coloneqq  (P e_j)\cdot (Q^\bullet e_i)$ is skew-symmetric. 
The exchange matrix $B= D B^\bullet$ is said to be {\it{skew-symmetrizable}}.
We also consider the dual lattice $M^\bullet$ to $N^\bullet$.  We may view $M^\bullet$ and $M$ as finite-index sublattices of a common refinement.    In addition to $Q^{\bullet}=Q\circ D^{-1}:\Z^r\rar N^{\bullet}$, we shall consider $P^{\bullet}\coloneqq P\circ D:\Z^r\rightarrow M^{\bullet}$.

A seed datum as above may be denoted $(P,Q)$.

The vectors $Pe_i$ and $Q e_i$ here correspond to the elements $m_i\in M$ and $n_i\in N$ from \eqref{eq:cluster-Din}, respectively.  That is, we define $\sQ\coloneqq \N^r\subset \Z^r$, an initial scattering diagram
 \begin{align}\label{eq:DPQin}
     \f{D}^{(P,Q)}_{\In}\coloneqq\lrc{\left.\lrp{(Q e_i)^{\perp},1+x^{Pe_i}y^{e_i},Qe_i}\, \right| \, i=1,\ldots,r},
     \end{align} 
and an associated consistent scattering diagram $\f{D}^{(P,Q)}\coloneqq\Scat(\f{D}^{(P,Q)}_{\In})$ as in Theorem \ref{thm:ScatD}.  So, e.g., the elementary transformation associated to the initial wall indexed by $i$ is
     \begin{align*}
         E_{Qe_i,1+x^{Pe_i}y^{e_i}}(x^my^q)=x^my^q(1+x^{Pe_i}y^{e_i})^{m\cdot Qe_i}.
     \end{align*}

\begin{rem}\label{rem:scale}
    Rescaling $D$ and $P^{\bullet}$ by a constant $k\in \Q_{>0}$ while rescaling $Q^{\bullet}$ by $k^{-1}$ will not affect $(P,Q)$, hence will not affect $\f{D}^{(P,Q)}$.  One may therefore always rescale to assume that $D$ is $\Z_{>0}$-valued.
\end{rem}

It is often convenient to fix an identification $M\simeq \mathbb{Z}^d$; this determines an identification $N\coloneqq \mathbb{Z}^d$ under which the pairing between $M$ and $N$ becomes the dot product. Given such identifications, $P$ and $Q$ may be expressed as integral $d\times r$-matrices such that ${B}^{\bullet}=(Q^{\bullet})^{\top} P$ and $B=Q^\top P=DB^{\bullet}$.  Our notation will typically be based on this matrix viewpoint.

\begin{exam}\label{eg:clusterA}
Consider an integral skew-symmetrizable $r\times r$-matrix $B=DB^{\bullet}$ with $M$ and $N$ identified with $\Z^d$ as above.  If $d=r$, we may consider the seed datum $(B,\Id)$.  Then the theta functions associated to $\f{D}^{(B,\Id)}$  include the cluster variables in the \emph{coefficient-free cluster algebra $\s{A}$ with exchange matrix} $B$.  Theta functions on the corresponding cluster Poisson variety $\s{X}$ are constructed using the seed datum $(\Id, B^{\top})$.  

More generally, if $d\geq r$ and $\widetilde{B}=\begin{bmatrix} B \\ C\end{bmatrix}$ is a $d\times r$-matrix whose top $r\times r$ submatrix is $B$, then the seed datum 
\[\left(\widetilde{B}, \begin{bmatrix} \Id_{r\times r} \\ 0_{(d-r)\times r}\end{bmatrix}\right)\]
defines a scattering diagram and theta functions corresponding to the \emph{cluster algebra with extended exchange matrix} $\wt{B}$ and principal exchange matrix $[\Id_{r\times r} ~ 0_{r\times (d-r)}]\widetilde{B}=B$. The associated scattering diagrams are called \emph{$\s{A}$-type} in \cite{GHKK}.  A corresponding $\s{X}$-type scattering diagram is given by 
\[\left(\begin{bmatrix} \Id_{r\times r} \\ 0_{(d-r)\times r}\end{bmatrix}, \begin{bmatrix} B^{\top} \\ C'\end{bmatrix}\right)\]
for any integral  $(d-r)\times r$-matrix $C'$.
\end{exam}

An important special feature of these scattering diagrams $\f{D}^{(P,Q)}$ is the following:
\begin{lem}\label{lem:D-skew} 
    Up to equivalence, every wall $(W,f,n)\in \f{D}^{(P,Q)}$ satisfies $$f\in 1+x^{Pv}y^v\Z\llb x^{Pv}y^v\rrb$$ for some $v\in \N^r$ such that $n\in N$ is a positive rational multiple of $Q^{\bullet}v\neq 0$.
\end{lem}

In particular, Assumption \ref{assum:A-parallel} holds.

\begin{proof} 
First note that the initial scattering diagram \eqref{eq:DPQin} has the desired form since $Qe_i=d_iQ^{\bullet}e_i$.  Given this, the claim follows from standard constructions as in \cite{GPS,WCS,GHKK,Man3}.\footnote{Indeed, Lemma \ref{lem:D-skew} is essentially already known in the cluster setting, cf. \cite[Equation 1.1 and the surrounding discussion]{GHKK}.  Here we just reframe this fact in the seed datum perspective.}  More precisely, one considers the module of log derivations $\wh{\f{g}}=\QQ[x^M]\llb y^{\sQ}\rrb \otimes N$ with Lie bracket $$[x^{m_1}\otimes \partial_{n_1},x^{m_2}\otimes \partial_{n_2}]=x^{m_1+m_2}\otimes \partial_{(m_2\cdot n_1)n_2-(m_1\cdot n_2)n_1}$$  Then the elementary transformations $E_{n,f}$ are understood as actions of $\exp (\log f\otimes \partial_n)$ for $\log f\otimes \partial_n\in \wh{\f{g}}$.  The claim then is equivalent to saying that for each wall, the associated element in $\wh{\f{g}}$ lies in $x^{Pv}y^v\Q\llb x^{Pv}y^v\rrb \otimes \partial_{Q^{\bullet}v}$ for some $v$.  Elements of this form are closed under the Lie bracket:
\begin{align*}
    \left[x^{Pv_1}y^{v_1}\otimes \partial_{Q^{\bullet}v_1},x^{Pv_2}y^{v_2}\otimes \partial_{Q^{\bullet}v_2}\right] &=  x^{P(v_1+v_2)}y^{v_1+v_2}\otimes \partial_{(Pv_2\cdot Q^{\bullet}v_1)Q^{\bullet}v_2-(Pv_1\cdot Q^{\bullet}v_2)Q^{\bullet}v_1}\\
    &=B^{\bullet}(v_1,v_2)\left(x^{P(v_1+v_2)}y^{v_1+v_2}\otimes \partial_{Q^{\bullet}(v_1+v_2)}\right)
\end{align*}
where in the last line we used that $B^{\bullet}=(Q^{\bullet})^{\top}P$ is skew-symmetric, hence $-Pv_1\cdot Q^{\bullet}v_2=Pv_2\cdot Q^{\bullet}v_1=B^{\bullet}(v_1,v_2)$.  
So elements of $\wh{\f{g}}$ with the desired form are a Lie subalgebra of $\wh{\f{g}}$.  Since we saw that the initial walls come from this Lie subalgebra, it follows that all the walls of $\f{D}^{(P,Q)}$ come from this Lie subalgebra as well.  This proves the claim---the condition $Q^{\bullet} v\neq 0$ may be assumed to hold because otherwise the wall would act trivially.
\end{proof}

\begin{exam}
Let $M\coloneqq  \mathbb{Z}^2=:N$, and consider the seed
\[ P = \begin{bmatrix} 0 & -1 \\ 1 & 0 \end{bmatrix} \text{ and }
Q =\begin{bmatrix} 1 & 0 \\ 0 & 1 \end{bmatrix} \]
This seed defines two incoming walls in $M_\mathbb{R}\coloneqq \mathbb{R}^2$ which are supported on the axes. Consistent completion introduces one additional wall $(\R_{\leq 0}(-1,1),1+x^{(-1,1)}y^{(1,1)},(1,1))$.  The resulting scattering diagram $\f{D}^{(P,Q)}$ is depicted below.
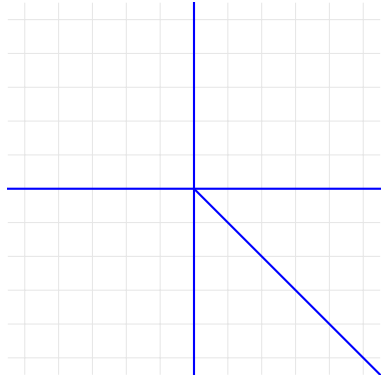
\[ 
\begin{tikzpicture}
\begin{scope}[scale=.45]
	\clip (-5.5,-5.5) rectangle (5.5,5.5);
    	\draw[step=1,draw=black!10,very thin] (-5.5,-5.5) grid (5.5,5.5);
	\draw[blue, thick] (0,-6) to (0,6);
	\draw[blue, thick] (-6,0) to (6,0);
	\draw[blue, thick] (0,0) to (6,-6);
\end{scope}
\end{tikzpicture}\]
\end{exam}

\subsection{Linear morphisms}\label{sub:lin-morph}

If $(P,Q)$ and $(P',Q')$ are seed data of rank $r$ with lattices $M$ and $M'$, respectively, an \textbf{(integral) linear morphism} from $(P,Q)$ to $(P',Q')$ is a linear map $A:M\rightarrow M'$ such that $AP=P'$ and $A^\top Q'=Q$, where $A^\top:N'\rightarrow N$ is the adjoint map.  Equivalently, $A$ makes the following diagram commute.
\[\begin{tikzpicture}[scale=1.5]
    \node (a) at (-1.414,0) {$\mathbb{Z}^r$};
    \node (b) at (1.414,0) {$\mathbb{Z}^r$};
    \node (c) at (0,.5) {$M$};
    \node (d) at (0,-.5) {$M'$};
    \draw[->] (a) to node[above] {$P$} (c);
    \draw[->] (a) to node[below] {$P'$} (d);
    \draw[->] (c) to node[above] {$Q^\top$} (b);
    \draw[->] (d) to node[below] {$Q'^\top$} (b);
    \draw[->] (c) to node[left] {$A$} (d);
\end{tikzpicture}\]
Note that such a map can only exist if $Q^\top P= Q'^\top P'$; that is, if $(P,Q)$ and $(P',Q')$ have the same exchange matrix $B=B'$.  It follows that the two seed data have the same diagonal matrix $D$ up to some overall re-scaling\footnote{If any rows/columns of $B$ are identically $0$, there may be additional freedom to alter the corresponding entries of $D$.  We still assume for convenience that $D$ and $D'$ are equal whenever we consider a linear morphism.} as in Remark \ref{rem:scale}.  For convenience, we will always assume that we have scaled to ensure $D=D'$ when we consider a linear morphism $(P,Q)\rar (P',Q')$.  In particular, the equalities $AP=P'$ and $A^{\top}Q'=Q$ are equivalent to $AP^{\bullet}=(P')^{\bullet}$ and $A^{\top} (Q')^{\bullet} =Q^{\bullet}$, respectively.

More generally, a \textbf{rational linear morphism} is defined the same way but with $A:M_{\Q}\rar M'_{\Q}$ and $A^\top:N'_{\Q}\rar N_{\Q}$.

\begin{exam}\label{eg:ensemble}
Consider a seed datum $(B,\Id)$ as in Example \ref{eg:clusterA} with corresponding $\s{X}$-type seed datum $(\Id,B^{\top})$.  Then there is a linear morphism
\[ B: (\mathrm{Id},B^{\top}) \rightarrow (B,\Id). \]
The induced map on cluster varieties is the \emph{cluster ensemble map} of \cite{FG1}.  In fact, if $B=Q^{\top} P$ for some seed datum $(P,Q)$, we can factor the above linear morphism $B$ as
\begin{align}\label{eq:ensemble}
    (\Id,B^{\top})\stackrel{P}{\longrightarrow} (P,Q)\stackrel{Q^{\top}}{\longrightarrow} (B,\Id).
\end{align}

Similarly, for the setup from the second part of Example \ref{eg:clusterA}, a cluster ensemble map $$\left(\begin{bmatrix} \Id_{r\times r} \\ 0_{(d-r)\times r}\end{bmatrix}, \begin{bmatrix} B^{\top} \\ C'\end{bmatrix}\right)\rar \left(\widetilde{B}, \begin{bmatrix} \Id_{r\times r} \\ 0_{(d-r)\times r}\end{bmatrix}\right)$$
is given by $A=\begin{bmatrix} B & (C')^{\top} \\ C & E\end{bmatrix}$ for any integral $(d-r)\times (d-r)$-matrix $E$.  If $C$, $C'$, and $E$ are such that there exist matrices $F$ and $G$ satisfying $Q^{\top}G=(C')^{\top}$, $FP=C$, and $FG=E$, then we can factor $A$ as 
\begin{align}\label{eq:ext-ebsemble}
    \left(\begin{bmatrix} \Id_{r\times r} \\ 0_{(d-r)\times r}\end{bmatrix}, \begin{bmatrix} B^{\top} \\ C'\end{bmatrix}\right)\stackrel{\begin{bmatrix} P & G\end{bmatrix}}{\longrightarrow} (P,Q)\stackrel{\begin{bmatrix} Q^{\top} \\ F\end{bmatrix}}{\longrightarrow}  \left(\widetilde{B}, \begin{bmatrix} \Id_{r\times r} \\ 0_{(d-r)\times r}\end{bmatrix}\right).
\end{align}
\end{exam}

Recall that a sublattice $M'\subset M$ is saturated if $M/M'$ is torsion-free; equivalently, $M = M' \oplus M''$ for some sublattice $M''\subset M$. A linear map of lattices $A:M'\rightarrow M$ is called \textbf{saturated} if its image is a saturated sublattice.  We say that a seed datum $(P,Q)$ is \textbf{saturated-injective} if $P$ and $Q$ are saturated injections. 

\begin{lem}\label{lemma: saturatedresolution}
Given any seed $(P,Q)$, there are linear morphisms of seeds 
\[ (P,Q) \stackrel{A}{\hookrightarrow} (P',Q') \stackrel{B}{\twoheadleftarrow} (P'',Q'') \]
such that $A$ is injective, $B$ is surjective, and $(P'',Q'')$ is saturated-injective.
\end{lem}
\begin{proof}
Let $M'=M\oplus \Z^r$ and $M'' = M\oplus \Z^r \oplus \Z^r$.
Now define the following lattice maps:

\noindent
\begin{minipage}{.24\textwidth}
\eqn{P':\Z^r &\to M'\\
    v&\mapsto (Pv,0), }    
\end{minipage}
\begin{minipage}{.24\textwidth}
\eqn{P'':\Z^r &\to M''\\
    v&\mapsto (Pv,0,v), }
\end{minipage}
\begin{minipage}{.24\textwidth}
\eqn{Q':\Z^r &\to N'\\
    v&\mapsto (Qv,v), }
\end{minipage}
\begin{minipage}{.24\textwidth}
\eqn{Q'':\Z^r &\to N''\\
    v&\mapsto (Qv,v,0), }
\end{minipage}

Our candidates for the linear morphisms are

\noindent
\begin{minipage}{.45\textwidth}
\eqn{A:M &\to M'\\
    m&\mapsto (m,0), }
\end{minipage}
and
\begin{minipage}{.45\textwidth}
\eqn{B:M'' &\to M'\\
    (m,a,b)&\mapsto (m,a). }
\end{minipage}

We compute:
\eqn{m\cdot A^\top (n,q) &= A m \cdot (n,q) \\
&=(m,0) \cdot (n,q)\\
&= m\cdot n,}
so
\eqn{A^\top : N' &\to N\\
    (n,q) &\mapsto n.
}
Similarly,
\eqn{(m,a,b) \cdot B^\top (n,q) &= B (m,a,b) \cdot (n,q)\\
&= (m,a)\cdot(n,q)\\
&=(m,a,b)\cdot (n,q,0),}
so 
\eqn{B^\top : N' &\to N''\\
    (n,q) &\mapsto (n,q,0).
}
Then $APv=(Pv,0)=P'v$ and $A^\top Q'v=Qv$, so $A$ is indeed a linear morphism.
Meanwhile,
$BP''v= (Pv,0) = P'v$ and $B^\top Q'v = (Qv,v,0)=Q''v$, so $B$ is a linear morphism as well.

Finally, it is easy to see that $A$ is injective, $B$ is surjective, and $(P'',Q'')$ is saturated-injective. 
\end{proof}

Given a linear morphism $A:(P,Q)\rightarrow (P',Q')$, define a homomorphism 
$\rho_A:\mathbb{Z}[x^M]\llb y^{\N^r}\rrb \rightarrow \mathbb{Z}[x^{M'}]\llb y^{\N^r}\rrb$ by
\[ \rho_A(x^my^q) \coloneqq  x^{Am}y^q. \]
When it is not clear from context, we will write theta functions with the superscript $(P,Q)$ to indicate that they are associated to the scattering diagram $\f{D}^{(P,Q)}$.  The following lemmas will be used to prove Theorem \ref{thm: linearmorphism}, which relates theta functions associated with $(P,Q)$ to those associated with $(P',Q')$.

\begin{lem}\label{lem:ImQperp}
    All walls of $\f{D}^{(P,Q)}$ are invariant under  translations by $\im(Q)^\perp$.
\end{lem}
\begin{proof}
    Observe that this holds for $\f{D}^{(P,Q)}_{\In}$.  Consider the recursive construction of the consistent completion $\f{D}^{(P,Q)}$ as in \cite[\S C.1]{GHKK}.  In each step, the supports of new walls have the form $\f{j}-\R_{\geq 0} m$ for some joint $\f{j}$ of $(\f{D}^{(P,Q)})^k$ and some vector $m\in M$.  The claim follows from induction on $k$: the inductive assumption ensures that each joint $\f{j}$ of $(\f{D}^{(P,Q)})^{k}$ is invariant under translations by $\im(Q)^{\perp}$, so $(\f{D}^{(P,Q)})^{k+1}$ will be as well.
\end{proof}

\begin{lem}\label{lem:LinMorResInj}
Let $A:(P,Q) \to (P',Q')$ be a linear morphism, and fix a splitting of the domain $M_{\R}$ of $A$ (tensored with $\R$) as $M_{\R}\cong \im(Q)^* \oplus \im(Q)^\perp$.
Then 
\begin{enumerate}
    \item $A(\im(Q)^{\perp})\subset \im(Q')^{\perp}$, and
    \item the composition $\left.A\right|_{\im(Q)^*}:\im(Q)^*\rightarrow M'_{\R}\rightarrow M'_{\R}/\im(Q')^{\perp}$ is injective.
\end{enumerate}
    Thus, there is a splitting $M'_{\R}\cong \im(Q')^*\oplus \im(Q')^{\perp}$ such that $A$ splits as $A_1\oplus A_2$ for $$A_1=A|_{\im(Q)^*}:\im(Q)^*\hookrightarrow M'_{\R}/\im(Q')^{\perp}\cong \im(Q')^*$$ injective and $$A_2=A|_{\im(Q)^{\perp}}:\im(Q)^\perp\rightarrow \im(Q')^{\perp}.$$
\end{lem}

\begin{proof}
    For (1), let $m\in \im(Q)^{\perp}$.  Then for any $u\in \sQ=\N^r$, we have
    \begin{align*}
        Am\cdot Q'u=m\cdot A^{\top} Q'u = m\cdot Qu =0
    \end{align*}
    as desired.

    Now for (2), $m\in \im(Q)^*$ being nonzero implies that there is some $u\in \sQ$ such that
    \begin{align*}
        0\neq m\cdot Qu = m\cdot A^{\top} Q'u = Am\cdot Q'u.
    \end{align*}
    So then $Am$ is nonzero in $M_{\R}/\im(Q')^{\perp}\cong \im(Q')^*$, as desired.

    The final claim follows from (1) and (2).
\end{proof}

Note that since $Q=Q^{\bullet} D$ for $D$ diagonal and nondegenerate, we have $\im(Q)=\im(Q^{\bullet})$ over $\R$.  Similarly for $\im(Q')=\im((Q')^{\bullet})$ over $\R$.

\begin{thm}\label{thm: linearmorphism}
Given a linear morphism of seeds $A:(P,Q)\rightarrow (P',Q')$, we have
\[ \rho_A(\vartheta^{(P,Q)}_{m,p}) = \vartheta^{(P',Q')}_{Am,Ap} \]
for all $m\in M$ and generic\footnote{The endpoint $Ap$ is not generic in $M'_{\R}$ since it is contained in the subspace $AM_{\R}$.  So to avoid the possibility of broken lines passing through joints of $\f{D}^{(P',Q')}$, we take $\vartheta^{(P',Q')}_{Am,Ap}$ to mean $\lim_{\epsilon \rar 0^+} \vartheta^{(P',Q')}_{Am,Ap+\epsilon \mu}$ for generic $\mu \in M'_{\R}$.  The choice of $\mu$ does not matter because if $Ap+\epsilon \mu$ and $Ap+\epsilon \mu'$ lie on opposite sides of a wall $(W',f',n')\in \f{D}^{(P',Q')}$ for generic $p$ and arbitrarily small $\epsilon>0$, then we must have $AM_{\R}\subset (n')^{\perp}$, so this wall acts trivially on broken lines parallel to $AM_{\R}$ (and we'll see that all broken lines contributing to any $\vartheta_{Am,p'}^{(P',Q')}$ are indeed parallel to $AM_{\R}$).\label{foot:generic}} $p\in M_{\R}$.
\end{thm}

\begin{proof}
We start by defining the restriction $\f{D}^{(P',Q')}|_{AM_{\R}}$ of $\f{D}^{(P',Q')}$ to $AM_{\R}\subset M'_{\R}$.  First, note that each scattering function $f'\in \Z[x^{M'}]\llb y^{\sQ}\rrb$ in fact lies in the subring $\Z[x^{AM}]\llb y^{\sQ}\rrb$ because the exponents have the form $P'v=APv$ for $v\in \sQ$.  In particular, this means that broken lines with initial exponent in $AM_{\R}$ remain parallel to $AM_{\R}$.

Now, we would like to say that for each wall $(W',f',n')\in \f{D}^{(P',Q')}$, the restriction $\f{D}^{(P',Q')}|_{AM_{\R}}$ will contain a wall $(W'\cap AM_{\R},f',n'|_{AM_{\R}})$.   However, if a joint $\f{j}$ of $\f{D}^{(P',Q')}$ intersects $AM_{\R}$ non-transversely, we must be more careful.  In this case, given a small generic path $\gamma$ in $AM_{\R}$ through $\f{j}\cap AM_{\R}$, we choose a small generic perturbation $\gamma'$ in $M'_{\R}$ with the same endpoints as $\gamma$, and we assign the scattering function along $\f{j}$ so that $E_{\gamma}=E_{\gamma'}$.  One similarly defines $(\f{D}^{(P',Q')}_{\In})|_{AM_{\R}}$. 

Now, by construction, $\vartheta_{Am,Ap}^{\f{D}^{(P',Q')}|_{AM_{\R}}} = \vartheta_{Am,Ap}^{(P',Q')}.$  Our goal is to see that the broken lines $\Gamma'$ for $\f{D}^{(P',Q')}|_{AM_{\R}}$ with $\Ends(\Gamma')= (Am,Ap)$ are precisely obtained by applying $A$ to the supports of broken lines $\Gamma$ for $\f{D}^{(P,Q)}$ with $\Ends(\Gamma)= (m,p)$ while applying $\rho_A$ to the monomials decorating each linear segment of $\Gamma$.  We write $\rho_A(\Gamma)$ to denote the broken line $\Gamma'$ obtained from $\Gamma$ in this way.

Next, given $\f{D}^{(P,Q)}$, we define a scattering diagram $\rho_A(\f{D}^{(P,Q)})$ in $AM_{\R}$ as follows.  Consider a wall $(W,f,n)\in \f{D}^{(P,Q)}$.  By Lemma \ref{lem:D-skew}, $n=kQ^{\bullet}v\in N$ for some $k\in \Q_{>0}$ and $v\in \sQ=\N^r$ with $f\in 1+ x^{Pv}y^v \Z\llb x^{Pv}y^v\rrb$.  Let us view $f$ as $f(x^{Pv}y^v)$ for $f(t)\in 1+t\Z\llb t\rrb$, so $\rho_A(f)=f(x^{APv}y^v)=f(x^{P'v}y^v)$.  Then $\rho_A(\f{D}^{(P,Q)})$ will contain a corresponding wall $(AW,\rho_A(f),k(Q')^{\bullet}v)$.\footnote{The lattice of $x$-exponents for $\rho_A(\f{D}^{(P,Q)})$ is $AM\subset M'$, so the dual vectors should lie in $(AM)^*$.  Indeed, for a wall $(AW,\rho_A(f),k(Q')^{\bullet}v)$ as in the construction above and for any $m\in M$, we have $Am\cdot k(Q')^{\bullet}v=m\cdot kQ^{\bullet} v$, and this is integral since $kQ^{\bullet} v=n\in N$.}

Let us check compatibility of the elementary transformations associated to wall-crossing. 
The elementary transformation associated to $(W,f,n)$ as above is 
\eqn{E_{kQ^{\bullet}v,f(x^{Pv}y^v)}(x^{m}y^q) = x^{m}y^q f(x^{Pv}y^v)^{m \cdot kQ^{\bullet} v}.}
So $\rho_A\circ E_{kQ^{\bullet}v,f(x^{Pv}y^v)}$ is given by
\eqn{x^{m}y^q \mapsto x^{A m}y^q f(x^{APv}y^v)^{m \cdot kQ^{\bullet} v}&=
x^{A m}y^q f(x^{APv}y^v)^{m \cdot A^{\top} k(Q')^{\bullet} v}\\
&=x^{A m}y^q f(x^{P'v}y^v)^{A m \cdot k(Q')^{\bullet} v}.}
We now see that the following diagram commutes:
\eq{\begin{tikzpicture}[scale=1.5]
    \node (tl) at (-2,1) {$x^{m}y^{q}$};
    \node (tr) at (2,1) {$x^{m}y^q f(x^{Pv}y^v)^{m \cdot Q^{\bullet} v}$};
    \node (bl) at (-2,-1) {$x^{Am}y^{q}$};
    \node (br) at (2,-1) {$x^{Am}y^q f(x^{P'v}y^v)^{Am \cdot (Q')^{\bullet} v}$};
    \draw[|->] (tl) to node[above] {$E_{kQ^{\bullet}v,f(x^{Pv}y^v)}$} (tr);
    \draw[|->] (bl) to node[below] {$E_{k(Q')^{\bullet}v,f(x^{P'v}y^v)}$} (br);
    \draw[|->] (tl) to node[left] {$\rho_A$} (bl);
    \draw[|->] (tr) to node[right] {$\rho_A$} (br);
\end{tikzpicture}}{eq:rhoA}

We claim that \begin{align}\label{eq:AD}
\rho_A(\f{D}^{(P,Q)}) = \f{D}^{(P',Q')}|_{AM_{\R}}    
\end{align}
(up to equivalence).  It is straightforward to check this for the initial scattering diagrams:
$$\rho_A\lrp{(\f{D}^{(P,Q)})_{\mathrm{in}}} = \left.\lrp{\f{D}^{(P',Q')}}_{\mathrm{in}}\right|_{AM_{\R}}.$$
We see from \eqref{eq:rhoA} that consistency of $\f{D}^{(P,Q)}$ implies the consistency of $\rho_A(\f{D}^{(P,Q)})$.  Moreover, the consistency of $\f{D}^{(P',Q')}|_{AM_{\R}}$ is immediate from the consistency of $\f{D}^{(P',Q')}$.  It is also clear that outgoing walls of $\f{D}^{(P',Q')}$ produce outgoing walls of $\f{D}^{(P',Q')}|_{AM_{\R}}$. So by the uniqueness in Theorem \ref{thm:ScatD}, it now suffices to show that $\rho_A$ takes outgoing walls of $\f{D}^{(P,Q)}$ to outgoing walls of $\rho_A(\f{D}^{(P,Q)})$.

Let $(W,f(x^{Pv}y^q), kQ^{\bullet}v)\in \f{D}^{(P,Q)}$ be outgoing, so $Pv \notin W$.  We wish to show that $P'v\notin AW$ so that $(AW,f(x^{P'v}y^q),k(Q')^{\bullet}v)$ will be outgoing.
To see this, split $M_{\Q}$ as $\im(Q^{\bullet})^*\oplus \im(Q^{\bullet})^\perp$ as in Lemma~\ref{lem:LinMorResInj}.
Now use this splitting to write $Pv= (Pv)_1 +(Pv)_2$ and $W= W_1 + W_2$.
By Lemma~\ref{lem:ImQperp}, $W_2 = \im(Q^{\bullet})^\perp$ so failure of $Pv$ to lie in $W$ must come from failure of $(Pv)_1$ to lie in $W_1$.
Next, by Lemma~\ref{lem:LinMorResInj}(2), $(Pv)_1\notin W_1$ implies $A(Pv)_1\notin A(W_1)$, hence $APv\notin AW$ by the final claim of Lemma \ref{lem:LinMorResInj}.  Since $AP=P'$, the claim $P'v\notin AW$ follows.

We have thus shown that $\rho_A(\f{D}^{(P,Q)}) = \f{D}^{(P',Q')}|_{AM_{\R}}$.  By \eqref{eq:rhoA}, we see that broken lines $\Gamma$ with ends $(m,p)$ for $\f{D}^{(P,Q)}$ correspond to the broken lines $\Gamma'=\rho_A(\Gamma)$ with ends $(Am,Ap)$ for $\rho_A(\f{D}^{(P,Q)}) = \f{D}^{(P',Q')}|_{AM_{\R}}$.  The desired equality $\rho_A(\vartheta^{(P,Q)}_{m,p}) = \vartheta_{Am,Ap}^{\f{D}^{(P',Q')}|_{AM_{\R}}} = \vartheta_{Am,Ap}^{(P',Q')}$ follows. 
\end{proof}

Recall the sets $\Theta_{u,p}$ consisting of final exponents for rational broken lines as in \eqref{eq:ThetaQ} and the limiting sets $\lim_{\epsilon \rar 0^+} \Theta_{u,p+\epsilon \mu}$ as in \eqref{eq:limT}.  Theorem \ref{thm: linearmorphism} immediately implies the following:

\begin{cor}\label{thm:rho-rational-Theta}
    Given a rational linear morphism of seeds $A:(P,Q)\rightarrow (P',Q')$, we have
\[ \rho_A(\Theta^{(P,Q)}_{m,p}) = \Theta^{(P',Q')}_{Am,Ap} \]
for all $m\in M_{\Q}$ and generic $p\in M_{\R}$.
\end{cor}

It will also be helpful to recall that following important fact:

\begin{thm}[Positivity for cluster scattering diagrams \cite{GHKK}]\label{thm:pos-scattering}
    The scattering diagram $\f{D}^{(P,Q)}$ associated to a skew-symmetrizable seed datum $(P,Q)$ is positive.
\end{thm}
\begin{proof}
    This is essentially \cite[Thm. 1.13]{GHKK}.  Our skew-symmetrizable seed data are a bit more general than those of loc.~ cit., but we can always reduce to the setting of loc.~ cit.~ using linear morphisms (e.g., as in Lemma \ref{lemma: saturatedresolution}) and the relationship \eqref{eq:AD} between scattering diagrams connected by linear morphisms.
\end{proof}

\subsection{Theta function extension for cluster algebras}\label{sub:theta-extension-cluster}

For $i=1,2$, let $(P_i,Q_i)$ be seed data with the same target spaces $M$ and $N$, and let $\sQ_i^{\gp}$ be the domain of $P_i$ and $Q_i$.  We say that $(P_2,Q_2)$ is an extension of $(P_1,Q_1)$ if $\sQ_1$ is a face of $\sQ_2$ (i.e., $\sQ_2\supset \sQ_1$ is an extension in the sense of \S \ref{sub:Theta-ext}), $\sQ_1\subset \sQ_2$ is saturated, $P_2|_{\sQ_1}=P_1$, and $Q_2|_{\sQ_1}=Q_1$.  For example, in the usual language of cluster algebras, this situation arises when one ``unfreezes'' a frozen index.  The theta function extension theorem (Theorem \ref{thm:theta-extension-general}) has the following corollary in the present setting.

\begin{thm}[Theta function extension theorem, seed datum version]\label{thm:theta-ext-seed}
    Let $(P_1,Q_1)$ and $(P_2,Q_2)$ be skew-symmetrizable seed data such that $(P_2,Q_2)$ is an extension of $(P_1,Q_1)$.  If $\vartheta_{m,p}^{(P_1,Q_1)}$ is a finite linear combination of theta functions $\vartheta_{m_i,p}^{(P_2,Q_2)}$, then $\vartheta_{m,p}^{(P_1,Q_1)}=\vartheta_{m,p}^{(P_2,Q_2)}$.
\end{thm}
\begin{proof}
Given Theorem \ref{thm:theta-extension-general} and the positivity from Theorem \ref{thm:pos-scattering}, we just have to check the condition that, for every monomial $cx^my^q\neq 1$ of every scattering function of $\f{D}_2\setminus \f{D}_1$, we have $q\notin \sQ_{1,\R}$.  This follows from considering $\f{D}^{(P_2,Q_2)}$ modulo the ideal $\langle y^q|q\in \sQ_2\setminus \sQ_1\rangle$.  Indeed, the initial scattering diagrams agree modulo this ideal, so the uniqueness from Theorem \ref{thm:ScatD} ensures that the consistent completions $\f{D}^{(P_2,Q_2)}$ and $\f{D}^{(P_1,Q_1)}$ will be equivalent modulo this ideal.  The claim follows.
\end{proof}

\subsection{The positive chamber}\label{sub:chambers}

Recall that a chamber in a scattering diagram is a connected component of the complement of its support. By Lemma \ref{lem:CPS}, moving a basepoint to a different point in the same chamber does not change the local expressions of the theta functions, so we can unambiguously refer to the basepoint being in a certain chamber without specifying the basepoint precisely.  Corollary \ref{cor:Q-CPS} lets us extend this to $\Theta_m$ for any $m\in M_{\Q}$.

By Lemma~\ref{lem:D-skew}, every wall of $\f{D}^{(P,Q)}$ has support in $(Q^{\bullet}v)^{\perp}$ for some $v\in \N^r$.  The \textbf{positive side} of the wall is then defined to be the half-space where $Q^{\bullet}v$ is positive.  The \textbf{positive chamber} $C^+$ of $\f{D}^{(P,Q)}$ is the unique chamber (if it exists) on the positive side of every wall (equivalently, on the positive side of every initial wall). Let $\vartheta^{(P,Q)}_{m,+}$ denote a theta function with a generic basepoint in the positive chamber. 

If $\f{D}^{(P,Q)}$ does not have a positive chamber, there are two methods to define $\vartheta^{(P,Q)}_{m,+}$ which give the same result. First, one may pick a linear inclusion of seed data $A:(P,Q)\rightarrow (P',Q')$ for which $Q'$ is an inclusion. Such an inclusion always exists (cf. the map $A$ from the proof of Lemma \ref{lemma: saturatedresolution}), and $\f{D}^{(P',Q')}$ always has a positive chamber. Then define $\vartheta^{(P,Q)}_{m,+}$ to be the unique element of $\mathbb{Z}[x^M]\llb y^{\N^r}\rrb$ for which
\[ \rho_A(\vartheta^{(P,Q)}_{m,+}) =\vartheta^{(P',Q')}_{Am,+} . \]
An alternative approach is to rigidly translate the initial walls of $\f{D}^{(P,Q)}$ so that the origin is on the positive side of each initial wall. The corresponding consistent completion (which works mutatis mutandis for affine walls---cf. \S \ref{sub:gen}(2)) will still contain this positive chamber around the origin, and one may define $\vartheta^{(P,Q)}_{m,+}$ using broken lines with basepoint in this positive chamber.

\begin{exam}\label{eg:pos-chamber}
Let $M\coloneqq  \mathbb{Z}^2 =: N$, and consider the seed
\[ 
{P} = \begin{bmatrix} 0 & -1 & 1 \\ 1 & 0 & -1 \end{bmatrix},
\hspace{1cm}
{Q} = \begin{bmatrix} 1 & 0 & -1 \\ 0 & 1 & -1 \end{bmatrix}
\]
Since there are no points $p=(p_1,p_2)\in M_\mathbb{R}=\mathbb{R}^2$ for which the functionals
\[ p\cdot Qe_1 = p_1 , 
\hspace{1cm}
p\cdot Qe_2 = p_2,
\hspace{1cm}
p\cdot Qe_3 = -p_1-p_2
\]
are all positive, the incoming walls must be translated away from the origin to allow for a positive chamber. A consistent scattering diagram obtained from such a translation is depicted below, with the positive chamber in green.
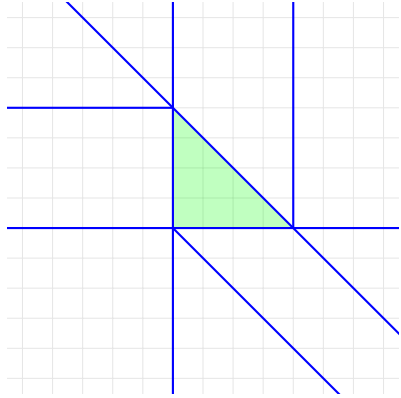
\[\begin{tikzpicture}
\begin{scope}[scale=.4]
	\clip (-6.5,-6.5) rectangle (6.5,6.5);
    	\draw[step=1,draw=black!10,very thin] (-7.5,-7.5) grid (7.5,7.5);
    \path[fill=green,opacity=.25] (-1,-1) to (3,-1) to (-1,3) to (-1,-1);
	\draw[blue, thick] (-1,-8) to (-1,8);
	\draw[blue, thick] (-8,-1) to (8,-1);
	\draw[blue, thick] (-6,8) to (8,-6);
	\draw[blue, thick] (3,-1) to (3,8);
	\draw[blue, thick] (-1,3) to (-8,3);
	\draw[blue, thick] (-1,-1) to (6,-8);
\end{scope}
\end{tikzpicture}\]
Alternatively, one may consider the seed with $M'\coloneqq \mathbb{Z}^3=:N'$ given by
\[ 
{P}' = \begin{bmatrix} 0 & -1 & 1 \\ 1 & 0 & -1 \\ -1 & 1 & 0 \end{bmatrix},
\hspace{1cm}
{Q}' = \begin{bmatrix} 1 & 0 & 0 \\ 0 & 1 & 0 \\ 0 & 0 & 1 \end{bmatrix}
\]
Since $Q'$ is injective, $\f{D}^{(P',Q')}$ has a positive chamber. 
There is a linear inclusion $A:(P,Q)\rightarrow (P',Q')$ given by 
\[ A =\begin{bmatrix}
1 & 0 \\
0 & 1 \\
-1 & -1 
\end{bmatrix} \]
The scattering diagram with translated walls depicted above can be realized as an affine slice of $\f{D}^{(P',Q')}$ parallel to the image of $A$; the green triangle is then the intersection of this slice with the positive chamber of $\f{D}^{(P',Q')}$.  This is a general phenomenon.
\end{exam}

Analogously, we may define a \textbf{negative chamber} $C^-$ of $\f{D}^{(P,Q)}$ and the corresponding theta functions $\vartheta^{(P,Q)}_{m,-}$ (with analogous workarounds if a negative chamber does not exist). 

\begin{lem}\label{lemma: swap}
For all $(P,Q)$,
$ \vartheta^{(P,Q)}_{m,-} = \vartheta^{(P,-Q)}_{m,+}$.
\end{lem}
\begin{proof}
The support and wall-crossing transformations of the initial walls in $\f{D}^{(P,Q)}$ and $\f{D}^{(P,-Q)}$ coincide---the signs of the $n_i$'s and $s_i$'s in \eqref{eq:Egamma} both change, and these pairs of changes cancel.  The only difference is which side is considered ``positive.'' By the uniqueness of the consistent completions, the same is true for the non-initial walls, and the broken lines in the two scattering diagrams coincide. Since the negative chamber in $\f{D}^{(P,Q)}$ is the positive chamber in $\f{D}^{(P,-Q)}$, the result follows.
\end{proof}

A linear morphism $A:(P,Q)\rightarrow (P',Q')$ of seed data sends positive chambers to positive chambers and negative chambers to negative chambers. Therefore, as a special case of Theorem \ref{thm: linearmorphism},
\begin{align}\label{eq: rhoA}
     \rho_A(\vartheta^{(P,Q)}_{u,+}) = \vartheta^{(P',Q')}_{Au,+}\quad \text{and} \quad
\rho_A(\vartheta^{(P,Q)}_{u,-}) = \vartheta^{(P',Q')}_{Au,-} 
\end{align}

We note that all the above results of this subsection are easily extended to the rational broken line analogs $\Theta_m$ for $m\in M_{\Q}$, either directly or using Lemma \ref{lem:Theta}.

\subsection{Theta reciprocity}\label{sub:TR}

Given a seed datum $(P,Q)$, we define the \textbf{chiral dual} seed datum to be $(Q^{\bullet},P^{\bullet})$ with the same diagonal matrix $D$.  Note that the double chiral dual is $(P,Q)$ again.  Similarly, the \textbf{chiral-Langlands dual} seed datum to $(P,Q)$ is $(Q,P)$.  This terminology is based on \cite[\S 1.2]{FG1}---one computes that $(P^{\bullet})^{\top}Q^{\bullet}=-B$ and $P^{\top}Q=B^{\top}$.

Given a linear morphism $A:(P,Q)\rightarrow (P',Q')$, there are \textbf{adjoint linear morphisms}
$$A^\top:((Q')^{\bullet},(P')^{\bullet}) \rightarrow (Q^{\bullet},P^{\bullet}) \qquad \text{and} \qquad A^{\top}:(Q',P')\rar (Q,P).$$ 
This defines a contravariant equivalence between the category of seed data with exchange matrix $B$ and those with exchange matrix $-B$ (respectively, $B^\top$). 

Recall from  \S \ref{sub:val} that each $n\in N_{\R}$ determines a valuation $\val_n:\wh{A}\rar \R^{\min}$.  Here and below we continue to employ the modified valuation and theta function notation described in \S \ref{sub:lin-indep}---in particular, we restrict to indices in $N$ and $M$ rather than $N\oplus \check{\sQ}^{\gp}$ and $M\oplus \sQ^{\gp}$.\footnote{An extension to the larger lattices is possible as well, either by direct modification or by applying some carefully chosen linear morphisms that serve to extend $N$ and $M$ to $N\oplus \check{\sQ}^{\gp}$ and $M\oplus \sQ^{\gp}$.}  Linear morphisms of seed data satisfy the following adjunction property with respect to valuations.

\begin{prop}[Adjunction]\label{prop: adjunction}
For any (rational) linear morphism of seed data $A:(P,Q)\rightarrow (P',Q')$, $m \in M_{\Q}$, and $n\in N'_\mathbb{R}$,
\[ \mathrm{val}_n(\Theta^{(P',Q')}_{Am,+}) = \mathrm{val}_{A^\top n} (\Theta^{(P,Q)}_{m,+}).\]
\end{prop}
\begin{proof}
Both sides are equal to $\mathrm{val}_{n} (\rho_A(\Theta^{(P,Q)}_{m,+}))$.  See \eqref{eq: rhoA} for the left-hand side.  This equality for the right-hand side follows from the fact that $Au_{\Gamma}\cdot n=u_{\Gamma}\cdot A^{\top} n$ for all $u_{\Gamma}\in \Theta_{m,+}^{(P,Q)}$.
\end{proof}

We now state our first version of Theta Reciprocity for chiral dual pairs.

\begin{claim}[Theta reciprocity, version 1]\label{conj: theta1}
For any seed datum $(P,Q)$, $m\in M_{\Q}$, and $n\in N_{\Q}$,  
\[ \mathrm{val}_n(\Theta^{(P,Q)}_{m,+}) = \mathrm{val}_m(\Theta^{(Q^{\bullet},P^{\bullet})}_{n,+}) \]
whenever both sides are finite.
\end{claim}

 By the symmetry of the statement, we see that Claim \ref{conj: theta1} holds for $(P,Q)$ if and only if it holds for $(Q^{\bullet},P^{\bullet})$.

We shall prove Claim \ref{conj: theta1} in \S \ref{sec:Lambda}. The next lemmas show that it suffices to prove the claim under certain simplifying assumptions.

\begin{lem}\label{lemma: deduceTR1}
Let $A:(P,Q)\rightarrow (P',Q')$ be a linear morphism of seed data.
\begin{enumerate}
    \item If $A$ is surjective and Claim \ref{conj: theta1} holds for $(P,Q)$, then Claim \ref{conj: theta1} holds for $(P',Q')$.
    \item If $A$ is injective and Claim \ref{conj: theta1} holds for $(P',Q')$, then Claim \ref{conj: theta1} holds for $(P,Q)$.
\end{enumerate}
\end{lem}
\begin{proof}
Suppose that $A$ is surjective and Claim \ref{conj: theta1} holds for $(P,Q)$. For any $m'\in M'$, we may find an $m\in M$ such that $A m=m'$. Then, for any $n'\in N'$, 
\[ \mathrm{val}_{n'}(\vartheta^{(P',Q')}_{m',+}) 
= \mathrm{val}_{n'}(\vartheta^{(P',Q')}_{Am,+}) 
= \mathrm{val}_{A^\top n'}(\vartheta^{(P,Q)}_{m,+}) 
\]
where for the last equality we use Proposition \ref{prop: adjunction}.  Similarly,
\[ \mathrm{val}_{m'}(\vartheta^{((Q')^{\bullet},(P')^{\bullet})}_{n',+}). 
 = \mathrm{val}_{A m}(\vartheta^{((Q')^{\bullet},(P')^{\bullet})}_{n',+})  =  \mathrm{val}_{m}(\vartheta^{(Q^{\bullet},P^{\bullet})}_{A^{\top} n',+}). 
\]
So if $\mathrm{val}_{n'}(\vartheta^{(P',Q')}_{m',+})$ and $\mathrm{val}_{m'}(\vartheta^{((Q')^{\bullet},(P')^{\bullet})}_{n',+})$ are both finite, then  $\mathrm{val}_{A^\top n'}(\vartheta^{(P,Q)}_{m,+})$ and $\mathrm{val}_{m}(\vartheta^{(Q^{\bullet},P^{\bullet})}_{A^{\top} n',+})$ are also finite, hence equal by the application of Claim \ref{conj: theta1}.  So $\mathrm{val}_{n'}(\vartheta^{(P',Q')}_{m',+})=\mathrm{val}_{m'}(\vartheta^{((Q')^{\bullet},(P')^{\bullet})}_{n',+})$, thus proving Claim (1).

Next, assume that $A$ is injective and Claim \ref{conj: theta1} holds for $(P',Q')$. Then the claim holds for $((Q')^{\bullet},(P')^{\bullet})$, and the adjoint linear morphism $A^\top:((Q')^{\bullet},(P')^{\bullet})\rightarrow (Q^{\bullet},P^{\bullet})$ is surjective.  So by the previous case, Claim \ref{conj: theta1} holds for $(Q^{\bullet},P^{\bullet})$ and therefore for $(P,Q)$.
\end{proof}

One consequence is that we may reduce to the case of saturated-injective seed data.

\begin{lem}\label{lemma: satTR}
If Claim \ref{conj: theta1} holds for all saturated-injective seed data, then it holds for all seed data.
\end{lem}
\begin{proof}
This is an immediate consequence of Lemmas \ref{lemma: saturatedresolution} and \ref{lemma: deduceTR1}.
\end{proof}

The following proposition applies ideas from \S \ref{sub:lin-indep} to show that when $D$ and $B^{\bullet}$ are both integral, one may use the chiral-Langlands dual seed datum $(Q,P)$ for Claim \ref{conj: theta1} in place of the chiral dual seed datum $(Q^{\bullet},P^{\bullet})$.

\begin{prop}\label{lem:D-move}
    Assume that $D$ and $B^{\bullet}$ are both integral. Then  $\val_{n}(\Theta^{(P,Q)}_{m,+})=\val_{n}(\Theta^{(P^{\bullet},Q^{\bullet})}_{m,+})$ for all $m\in M_{\Q}$ and $n\in N_{\Q}$.  In particular, Claim \ref{conj: theta1} in these cases implies \[ \mathrm{val}_n(\Theta^{(P,Q)}_{m,+}) = \mathrm{val}_m(\Theta^{(Q,P)}_{n,+}) \]
    whenever both sides are finite.
\end{prop}

\begin{proof}
Recall that $D=\diag(d_1,\ldots,d_r)$ is assumed to be integral.  Let $$\sQ'=\bigoplus_{i=1}^r \bigoplus_{j=1}^{d_i} \N \cdot e_{ij}$$ and consider the seed datum $(P',Q')$ with $$P':e_{ij}\mapsto Pe_i, \qquad Q':e_{ij}\mapsto Q^{\bullet}e_i$$ with $D'$ the identity matrix.  Here we use the integrality of $B^{\bullet}=(Q^{\bullet})^{\top}P$ to ensure that $P'$ and $Q'$ take values in some dual lattices $M'\subset M_{\Q}$ and $N'\subset N_{\Q}$, respectively; i.e. $B'=(Q')^{\top}P'$ is integral.  Then $$\f{D}^{(P',Q')}_{\In}=\left\{((Q^{\bullet}e_i)^{\perp},1+x^{Pe_i}y^{e_{ij}},Q^{\bullet}e_i)|i=1,\ldots,r; ~ j=1,\ldots,d_i\right\}.$$  

Note that $$\f{D}^{(P,Q)}_{\In}=\lrc{\left.\lrp{(Q e_i)^{\perp},(1+x^{Pe_i}y^{e_i})^{d_i},Q^{\bullet}e_i}\, \right| \, i=1,\ldots,r},$$ and this is equivalent to the scattering diagram obtained from $\f{D}_{\In}^{(P',Q')}$ via the specialization of coefficients $y^{e_{ij}}\mapsto y^{e_i}$.  Similarly, $$\f{D}_{\In}^{(P^{\bullet},Q^{\bullet})}=\lrc{\left.\lrp{(Q^{\bullet}e_i)^{\perp},1+x^{P^{\bullet}e_i}y^{e_i},Q^{\bullet}e_i}\, \right| \, i=1,\ldots,r}$$
 is equivalent to the scattering diagram obtained from $\f{D}_{\In}^{(P',Q')}$ via the specialization of coefficients $y^{e_{ij}}\mapsto -\zeta_j y^{e_i/d_i}$  (using a refinement of $\sQ$ containing each $e_i/d_i$) where $\zeta_1,\ldots,\zeta_{d_i}$ are the $d_i$th roots of $(-1)$.  The first claim then follows from Theorem \ref{thm:nu}, which implies that valuations do not depend on the precise specialization of coefficients.

The final claim follows from applying the first claim to relate the seed $(Q^{\bullet},P^{\bullet})$ to $(Q,P)$, noting that $Q=(Q^{\bullet})^{\bullet}$, $P=(P^{\bullet})^{\bullet}$, and $((P^{\bullet})^{\bullet})^{\top}Q^{\bullet}=(B^{\bullet})^{\top}$ is integral if and only if $B^{\bullet}$ is integral.\footnote{As in Remark \ref{rem:scale}, one may always rescale $D$ and $P$ by any $k\in \Q_{>0}$ while simultaneously rescaling $Q^{\bullet}$ by $k^{-1}$ (fixing $P$ and $Q$), hence rescaling $B^{\bullet}$ by $k^{-1}$.  So it is always possible to ensure that $D$ or $B^{\bullet}$ is integral, but making both integral simultaneously is not always possible.  Without this integrality, the scattering diagram $\f{D}_{\In}^{(P',Q')}$ in the proof of Proposition \ref{lem:D-move} is not well-defined.  One might hope to use the equivalence $E_{kn,f}\simeq E_{n,f^k}$ to define $(\f{d},f,n)$ with $n\in N_{\Q}$---e.g., defining $(Q^{\bullet}e_i^{\perp},1+x^{Pe_i}y^{e_{ij}},Q^{\bullet}e_i)$ to be $(Qe_i^{\perp},(1+x^{Pe_i}y^{e_{ij}})^{1/d_i},Qe_i)$ where $(1+x^{Pe_i}y^{e_{ij}})^{1/d_i}$ is defined in $\Q[x^M]\llb y^{\sQ}\rrb$, but this loses the positivity that we require prior to specialization of coefficients.}
\end{proof}

The mirror constructions in \cite{GHKK} are typically understood as relating a cluster $\s{A}$-variety to its Langlands dual cluster $\s{X}$-variety.  As in \cite[\S 1.2]{FG1}, given a seed with exchange matrix $B$, the Langlands dual seed should have exchange matrix $-B^{\top}$ (which we note equals $B$ in the skew-symmetric cases).  So using the descriptions from Example \ref{eg:clusterA}, an $\s{A}$-type seed datum $(B,\Id)$ should be mirror to the $\s{X}$-type seed datum $(\Id,-B)$.  More generally, given a seed datum $(P,Q)$, we define the \textbf{Langlands dual} seed datum to be $(Q,-P)$.  We note that the double Langlands dual $(-P,-Q)$ is isomorphic to $(P,Q)$ via the invertible linear morphism $-\Id$.

\begin{cor}[Theta reciprocity for Langlands dual pairs]\label{cor:Langlands}
    Consider a seed datum $(P,Q)$ for which $D$ and $B^{\bullet}$ are both integral.  Then Claim \ref{conj: theta1}  implies
    \[ \mathrm{val}_n(\Theta^{(P,Q)}_{m,+}) = \mathrm{val}_m(\Theta^{(Q,-P)}_{n,-})\]
    whenever both sides are finite.
\end{cor}
\begin{proof}
    We have
    \begin{align*}
        \mathrm{val}_n(\Theta^{(P,Q)}_{m,+}) &= \mathrm{val}_m(\Theta^{(Q,P)}_{n,+}) \qquad \text{(Prop. \ref{lem:D-move})} \\
        &=\mathrm{val}_m(\Theta^{(Q,-P)}_{n,-}) \qquad \text{(Lem. \ref{lemma: swap})}
    \end{align*}
\end{proof}

\section{$\Lambda$-structures and the proof of Theta Reciprocity}\label{sec:Lambda}

The key to our proof of theta reciprocity is to translate it into a relation between pairs of theta functions on the same scattering diagram.  This is accomplished by choosing a $\Lambda$-structure.

\subsection{$\Lambda$-structures}

Let $(P,Q)$ be a seed datum. An \textbf{integral (or rational) $\Lambda$-structure} on $(P,Q)$ is an integral (respectively, rational) linear morphism of seed data $\Lambda:(P,Q)\rar (Q^{\bullet},-P^{\bullet})$.    I.e., 
\begin{align}\label{eq:LambdaP}
    \Lambda P = Q^{\bullet} \quad \text{and} \quad \Lambda^\top P = -Q^{\bullet}
\end{align}
(the latter equation is equivalent to $\Lambda^{\top}(-P^{\bullet})=Q$).  Note that given the first condition $\Lambda P = Q^{\bullet}$, the condition $\Lambda^\top P = -Q^{\bullet}$ is equivalent to saying that $\Lambda|_{\im(P)}$ is skew-symmetric.

\begin{rem}[$\Lambda$ as a bilinear pairing]\label{rem:omega}
    A rational $\Lambda$-structure may equivalently be defined in terms of a $\Q$-valued bilinear pairing $\omega$ on $M$ which satisfies \begin{align}\label{eq:omega}
    \omega(m,Pv)=m\cdot Q^{\bullet}v=-\omega(Pv,m)
\end{align} for all $m\in M$ and $v\in \N^r$.  The pairing $\omega$ is related to $\Lambda$ by $$\omega(m_1,m_2)=m_1\cdot \Lambda m_2.$$  Indeed, for $\omega$ defined in terms of $\Lambda$ in this way, we have 
$$\omega(m,Pv)=m\cdot \Lambda Pv = m\cdot Q^{\bullet}v
$$
giving the left equality from \eqref{eq:omega}, and
$$\omega(Pv,m)=Pv\cdot \Lambda m = m \cdot \Lambda^{\top} Pv = -m\cdot Q^{\bullet}v
$$ giving the other half of \eqref{eq:omega}. Conversely, for $\Lambda$ defined by $m_1\cdot \Lambda m_2 = \omega(m_1,m_2)$ for $\omega$ satisfying \eqref{eq:omega}, we have $m \cdot \Lambda Pv = \omega(m ,Pv)=m \cdot Q^{\bullet}v$, so $\Lambda P=Q^{\bullet}$.  Furthermore, $$m\cdot \Lambda^\top P v=Pv \cdot \Lambda m=\omega(Pv,m)=-m\cdot Q^{\bullet}v$$ so $\Lambda^\top P=-Q^{\bullet}$ holds as well.
\end{rem}

\begin{rem}[$\Lambda$-structures and quantization]
If the pairing $\omega$ from Remark \ref{rem:omega} is skew-symmetric, it determines a Poisson structure on $\Z[x^M][y^{\sQ}]$ via 
\[ \{ x^{m_1},x^{m_2} \} = \omega(m_1,m_2) x^{m_1+m_2}. \]
This yields a quantization in which the Laurent polynomial rings become quantum tori.

Given an integral skew-symmetrizable matrix $B$, the $\s{X}$-type seed datum $(\Id,B^{\top})$ always admits a $\Lambda$-structure $\Lambda:(\Id,B^{\top})\rar (-B^{\bullet},-D)$ with $\Lambda=-B^{\bullet}$.  This is used in \cite{FG1} to quantize the cluster $\s{X}$-varieties.  On the other hand, the $\s{A}$-type seed datum $(B,\mathrm{Id})$ only admits a $\Lambda$-structure when $B$ is invertible, in which case $\Lambda$ is represented by a ``\emph{compatible matrix}'' as in the quantization of cluster algebras from \cite{BZ}.

\end{rem}

\begin{lem}\label{lem:Lambda-exist}
 $(P,Q)$ admits a rational $\Lambda$-structure if and only if $\ker(P)\subset \ker(Q^{\bullet})$.  In particular, rational $\Lambda$-structures exist whenever $P$ is injective.
\end{lem}
\begin{proof}
    The necessity of $\ker(P)\subset \ker(Q^{\bullet})$ is immediate from the condition $ \Lambda P = Q^{\bullet}$.  
    
    Conversely, assume $\ker(P)\subset \ker(Q^{\bullet})$.  For $s=\rank(P)$, $r=\rank(\sQ^{\gp})$, and $d=\rank(M)=\rank(N)$, we may choose bases for $\sQ_{\Q}^{\gp}$ and $M_{\Q}$ such that 
    \[ P\simeq \begin{bmatrix} \mathrm{Id}_{s\times s} & 0_{s\times (r-s)} \\ 0_{(d-s)\times s} & 0_{(d-s)\times (r-s)} \end{bmatrix}.\]
    Then since $\ker(P)\subset \ker(Q^{\bullet})$, we know that for this same basis of $\sQ_{\Q}^{\gp}$, the last $(r-s)$ columns of $Q^{\bullet}$ will also be $0$.  So $Q^{\bullet}$ will have the form
    \[ Q^{\bullet}\simeq \begin{bmatrix} -\bar{B}^{\bullet}_{s\times s} & 0_{s\times (r-s)}\\C & 0_{(d-s)\times (r-s)}\end{bmatrix}\]
    where $-\bar{B}^{\bullet}_{s\times s}$ is skew-symmetric and $C$ is some $(d-s)\times s$-matrix.  Indeed, we compute $$B^{\bullet}=(Q^{\bullet})^{\top}P=\begin{bmatrix} -(\bar{B}^{\bullet}_{s\times s})^{\top} & 0_{s\times (r-s)}\\0_{(r-s)\times s} & 0_{(r-s)\times (r-s)}\end{bmatrix}$$
    so the skew-symmetry of $B^{\bullet}$ is equivalent to the skew-symmetry of $\bar{B}^{\bullet}_{s\times s}$ (which then equals the top-left $s\times s$ submatrix of $B^{\bullet}$)

    Now take $$\Lambda=\begin{bmatrix} -\bar{B}^{\bullet}_{s\times s} & -C^{\top} \\C & A\end{bmatrix}$$
    for any $(d-s)\times(d-s)$-matrix $A$.  One checks that $\Lambda P=Q^{\bullet}$ and $\Lambda^{\top} P=-Q^{\bullet}$, so this gives a rational $\Lambda$-structure as desired.
\end{proof}

Rational $\Lambda$-structures will suffice for our purposes, but we include the following sufficient condition for the existence of integral $\Lambda$-structures.

\begin{lem}\label{lem: satlambda}
If $P$ is saturated and $\ker(P)\subset \ker(Q^{\bullet})$, then $(P,Q)$ admits an integral $\Lambda$-structure.
\end{lem}

\begin{proof}
This follows by essentially the same argument as Lemma \ref{lem:Lambda-exist}, except that $P$ being saturated ensures that we can choose our bases over $\Z$ instead of over $\Q$.
\end{proof}

\begin{lemdfn}\label{lemdfn:comp}
If $A:(P,Q)\rightarrow (P',Q')$ is a linear morphism and $\Lambda'$ is a $\Lambda$-structure for $(P',Q')$, then $\Lambda\coloneqq A^\top \Lambda' A$ is a $\Lambda$-structure for $(P,Q)$.  In this case, we say that the $\Lambda$-structures $\Lambda$ and $\Lambda'$ are \textbf{congruent}.
\end{lemdfn}
\begin{proof}
This is a direct check: 
\begin{align*}
    &(A^\top \Lambda' A) P = A^\top \Lambda' P' = A^\top {Q'}^{\bullet} = Q^{\bullet} \\ &(A^\top \Lambda' A)^{\top} P = A^\top (\Lambda')^{\top} A P = A^\top (\Lambda')^{\top} P' = - A^\top {Q'}^{\bullet} = -Q^{\bullet}.
\end{align*}
\end{proof}

We note that the congruence condition above is equivalent to the condition that $$m_1\cdot \Lambda m_2=Am_1 \cdot \Lambda' Am_2$$ for all $m_1,m_2\in M$.  We will not actually utilize the congruence condition, but it seems to be a natural condition to ask for in Lemma \ref{lemma: invlambda} below.

We say that an integral (respectively, rational) $\Lambda$-structure is \textbf{invertible} if the map $\Lambda:M\rar N^\bullet$ (respectively, $\Lambda:M_\Q\rar N^\bullet_\Q$) is invertible. 

\begin{lem}\label{lemma: invlambda}
Every seed datum with an integral (respectively, rational) $\Lambda$-structure $\Lambda$ embeds into a seed datum with an invertible integral (respectively, rational) $\Lambda$-structure.  That is, there exists a linear morphism $A:(P,Q)\rightarrow (P',Q')$ with $A$ injective and $(P',Q')$ admitting an invertible $\Lambda$-structure $\Lambda'$.  Moreover, $A$ and $\Lambda'$ can be chosen so that $\Lambda$ and $\Lambda'$ are congruent in the sense of Lemma/Definition \ref{lemdfn:comp}.  Additionally, if $(P,Q)$ is saturated-injective, then we may take $(P',Q')$ to be saturated-injective as well.
\end{lem}

\begin{proof}
Let $(P,Q)$ be a seed datum with diagonal matrix $D$ and $\Lambda$-structure $\Lambda$.
Define a seed datum $(P',Q')$ with $D'=D$ and \begin{align*}
    M'=M\oplus N^\bullet, \qquad &N'=N\oplus M^{\bullet}\\
    P':v\mapsto(Pv,0)\in M', \qquad &Q':v\mapsto (Qv,P^{\bullet} v)\in N'.
\end{align*} 
Then 
\eqn{
\Lambda':M' &\to {N'}^\bullet=N^{\bullet}\oplus M\\
    (m,n)   &\mapsto \lrp{\Lambda m - n,m}
}
is an invertible $\Lambda$-structure for $(P',Q')$.  The inverse is $(\Lambda')^{-1}:(n,m)\mapsto (m,\Lambda m-n)$.  Checking the conditions $\Lambda'P'=(Q')^{\bullet}$ and $(\Lambda')^{\top}P'=-(Q')^{\bullet}$ is straightforward after representing the maps with matrices:
\[
P'=
\begin{bmatrix}
P \\ 
0_{d\times r}
\end{bmatrix},
\qquad 
(Q')^{\bullet}
=\begin{bmatrix}
Q^{\bullet} \\ 
P
\end{bmatrix},
\qquad \text{and} \qquad  \Lambda'=\begin{bmatrix}
\Lambda & -\mathrm{Id}_{d\times d} \\
\mathrm{Id}_{d\times d} & 0_{d\times d}
\end{bmatrix}.
\]

Finally, we have an injective linear morphism $A:(P,Q)\rightarrow (P',Q')$, $m\mapsto (m,0)$.  In terms of matrices, $A=\begin{bmatrix}
\mathrm{Id}_{d\times d} \\
0_{d\times d} 
\end{bmatrix}$.  The congruence equation $A^{\top}\Lambda'A=\Lambda$ is now straightforward to check.  For the final claim, note that our $P'$ and $Q'$ are indeed saturated inclusions whenever $P$ and $Q$ are.
\end{proof}

Using a $\Lambda$-structure, we can translate theta reciprocity into a statement relating valuations in the same seed datum.

\begin{claim}[Theta reciprocity, version 2]\label{conj: theta2}
Given a seed datum $(P,Q)$ with a rational $\Lambda$-structure $\Lambda$ and $m_1,m_2\in M_{\Q}$,
\[ \mathrm{val}_{\Lambda m_1}(\Theta^{(P,Q)}_{m_2,+})
= \mathrm{val}_{\Lambda^\top m_2}(\Theta^{(P,Q)}_{m_1,-}).
\]
whenever both sides are finite.
\end{claim}

\begin{lem}\label{lem:theta2-theta1}
Claim \ref{conj: theta2} is equivalent to Claim \ref{conj: theta1}.
\end{lem}

\begin{proof}
We have
\begin{align*}
\val_{\Lambda^\top m_2} (\Theta^{(P,Q)}_{m_1,-}) 
&= \val_{m_2}(\Theta^{(Q^{\bullet},-P^{\bullet})}_{\Lambda m_1,-})  \qquad \text{(by adjunction, Proposition \ref{prop: adjunction})} \\
&= \val_{m_2}( \Theta^{(Q^{\bullet},P^{\bullet})}_{\Lambda m_1,+}) \qquad \text{(by Lemma \ref{lemma: swap})} 
\end{align*}
and so the equality in Claim \ref{conj: theta2} is equivalent to 
\[ \mathrm{val}_{\Lambda m_1}(\Theta^{(P,Q)}_{m_2,+})
= \val_{m_2}( \Theta^{(Q^{\bullet},P^{\bullet})}_{\Lambda m_1,+}).
\]
This is the case of Claim \ref{conj: theta1} in which $n=\Lambda m_1$ and $m=m_2$; furthermore, if $\Lambda$ is invertible (or at least surjective), then every case of Claim \ref{conj: theta1} for the seed datum $(P,Q)$ is of this form.

Therefore, Claim \ref{conj: theta1} implies Claim \ref{conj: theta2}, and the reverse holds whenever $(P,Q)$ has an invertible $\Lambda$-structure.  By Lemma \ref{lemma: satTR}, to prove Claim \ref{conj: theta1}, it suffices to consider saturated-injective seed data, and these always admit rational $\Lambda$-structures by Lemma \ref{lem:Lambda-exist} (in fact \textit{integral} $\Lambda$-structures by  Lemma \ref{lem: satlambda}).  Finally, we may assume that these $\Lambda$-structures are invertible by Lemma~\ref{lemma: invlambda} combined with Lemma~\ref{lemma: deduceTR1}(2).
\end{proof}

Our main goal in the remainder of this section is to prove Claim \ref{conj: theta2}. We extract the following from the proof of Lemma \ref{lem:theta2-theta1}.  
\begin{lem}\label{lem:sat}
    To prove Claim \ref{conj: theta2}, it suffices to prove the cases where $(P,Q)$ is saturated-injective.  In particular, we may assume that a full-dimensional positive chamber $C^+$ exists.
\end{lem}
\begin{proof}
    Suppose $(P,Q)$ is saturated-injective.  By Lemma \ref{lemma: invlambda}, there is an injective linear morphism $A:(P,Q)\rar (P',Q')$ for $(P',Q')$ saturated-injective and admitting an invertible $\Lambda$-structure.  By the proof of Lemma \ref{lem:theta2-theta1}, proving Claim \ref{conj: theta2} for $(P',Q')$ is sufficient for proving Claim \ref{conj: theta1} for $(P',Q')$.  Then Lemma \ref{lemma: deduceTR1}(2) implies Claim \ref{conj: theta1} for $(P,Q)$.  Now Lemma \ref{lemma: satTR} yields Claim \ref{conj: theta1} for all $(P,Q)$, and then Lemma \ref{lem:theta2-theta1} implies Claim \ref{conj: theta2} for all $(P,Q)$. 
\end{proof}

To simplify notation, we will begin writing $\Theta_{m,p}^{(P,Q)}$ as $\Theta_{m,p}$ and $\vartheta_{m,p}^{(P,Q)}$ as $\vartheta_{m,p}$, the superscript $(P,Q)$ being understood unless indicated otherwise.

\begin{rem}\label{rmk:irrational-exp}
    We noted in Footnote \ref{foot:irr} that the definition of rational broken lines could be extended to allow exponents in $M_{\R}$, not just in $M_{\Q}$.  In the present setting, the idea is that when a rational broken line $\Gamma$ crosses a wall $(W,f,n)$, $f\in 1+x^{Pv}y^v \Z\llb x^{Pv} y^v\rrb$ with $\deg f = K\in \Z_{\geq 0}\cup \{\infty\}$, the attached exponent $m\in M_{\R}$ can change to $m+\sigma k(m\cdot n)Pv$ for $\sigma = \sign(\Gamma(t-\epsilon)\cdot n)$ and $k$ any \emph{rational} number with $0\leq k\leq K$.  In what follows, we may allow our rational broken lines to take irrational exponents unless otherwise specified. 
\end{rem}

\subsection{$\Lambda$-tautness}\label{sub:Ltaut}

Let $(P,Q)$ be a seed datum with $\Lambda$-structure $\Lambda$.  Let $\Gamma$ be a corresponding rational broken line in $M_{\R}$.  Recall our notion of $v$-taut broken lines from \S \ref{sub:taut}, modified as in \S \ref{sub:lin-indep}.  In our present setting, we say that $\Gamma$ is $\Lambda$-\textbf{taut} if it is $n$-taut for $n=\Lambda m_{\Gamma}$ (recall $m_{\Gamma}$ is the $x$-exponent attached to the final straight segment).  Similarly, we say that $\Gamma$ is $\Lambda^{\top}$-taut if it is $n$-taut for $n=\Lambda^{\top}m_{\Gamma}$.  We shall now derive simple characterizations for these notions of tautness.

First, as in \S \ref{sub:taut} (modified as in \S \ref{sub:lin-indep}), let $t$ be a time when $\Gamma$ crosses some walls $\{(W_i,f_i,n_i)\}_i$ of $\f{D}^{(P,Q)}$, say with $f_i\in 1+x^{Pv}y^{v}\Z\llb x^{Pv}y^v\rrb$.  By Lemma \ref{lem:D-skew}, each $n_i\in N$ is a positive rational multiple of $Q^{\bullet} v$, so $W_i \subset (Q^{\bullet}v)^{\perp}$ for each $i$.  Using the equivalence $E_{kn,f}=E_{n,f^k}$, we can further assume that each $n_i$ is the primitive element $\nu\in N$ with direction $Q^{\bullet}v$ (primitive meaning $\nu$ is not a positive integer multiple of another element of $N$). Let $c\in \Q_{>0}$ be the number such that $\nu=cQ^{\bullet}v$.

Let $f_t=\prod_i f_i$ and let $K_t$ be the supremum of all $k\in \Z_{\geq 0}$ such that $kPv$ appears as an exponent in $f_t$.  Then $T_t:M_{\R}\rar M_{\R}$ (as in \S \ref{sub:lin-indep}) has the form 
\begin{align*}
    T_t:m&\mapsto m+\sigma ck (m\cdot Q^{\bullet}v) Pv 
\end{align*} for some $0\leq k\leq K_t$ and $\sigma \coloneqq \sign(\Gamma(t-\epsilon)\cdot Q^{\bullet}v)$.  The adjoint map is
\begin{align*}
    T_t^{\vee}:n\mapsto n+\sigma ck(Pv\cdot n)Q^{\bullet}v.
\end{align*}
In particular, for generic $s\in (-\infty,0]$, let $m_s=-\Gamma'(s)$ and $n_s=\Lambda m_s$.   Then \begin{align*}
    T_t^{\vee}(n_{t+\epsilon})=T_t^{\vee}(\Lambda m_{t+\epsilon})&=\Lambda m_{t+\epsilon}+\sigma ck(Pv\cdot \Lambda m_{t+\epsilon})Q^{\bullet}v \\
    &=\Lambda m_{t+\epsilon}-\sigma 
 ck( m_{t+\epsilon} \cdot Q^{\bullet} v)\Lambda Pv \\
    &= \Lambda \left(m_{t+\epsilon} -\sigma ck(m_{t+\epsilon} \cdot Q^{\bullet}v) Pv \right) \\
    &= \Lambda \left( T_t^{-1}(m_{t+\epsilon}) \right)\\
    &=\Lambda m_{t-\epsilon} \\
    &=n_{t-\epsilon}.
\end{align*}
Thus, $n_s=\Lambda m_s$ is the same as the vector (also denoted $n_s$) used for defining $\Lambda m_{\Gamma}$-tautness.  That is, the tautness condition \eqref{eq:taut-ineq} at time $t$ becomes
\begin{align*}
    \sigma ck(m_{t-\epsilon} \cdot Q^{\bullet}v)Pv \cdot \Lambda m_{t+\epsilon} \leq \sigma  ck'(m_{t-\epsilon}\cdot Q^{\bullet}v) Pv\cdot \Lambda m_{t+\epsilon}
\end{align*}
for all $0\leq k'\leq K_t$.  Since $m_{t-\epsilon}\cdot Q^{\bullet}v = m_{t+\epsilon} \cdot Q^{\bullet}v$ (because $Pv\cdot Q^{\bullet}v=0$ by the skew-symmetry of $B^{\bullet}$), and since $Pv\cdot \Lambda m_{t+\epsilon} = -m_{t+\epsilon}\cdot Q^{\bullet}v$, the inequality can be rewritten as
\begin{align*}
    - \sigma ck (m_{t+\epsilon}\cdot Q^{\bullet}v)^2 \leq -\sigma ck' (m_{t+\epsilon}\cdot Q^{\bullet}v)^2,
\end{align*}
or equivalently,
\begin{align*}
\sigma(k-k')\geq 0.
\end{align*}
Recall that the positive side of the above walls is where $Q^{\bullet}v$ is positive, so $\sigma$ is positive if we cross from the positive side of the walls to the negative side, and $\sigma$ is negative otherwise. Thus, $\Lambda$-tautness means the following:
\begin{lem}\label{lem:Lambda-taut}
    A broken line is $\Lambda$-taut if it takes:
    \begin{itemize}
        \item the largest possible bend\footnote{In particular, if a broken line crosses from the positive side to negative side of a wall with an infinite-degree scattering function (or infinitely many non-trivial walls with finite degree), then it cannot be $\Lambda$-taut.} ($k=K_t$) whenever crossing from the positive side of a wall to the negative side, and 
        \item the smallest possible bend ($k=0$) whenever crossing from the negative side of the wall to the positive side.
    \end{itemize}
\end{lem}

Similarly, one shows that $\Lambda^{\top}$-tautness is characterized by the following:
\begin{lem}\label{lem:LambdaT-taut}
    A broken line is $\Lambda^{\top}$-taut if it takes:
    \begin{itemize}
        \item the smallest possible bend ($k=0$) whenever crossing from the positive side of a wall to the negative side, and 
        \item the largest possible bend ($k=K_t$) whenever crossing from the negative side of the wall to the positive side.
    \end{itemize}
\end{lem}

Note that a straight broken line ending in the positive chamber will only cross from the negative sides of walls to the positive sides, and vice-versa for straight broken lines ending in the negative chamber.  Thus:

\begin{cor}\label{cor:straight-bl-taut}
    A straight broken line ending in the positive chamber is $\Lambda$-taut.  A straight broken line ending in the negative chamber is $\Lambda^{\top}$-taut.
\end{cor}

\begin{lem}\label{lem:Lambda-T-taut}
    Let $\Gamma$ be a rational or ordinary broken line with initial exponent $m$ and endpoint $p$.  Then $\Gamma$ is $(\Lambda p)$-taut if and only if it's $\Lambda^{\top}$-taut.  Similarly, $\Gamma$ is $(\Lambda^{\top} p)$-taut if and only if it's $\Lambda$-taut.
\end{lem}

    The idea behind the lemma is that one can modify the linear structure\footnote{See \cite[\S 2.1]{Man1} for background on (integral) linear structures.} on a conical neighborhood of $\Gamma\subset M_{\R}\setminus \Joints(\f{D})$, without changing the linear structure of the walls, so that $\Gamma$ looks like a  ray with no bends.\footnote{If $\Gamma$ self-intersects, one can work in a covering space of $M_{\R}\setminus \Joints(\f{D})$ with an induced scattering structure to avoid the self-intersections.}  With respect to this linear structure, the maps $T_{t_i}$ and $T_{t_i}^{\vee}$ are trivial.  So since $p$ and $m_{\Gamma}$ lie on opposite sides of the walls crossed by $\Gamma$ ($p$ being the base of the ray and $m_{\Gamma}$ being the direction), we find, for all $t$, that $(\Lambda p)_t=\Lambda p$ and $(\Lambda^{\top} m_{\Gamma})_t=\Lambda^{\top} m_{\Gamma}$ take the same signs on the exponents of every wall crossed by $\Gamma$.  Hence, $(\Lambda p)$-tautness is equivalent to $\Lambda^{\top}$-tautness.  The argument below is a more explicit version of this idea without using the language of linear structures. 

\begin{proof}
    Suppose $\Gamma$ is $n$-taut for $n=\Lambda p$.  We claim that 
    \begin{align}\label{eq:vt}
        n_t=\Lambda[\Gamma(t)-t\Gamma'(t)]
    \end{align} for all $t\in (-\infty,0]$ where $\Gamma'(t)$ is defined.  Let $t_1<t_2<\ldots<t_k$ be the times where $\Gamma$ bends.  For $t_k<t<0$, we have $n_t=\Lambda p$ and $p=\Gamma(t)-t\Gamma'(t)$, so the claim holds.  We now work by induction on decreasing values of $t$.  
    By the inductive assumption, $n_{t_i+\epsilon}=\Lambda p_{t_i}$ where $p_{t_i}\coloneqq\Gamma(t_i)-t_i\Gamma'(t_i+\epsilon)$.  The bend at $t_i$ has the form $T_{t_i}:m\mapsto m+c\sigma(m\cdot Q^{\bullet} v) Pv$ for some nonzero $c\in \Q$, $Q^{\bullet} v$ perpendicular to the wall (so $\Gamma(t_i)\cdot Q^{\bullet} v=0$), and $\sigma = \sign(m\cdot Q^{\bullet} v)$.  Then 
    \begin{align*}
        T_{t_i}^{\vee}:n\mapsto n+c\sigma(Pv\cdot n)Q^{\bullet}v.
    \end{align*}
    We now compute
    \begin{align*}
        n_{t_i-\epsilon} = T_{t_i}^{\vee}[\Lambda p_{t_i}] &= \Lambda(p_{t_i}) + c\sigma \left[Pv\cdot \left(\Lambda(\Gamma(t_i)-t_i\Gamma'(t_i+\epsilon))\right)\right]Q^{\bullet}v\\
        &=\Lambda(p_{t_i})+c\sigma t_i\left[\Gamma'(t_i+\epsilon)\cdot Q^{\bullet}v \right]\Lambda P v \\
        &= \Lambda\left[\Gamma(t_i)-t_i\Gamma'(t_i+\epsilon)+c\sigma t_i\left(\Gamma'(t_i+\epsilon)\cdot Q^{\bullet}v \right)Pv\right] \\
        &=\Lambda\left[\Gamma(t_i)-t_i T_{t_i}^{-1}(\Gamma'(t_i+\epsilon))\right] \\
        &=\Lambda\left[\Gamma(t_i)-t_i\Gamma'(t_i-\epsilon)
        \right].
    \end{align*}
    Since $\Gamma(t)-t\Gamma'(t)$ is constant on each straight segment, the claim \eqref{eq:vt} follows.

    Now, $\Gamma$ being $n$-taut means that as we increase $t$, $\Gamma$ always takes the bend which has minimal dual pairing with $n_{t+\epsilon}$.  Let $(W,f,n)$ be a wall containing $\Gamma(t)$.   By \eqref{eq:vt}, $n_{t+\epsilon}=\Lambda p_{t+\epsilon}$ for $p_{t+\epsilon}=\Gamma(t+\epsilon)-(t+\epsilon)\Gamma'(t+\epsilon)$.  Note that $p_{t+\epsilon}$ lies on the same side of $W$ as $\Gamma(t+\epsilon)$.  The bend $c\sigma (m\cdot Q^{\bullet} v)Pv$ is a non-negative multiple of $Pv$, and $Pv\cdot \Lambda p_{t+\epsilon}=- p_{t+\epsilon}\cdot Q^{\bullet}v$, so the $n$-tautness condition means taking the maximal bend if crossing to the side of $W$ where $Q^{\bullet} v$ is positive, and taking the minimal bend (i.e., no bend) otherwise.  By Lemma \ref{lem:LambdaT-taut}, this is precisely the condition  for $\Gamma$ to be $\Lambda^{\top}$-taut.  The first claim follows.

    The claim for $\Lambda^{\top} p$-tautness being equivalent to $\Lambda$-tautness follows by a completely analogous argument.
    \end{proof}

\subsection{Constancy along taut broken lines}\label{sub:LConvex}

Suppose we have a family $\varphi$ of piecewise-linear functions $\varphi_p:N_{\R}\rar \R^{\min}$ for each generic $p\in M_{\R}$.  We are particularly interested in the cases $\varphi_p=\Theta^{\trop}_{m,p}$.  Fix a rational broken line $\Gamma$, and let $m_t$ be the attached exponent at generic time $t$.  We say that $\varphi$ is $\Lambda$-\textbf{constant} along $\Gamma$ if $\varphi_{\Gamma(t)}(\Lambda m_t)$ is constant with respect to $t$.  Similarly,  we say $\varphi$ is $\Lambda^{\top}$-\textbf{constant} along $\Gamma$ if $\varphi_{\Gamma(t)}(\Lambda^{\top} m_t)$ is constant with respect to $t$.

\begin{lem}\label{lem:L-const}
    For all $m\in M_{\Q}$, $\Theta_m^{\trop}$ is $\Lambda$-constant along $\Lambda$-taut rational broken lines and $\Lambda^{\top}$-constant along $\Lambda^{\top}$-taut rational broken lines.
\end{lem}

\begin{proof}
    We prove the $\Lambda$-constancy claim.  The $\Lambda^{\top}$ analog is proved similarly.  Thanks to Lemmas \ref{lem:val-ku} and \ref{lem:Theta}, it suffices to consider $\vartheta_m^{\trop}$ with $m\in M$.  This allows us to more easily utilize Lemma \ref{lem:CPS}.  We write $\Theta^{\Z}_{m,p}$ to be the set of exponents appearing in $\vartheta_{m,p}$.

 Suppose a $\Lambda$-taut rational broken line $\Gamma$ crosses a wall $(W,f,\nu)$ at time $t$, say with $f\in 1+x^{Pv}y^v\Z\llb x^{Pv}y^v\rrb$ and $\nu\in N$ the primitive element with direction  $Q^{\bullet}v$.  For convenience, we assume this is the only wall crossed at this time (this always holds up to equivalence of scattering diagrams).  Let $K\in \Z_{\geq 0}\cup \{\infty\}$ be the degree of $f$, and define $c\in \Q_{>0}$ by $\nu=cQ^{\bullet}v$.  Using Lemma \ref{lem:Lambda-taut}, when $\Gamma$ crosses $W$, the attached exponent $m_t=-\Gamma'(t)$ changes from $m_{t-\epsilon}$ to \begin{align*}
     m_{t+\epsilon}=\begin{cases}
     m_{t-\epsilon} \quad &\text{if $\Gamma(t-\epsilon)\cdot Q^{\bullet} v<0$}\\
     m_{t-\epsilon}+ cK(m_{t-\epsilon}\cdot Q^{\bullet}v)Pv \quad &\text{if $\Gamma(t-\epsilon)\cdot Q^{\bullet} v>0$.}
 \end{cases}
 \end{align*}

Let $\beta$ be a broken line with initial exponent $m\in M$ and endpoint $\Gamma(t-\epsilon)$ for $\epsilon >0$ small enough that $\Gamma$ does not bend between $t-\epsilon$ and $t$.  For $\epsilon$ sufficiently small, there are three possibilities:
    \begin{enumerate}
        \item $\beta$ can be extended slightly to cross $W$,
        \item $\beta$ has just crossed $W$,
        \item the last straight segment of $\beta$ is parallel to $W$, i.e., perpendicular to $Q^{\bullet}v=\Lambda P v$.
    \end{enumerate}
In each case, we shall write $b_{t-\epsilon}$ to denote the final attached exponent for $\beta$ with endpoint at $\Gamma(t-\epsilon)$ and $b_{t+\epsilon}$ for a certain modified broken line $\beta'$ with the same initial exponent $m$ and endpoint near $\Gamma(t+\epsilon)$.\footnote{We can make the new endpoint arbitrarily close to $\Gamma(t+\epsilon)$ by making $\epsilon$ sufficiently small, so Lemma \ref{lem:CPS} ensures that there is some broken line with the same initial and final exponents which genuinely does end at $\Gamma(t+\epsilon)$.}  We will ensure that 
\begin{align}\label{eq:bt-mt}
b_{t+\epsilon} \cdot \Lambda m_{t+\epsilon} \leq b_{t-\epsilon} \cdot \Lambda m_{t-\epsilon}
\end{align}
hence
\begin{align}\label{eq:t-pm-epsilon-ineq}
\vartheta_{m,\Gamma(t+\epsilon)}^{\trop}(\Lambda m_{t+\epsilon})  &=\inf_{b\in \Theta^{\Z}_{m,\Gamma(t+\epsilon)}} b\cdot \Lambda m_{t+\epsilon}  \nonumber \\
&\leq \inf_{b\in \Theta^{\Z}_{m,\Gamma(t-\epsilon)}} b \cdot \Lambda m_{t-\epsilon} = \vartheta_{m,\Gamma(t-\epsilon)}^{\trop}(\Lambda m_{t-\epsilon}).
\end{align}
A similar procedure (details left to the reader) applies in the reverse direction (crossing from $\Gamma(t+\epsilon)$ to $\Gamma(t-\epsilon)$) to yield the reverse inequality; i.e., \eqref{eq:t-pm-epsilon-ineq} is in fact an equality.  The desired $\Lambda$-constancy along $\Gamma$ then follows.

In Case (1), if we extend $\beta$ to cross and possibly bend at $W$ to obtain $\beta'$ with final exponent $b_{t+\epsilon}$, then we have
    \begin{align}\label{eq:bt}
        b_{t+\epsilon}=b_{t-\epsilon}+\sigma ck(b_{t-\epsilon}\cdot Q^{\bullet}v)Pv
    \end{align}
    for some $0\leq k\leq K$ and $\sigma = \sign(\Gamma(t-\epsilon)\cdot Q^{\bullet} v)$.  In the case $\sigma=-1$, we choose $k=0$ so that $b_{t+\epsilon}=b_{t-\epsilon}$.  Since $m_{t+\epsilon}=m_{t-\epsilon}$ in this case, \eqref{eq:bt-mt} follows trivially.  When $\sigma=+1$, the $\Lambda$-tautness of $\Gamma$ implies $K$ is finite, so we can choose $k=K$.  We then obtain 
    \begin{align}
        b_{t+\epsilon}\cdot \Lambda m_{t+\epsilon} &= \left(b_{t-\epsilon}+cK(b_{t-\epsilon}\cdot Q^{\bullet}v)Pv\right) \cdot \Lambda\left(m_{t-\epsilon}+ cK(m_{t-\epsilon}\cdot Q^{\bullet}v)Pv\right) \label{eq:bt-dot}\\
        &= b_{t-\epsilon}\cdot \Lambda m_{t-\epsilon} +  cK(m_{t-\epsilon}\cdot Q^{\bullet}v)(b_{t-\epsilon}\cdot \Lambda Pv) + cK(b_{t-\epsilon}\cdot Q^{\bullet}v)(Pv\cdot \Lambda m_{t-\epsilon}) \nonumber\\
        &=b_{t-\epsilon}\cdot \Lambda m_{t-\epsilon}\nonumber
    \end{align}
    as desired.  In the last line we used $\Lambda Pv=Q^{\bullet}v$ and $Pv\cdot \Lambda m_{t-\epsilon} = m_{t-\epsilon} \cdot \Lambda^{\top}Pv = -m_{t-\epsilon} \cdot Q^{\bullet} v$. 
    
    Now consider Case (2).  Note that here, $b_{t-\epsilon}$ describes the final exponent of $\beta$ (after it crosses $W$), and $b_{t+\epsilon}$ describes the last exponent just before it crosses $W$.  We have
    $$b_{t-\epsilon}=b_{t+\epsilon}-\sigma ck(b_{t-\epsilon}\cdot Q^{\bullet} v)Pv$$
    for some $0\leq k\leq K$.  The minus sign appears here because the appropriate sign for $\beta$ is opposite that for $\Gamma$ since the broken lines cross $W$ from opposite sides.  Let us modify the bend $\beta$ across $W$ while keeping the other bends the same, choosing $k$ in order to minimize the following: $$b_{t-\epsilon}\cdot \Lambda m_{t-\epsilon}=b_{t+\epsilon}\cdot \Lambda m_{t-\epsilon} -\sigma ck(b_{t-\epsilon}\cdot Q^{\bullet}v)(Pv\cdot \Lambda m_{t-\epsilon}).$$  Since $m_{t-\epsilon}$ and $b_{t-\epsilon}$ are on opposite sides of $W$, and since $Pv\cdot \Lambda m_{t-\epsilon} = -m_{t-\epsilon} \cdot Q^{\bullet} v$, the sign of $(b_{t-\epsilon}\cdot Q^{\bullet}v)(Pv\cdot \Lambda m_{t-\epsilon})$ is positive.  So when  $\sigma=-1$, minimizing this pairing means minimizing $k$, i.e., taking $k=0$, so $b_{t-\epsilon}= b_{t+\epsilon}$.  Since we also have $m_{t-\epsilon} = m_{t+\epsilon}$ in this case, the desired relation \eqref{eq:bt-mt} follows immediately.  On the other hand, if $\sigma = +1$, then we take $k=K$ (which must be finite by $\Lambda$-tautness of $\Gamma$), and \eqref{eq:bt-mt} is checked almost exactly as in \eqref{eq:bt-dot}.  Thus, $b_{t+\epsilon}\cdot \Lambda m_{t+\epsilon}$ is less than or equal to the original $b_{t-\epsilon}\cdot \Lambda m_{t-\epsilon}$ (before the final bend of $\beta$ was modified).

    For Case (3), first note that being perpendicular to $Q^{\bullet}v$ implies that the action of the wall-crossing automorphism on the final monomial of $\beta$ will be trivial.  Thus, Lemma \ref{lem:CPS} and positivity ensure that a broken line with the same initial and final exponents as $\beta$ still exists with endpoint on the other side of $W$.   This will give the same pairing with $\Lambda m_{t+\epsilon}$ as with $\Lambda m_{t-\epsilon}$ because $\Lambda(m_{t+\epsilon}-m_{t-\epsilon})$ is parallel to $\Lambda Pv$,  hence zero on the final exponent of $\beta$.  So broken lines as in Case (3) will not cause $\vartheta_{m,\Gamma(s)}^{\trop}(\Lambda m_s)$ to change when we pass $s=t$.

    The above 3 cases show that, given a broken line $\beta$ ending at $\Gamma(t-\epsilon)$, we can produce a new broken line $\beta'$ with the same initial exponent but ending at $\Gamma(t+\epsilon)$ and giving an equal or lesser pairing with $\Lambda m_s$.  The arguments for each case are easily reversed to go from $t+\epsilon$ to $t-\epsilon$.  The desired $\Lambda$-constancy of $\vartheta_m^{\trop}$ along $\Gamma$ at $s=t$ follows.  Since $t$ was an arbitrary point where $\Gamma$ crosses a wall, and since the constancy is trivial when not crossing a wall (thanks to Lemma \ref{lem:CPS}), we obtain the desired $\Lambda$-constancy along all of $\Gamma$.
\end{proof}

\begin{lem}\label{lem:BL-const-finite}
   Fix an $m\in M$ and a $\Lambda$-taut rational broken line $\Gamma$ such that $\vartheta_{m,\Gamma(t)}^{\trop}(-\Lambda \Gamma'(t))$ is finite for some (hence every) generic $t\in (-\infty,0]$.  Then the $\Lambda$-constancy of $\vartheta_m^{\trop}$ along $\Gamma$ holds even if we replace $\f{D}=\f{D}^{(P,Q)}$ with $\f{D}^k$ for sufficiently large but finite $k$.  Similarly for $\Lambda^{\top}$-constancy along a $\Lambda^{\top}$-taut broken line.

    On the other hand, given $\Lambda$-taut $\Gamma$ as above, suppose $\vartheta_{m,\Gamma(t)}^{\trop}(-\Lambda \Gamma'(t))=-\infty$ for some (hence every) generic $t\in (-\infty,0]$.  Then for each $L<0$, there exists some $k\in \Z_{>0}$ such that $(\vartheta_{m,\Gamma(t)}^{\f{D}^k})^{\trop}(-\Lambda \Gamma'(t))<L$ for all generic $t\in (-\infty,0]$.  Similarly for the $\Lambda^{\top}$-taut analog.
\end{lem}
\begin{proof}
    We shall prove the $\Lambda$-taut case.  The proof of the $\Lambda^{\top}$ case is completely analogous.
    
    For the first claim, let us start at $\Gamma(0)$ with a $\Lambda m_{\Gamma}$-minimizing broken line $\s{P}_0$.  The proof of Lemma \ref{lem:L-const} shows how we can transform $\s{P}_0$ as we decrease $t$ to obtain $\Lambda$-minimizing broken lines $\s{P}_t$ for all $t\in (-\infty,0]$.  Define $k_t\in \Z_{\geq 0}$ so that the final monomial of $\s{P}_t$ is \textit{not} contained in $\s{I}_{k_t+1}$.  It suffices to show that every value of $k_t$ is bounded by some fixed $K \in \Z_{\geq 0}$, because then every $\s{P}_t$ will be a broken line defined for $\f{D}^K$.

    As we decrease $t$, when $\Gamma(t)$ crosses a wall $W$, we saw in the proof of Lemma \ref{lem:L-const} that there are 3 possibilities to consider.  If $\s{P}_{t+\epsilon}$ is parallel to $W$ (what we called Case (3)), then there is a $\Lambda m_{t-\epsilon}$-minimizing broken line $\s{P}_{t-\epsilon}$ on the other side of $W$ with the same final exponent as $\s{P}_t$.  Since the final exponent has not changed, neither does $k_t$.
    
    Now suppose $\s{P}_t$ is not parallel to $W$, so we are in a situation like what we called Case (1) or Case (2).  In these cases, if $\Gamma$ does not bend when crossing $W$, then $\s{P}_t$  also will not bend across $W$.  We thus see that $k_t$ is constant along straight segments of $\Gamma$.

    Since $\Gamma$ only bends for finitely many values of $t$, we see now that there are only finitely many distinct values for $k_t$.  We can thus take $K=\max_t k_t$ to complete the argument.

    The claim regarding $(\vartheta_{m,\Gamma(t)}^{\f{D}^k})^{\trop}(-\Lambda \Gamma'(t))<L$ is proved similarly.  This time though, since $\s{P}_t$ is not $\Lambda m_t$-minimizing, there could be Case (2) situations where $\s{P}_{t+\epsilon}$ has just crossed and bent non-trivially across a wall $W\ni \Gamma(t)$ along which $\Gamma$ does not bend, so the degree for $\s{P}_t$ could change non-trivially even along straight segments of $\Gamma$.  However, this degree-change corresponds to removing a bend as we decrease $t$, thus decreasing $k_t$, so it still suffices to use $k=\max_{t\in \{0,t_1,t_2,\ldots,t_s\}} \{k_{t-\epsilon}\}$ where $t_1,\ldots,t_s$ are the times at which $\Gamma$ bends.
\end{proof}

\subsection{Moving the basepoint}\label{sub:move}

The following utilizes the limiting theta functions as in \S \ref{sub:lim} along with the rational broken line analogs $\lim_{\epsilon \rar 0^+} \Theta_{m,p+\epsilon \mu}$ from \eqref{eq:limT}.

\begin{lem}\label{lem:Lambda-p-val}
Fix $m\in M_{\Q}$.  For any generic $p\in M_{\R}$, we have $\val_{\Lambda p}(\Theta_{m,+})=\val_{\Lambda p}(\Theta_{m,p})$.  That is,
    \begin{align}\label{eq:val-inf}
        \val_{\Lambda p}\left(\Theta_{m,+}\right)=\inf_{\Ends(\Gamma)=(m,p)} m_{\Gamma} \cdot \Lambda p.
    \end{align}
More generally, even if $p$ is not necessarily generic, we have
    \begin{align*}%\label{eq:lim-val-inf}
        \val_{\Lambda p}\left(\Theta_{m,+}\right)=\val_{\Lambda p}\left(\lim_{\epsilon \rar 0^+} \Theta_{m,p+\epsilon \mu}\right)
    \end{align*}
    for any generic $\mu \in M_{\R}$.

    Furthermore, if the infimum in \eqref{eq:val-inf} is attained for some $\Gamma$, then this $\Gamma$ is $\Lambda^{\top}$-taut.  Similarly, even if $p$ is not generic, suppose $\val_{\Lambda p'}\left(\Theta_{m,+}\right)$ is finite for $p'$ in a neighborhood of $p$.  Then for generic $\mu\in M_{\R}$, $\val_{\Lambda p}\left(\lim_{\epsilon \rar 0^+} \Theta_{m,p+\epsilon \mu}\right)$ is attained for some $\Gamma$---that is, for sufficiently small generic $\epsilon>0$, there exists $\Gamma$ with ends $(m,p+\epsilon \mu)$ such that $m_{\Gamma}\cdot \Lambda p= \val_{\Lambda p}\left(\lim_{\epsilon \rar 0^+} \Theta_{m,p+\epsilon \mu}\right)$---and such $\Gamma$ are $\Lambda^{\top}$-taut.
\end{lem}

\begin{proof}

By Lemmas \ref{lem:val-ku} and \ref{lem:Theta}, it suffices to restrict to  $m\in M$ and then work with ordinary theta functions.  We do this for convenience.

Consider $p\in M_{\R}$, not necessarily generic.  We generically translate the initial walls $\f{D}^{(P,Q)}_{\In}$ as suggested in \S \ref{sub:chambers} so that $0$ is contained in a chamber $C^+$ on the positive side of each wall.  Let $\f{D}'$ be the resulting consistent scattering diagram, the walls of which are affine cones.
 
 We see that $\f{D}$ is the ``asymptotic scattering diagram'' of $\f{D}'$, meaning that $\f{D}$ is equivalent to the scattering diagram obtained from $\f{D}'$ by translating each wall so that its apex lies at the origin.  It follows that for any $k>0$ and sufficiently small $\delta >0$ (small enough that the rays $\R_{\geq 0} p$ and $\R_{\geq 0} (p+\delta \mu)$ share a chamber of $\f{D}^k$), if $K\gg 0$ is sufficiently large, then $$\vartheta^{(\f{D}')^k}_{m,K(p+\delta \mu)} = \lim_{\epsilon \rar 0^+} \vartheta^{\f{D}^k}_{m,p+\epsilon \mu} \qquad (\text{mod } \s{I}^k).$$  Furthermore, $\vartheta^{(\f{D}')^k}_{m,K(p+\delta \mu)}$ and $\vartheta^{(\f{D}')^k}_{m,Kp}$ agree (mod $\s{I}^k$) up to the actions of walls parallel to $p$.
 
 This holds for arbitrary $k$, but the ``sufficiently large'' condition on $K$ might become more restrictive as we increase $k$.  We therefore define $\lim_{K\rar \infty} \vartheta_{m,Kp}^{\f{D}'}$ similarly to our definition of  $\lim_{\epsilon \rar 0^+} \vartheta^{(P,Q)}_{m,p+\epsilon \mu}$ in \S \ref{sub:lim}.  The above implies that $\lim_{K\rar \infty} \vartheta_{m,Kp}^{\f{D}'}$ and $\lim_{\epsilon \rar 0^+} \vartheta^{\f{D}}_{m,p+\epsilon \mu}$ only differ by the actions of walls parallel to $p$.   By Lemma \ref{lem:D-skew}, such walls have the form $(W,f,n)$ for $n$ a multiple of $Q^{\bullet}v$, $p\in n^{\perp}= Q^{\bullet}v^{\perp}$, and $f\in 1+x^{Pv}y^v \Z\llb x^{Pv}y^v\rrb$.  Since $Pv\cdot \Lambda p = -Q^{\bullet} v\cdot p = 0$, the actions of these wall-crossing automorphisms do not affect the value of $\val_{\Lambda p}$.  Thus, $$\val_{\Lambda p}\left(\lim_{K\rar \infty} \vartheta_{m,Kp}^{\f{D}'}\right) = \val_{\Lambda p}\left(\lim_{\epsilon \rar 0^+} \vartheta_{m,p+\epsilon \mu}^{\f{D}}\right).$$

For the first claims, it now suffices to prove that $$\val_{\Lambda p}\left(\lim_{K\rar \infty} \vartheta_{m,Kp}^{\f{D}'}\right) = \val_{\Lambda p}\left(\vartheta_{m,+}^{\f{D}}\right)$$ where the right-hand side is by definition the same as $\val_{\Lambda p}\left(\vartheta_{m,+}^{\f{D}'}\right)$.   For the scattering diagram $\f{D}'$, consider the straight rational broken line $\s{P}$ with ends $(p,0)$---i.e., $\s{P}$ has initial exponent $p$ and endpoint $0$, so the support is $\R_{\geq 0} p$.  Here we use Remark \ref{rmk:irrational-exp} since $p$ may be irrational.  By Corollary \ref{cor:straight-bl-taut}, $\s{P}$ is $\Lambda$-taut.  So Lemma \ref{lem:L-const} ($\Lambda$-constancy of tropical theta functions along $\Lambda$-taut broken lines) implies that the equality $\val_{\Lambda p}\left(\vartheta_{m,+}^{\f{D}}\right)=\val_{\Lambda p}\left(\vartheta_{m,p'}^{\f{D}'}\right)$ remains true as we slide the endpoint $p'$ along $\s{P}$ from $0$ to $Kp$ for any $K\in \R_{\geq 0}$. 

Now if $\val_{\Lambda p}\left(\vartheta_{m,+}^{\f{D}}\right)$ is finite, Lemma \ref{lem:BL-const-finite} implies that broken lines for $\vartheta_{m,p'}^{\f{D}'}$ which minimize $\val_{\Lambda p}$ are attained at all generic $p'\in \R_{\geq 0} p$ even for $(\f{D}')^k$ for some fixed finite $k>0$.  So for this fixed $k$ and all sufficiently large (but still finite) $K\gg 0$, we have $$\val_{\Lambda p}\left(\lim_{\kappa\rar \infty} \vartheta_{m,\kappa p}^{\f{D}'}\right)=\val_{\Lambda p}\left(\vartheta_{m,Kp}^{(\f{D}')^k}\right)=\val_{\Lambda p}\left(\vartheta_{m,Kp}^{\f{D}'}\right)=\val_{\Lambda p}\left(\vartheta_{m,+}^{\f{D}}\right).$$

Similarly, suppose $\val_{\Lambda p}\left(\vartheta_{m,+}^{\f{D}}\right)=-\infty$. Then by the second part of Lemma \ref{lem:BL-const-finite}, for any $L<0$, we can find finite $k>0$ such that, for all $K\gg 0$, we have
\begin{align*}
    \val_{\Lambda p}\left(\lim_{\kappa\rar \infty} \vartheta_{m,\kappa p}^{\f{D}'}\right) 
    \leq \val_{\Lambda p}\left(\vartheta_{m,Kp}^{(\f{D}')^k}\right)<L.
\end{align*}
Since $L$ is arbitrary, it follows that $ \val_{\Lambda p}\left(\lim_{\kappa\rar \infty} \vartheta_{m,\kappa p}^{\f{D}'}\right) =-\infty$.  Thus, in any case, we obtain
$\val_{\Lambda p}\left(\lim_{K\rar \infty} \vartheta_{m,Kp}^{\f{D}'}\right) =\val_{\Lambda p}\left(\vartheta_{m,+}^{\f{D}}\right)$
as desired.

For the final claims, recall from Theorem \ref{thm:min-taut} that a $\val_{\Lambda p,p}$-minimizing broken line $\Gamma$ must be $(\Lambda p)$-taut.  We know from Lemma \ref{lem:Lambda-T-taut} that this implies $\Gamma$ is $\Lambda^{\top}$-taut.  This proves the claim for generic $p$.  For the non-generic cases, we choose a generic basepoint $p'=p+\epsilon \mu$ sharing a domain of linearity of $\Theta_{m,+}^{\trop}$ with $p$ (utilizing our finiteness assumption).  Then the $\val_{\Lambda p',p'}$-minimizing broken line $\Gamma$ (which exists by Lemma \ref{lem:inf=min}) will also minimize $\val_{\Lambda p,p'}$, and Lemma \ref{lem:Lambda-T-taut} shows that this $\Gamma$ is $\Lambda^{\top}$-taut since it $(\Lambda p')$-taut.
\end{proof}

\begin{rem}\label{rmk:Lambda-top-p-val}
    A symmetric argument yields an analog of Lemma \ref{lem:Lambda-p-val} in which we switch the roles of $\Lambda$ and $\Lambda^{\top}$ and replace $C^+$ with $C^-$.
\end{rem}

\subsection{Proof of Theta Reciprocity}\label{sub:proof}

\begin{thm}[Theta Reciprocity]\label{thm:taut-pairing}
    Claims \ref{conj: theta1} and \ref{conj: theta2} hold.
\end{thm}

\begin{proof}
    By Lemma \ref{lem:theta2-theta1}, it suffices to prove Claim \ref{conj: theta2}.  That is, we assume $(P,Q)$ has a rational $\Lambda$-structure and then prove for all $m_1,m_2\in M_{\Q}$ that
\begin{align}\label{eq:theta-recip}
    \Theta_{m_2,+}^{\trop}(\Lambda m_1)
= \Theta_{m_1,-}^{\trop}(\Lambda^{\top} m_2)
\end{align}
assuming both sides are finite.  Furthermore, Lemma \ref{lem:sat} allows us to assume that $(P,Q)$ is a saturated-injective seed datum, so we may assume that $\f{D}(P,Q)$ contains a positive chamber $C^+$ and a negative chamber $C^-$.

Additionally, since tropical theta functions and valuations are both upper semicontinuous (cf. \S \ref{sub:trop-fun} and Lemma \ref{lem:v-upper-semi}), we may assume that $\Lambda m_1$ is in the interior of the locus where $\Theta_{m_2,+}^{\trop}$ is finite and $\Lambda^{\top} m_2$ is in the interior of the locus where $\Theta_{m_1,-}^{\trop}$ is finite (since upper semicontinuity then implies the extension to the boundary cases).

Now let $\Gamma_0$ be a straight broken line with initial exponent $m_1$ and endpoint in $C^+$.  Lemma \ref{lem:Lambda-taut} implies that $\Gamma_0$ is $\Lambda$-taut.  So by Lemma \ref{lem:L-const} ($\Lambda$-constancy along $\Lambda$-taut broken lines), we have 
\begin{align}\label{eq:1st-pairing}
    \Theta_{m_2,+}^{\trop}(\Lambda m_1) = \inf_{\Ends(\Gamma)=(m_2,p)} m_{\Gamma} \cdot \Lambda m_1
\end{align}
for any generic $p\in m_1+C^+$.  Choosing such $p$ to be sufficiently close to $m_1$ (so that $\Lambda p$  shares a domain of linearity of $\Theta_{m_2,+}^{\trop}$ with $\Lambda m_1$, using our finiteness assumptions), we see as in the proof of Lemma \ref{lem:Lambda-p-val} that the infimum in \eqref{eq:1st-pairing} is attained for a $\Lambda^{\top}$-taut broken line $\Gamma_+$.  Let $m_+\in M$ be the final exponent of $\Gamma_+$.  By construction, $$\Theta_{m_2,+}^{\trop}(\Lambda m_1)=m_+\cdot \Lambda m_1 = m_1 \cdot \Lambda^{\top} m_+.$$  

Now note that
\begin{align*}
    \Theta_{m_1,p}^{\trop}(\Lambda^{\top}m_+) = \inf_{\Ends(\Gamma)=(m_1,p)} m_{\Gamma} \cdot \Lambda^{\top} m_+ \leq m_1 \cdot \Lambda^{\top} m_+ 
\end{align*}
because the set of broken lines considered in this infimum includes the straight broken line with ends $(m_1,p)$.  We now have
\begin{align*}
   \Theta_{m_2,+}^{\trop}(\Lambda m_1) = m_1\cdot \Lambda^{\top} m_+ \geq \Theta_{m_1,p}^{\trop}(\Lambda^{\top}m_+).
\end{align*}

By Lemma \ref{lem:L-const} ($\Lambda^{\top}$-constancy of $\Theta_{m_1,p}^{\trop}$ along $\Lambda^{\top}$-taut broken lines) and the fact that $\Gamma_+$ is $\Lambda^{\top}$-taut with ends $(m_2,p)$, we deduce that $$\Theta_{m_1,p}^{\trop}(\Lambda^{\top} m_+) = \Theta_{m_1,\Gamma_+(t)}^{\trop}(\Lambda^{\top} m_2)$$ for all $t\ll 0$.  By Lemma \ref{lem:BL-const-finite}, a minimizing broken line $\beta$ for $\Theta_{m_1,\Gamma_+(t)}^{\trop}(\Lambda^{\top} m_2)$ (or, if $\Theta_{m_1,\Gamma_+(t)}^{\trop}(\Lambda^{\top} m_2)$ were $-\infty$, a broken line $\beta_L$ with ends $(m_{\beta_L},\Gamma_+(t))$ such that $m_{\beta_L}\cdot \Lambda^{\top}m_2<L$ for given $L\in \R$) is attained at a finite order $k$ which does not depend on $t$, so the equality also holds in the limit as $t\rar -\infty$.  Equivalently, letting $\mu=\Gamma(t)$ for some fixed $t\ll 0$, we have $\Theta_{m_1,\Gamma_+(t)}^{\trop}(\Lambda^{\top} m_2) = (\lim_{\epsilon \rar 0^+} \Theta_{m_1,m_2+\epsilon \mu})^{\trop}(\Lambda^{\top} m_2)$, which by Lemma \ref{lem:Lambda-p-val} (modified as in Remark \ref{rmk:Lambda-top-p-val}) equals $\Theta_{m_1,-}(\Lambda^{\top} m_2)$.  Thus,
\begin{align*}
    \Theta_{m_2,+}^{\trop}(\Lambda m_1) \geq \Theta_{m_1,-}^{\trop}(\Lambda^{\top} m_2).
\end{align*}
A completely symmetric argument proves the reverse inequality.  The claim follows.
\end{proof}

\subsection{Convexity along broken lines}\label{sub:BL-convex}

Here we apply our characterizations of tropical theta functions to recover some convexity results

\subsubsection{Convexity on the finite locus}

Following \cite[Def. 8.2]{GHKK}, we say that a piecewise-linear function $\varphi:M_{\R}\rar \R$ is \textbf{convex along broken lines} (called min-convex in loc.~cit.) if, for all broken lines $\Gamma$ (with generic endpoint), we have that $D_{\Gamma(t)}\varphi\cdot m_t$ is non-decreasing as we increase the generic time $t$.  Here, $D_{\Gamma(t)}\varphi$ is the differential of $\varphi$ at $\Gamma(t)$, viewed as an element of $N_{\R}$, and $m_t$ is the exponent attached to $\Gamma$ at time $t$. 
 Restricting to connected open subsets of broken lines, we similarly define convexity along broken line segments. 

Given $f\in \wh{A}^{\can}$, we would like to apply the above definition to $f^{\trop}\circ \Lambda: p\mapsto f_p^{\trop}(\Lambda p)$.  In general though, $f^{\trop}\circ \Lambda$ takes values in $\R^{\min}$, possibly including $-\infty$, so it is not clear how to define the differential here. 
 We therefore focus on the locus where $f_p^{\trop}$ is finite, saving our limited discussion of the infinite locus for \S \ref{subsub: infinite}.

\begin{prop}\label{prop: BL}
    Let $m\in M_{\Q}$.  The map $\Theta_m^{\trop}\circ \Lambda: p\mapsto \Theta_{m,p}^{\trop}(\Lambda p)$ is convex along broken line segments that are contained in the locus where $\Theta_m^{\trop}\circ \Lambda$ is finite.
\end{prop}

We note that a version of this proposition (for tropicalizations of regular functions) was conjectured in \cite[Conj. 8.11]{GHKK}, then proved in \cite[Lem. 15.6]{KY} for affine log Calabi-Yau varieties containing a Zariski-dense algebraic torus (including affine cluster varieties)---in fact, \cite[\S 10.7]{KY2} shows that the torus assumption is unnecessary. Our characterization of tropical theta functions in Lemma \ref{lem:Lambda-p-val} allows us to give the following direct proof in the cluster setting.

\begin{proof}
   As usual, by Lemmas \ref{lem:val-ku} and \ref{lem:Theta}, it suffices to consider $\vartheta_m^{\trop}$ for $m\in M$.

    Our approach here is similar to that of Lemma \ref{lem:L-const}.  Consider a broken line $\Gamma$ in $M_{\R}$ with attached exponents denoted $m_t$.  Let $p^{\pm}=\Gamma(t\pm \epsilon)$ for small $\epsilon>0$.  By Lemma \ref{lem:Lambda-p-val}, $\vartheta_{m,p^-}^{\trop}(\Lambda p^-)=\inf_{\Ends(\beta)=(m,p^-)} m_{\beta} \cdot \Lambda p^-$ is attained for a $\Lambda^{\top}$-taut broken line $\beta$.    Let $(W,f,\nu)$ be a (possibly trivial) wall containing $\Gamma(t)$, say with scattering function $f\in 1+x^{Pv}y^v\Z\llb x^{Pv}y^v\rrb$ and $\nu\in N$.  Let $K\in \Z_{\geq 0}\cup \{\infty\}$ be the degree of $f$, and define $c\in \Q_{>0}$ by $\nu=cQ^{\bullet}v$.  We have the following three possibilities:
    \begin{enumerate}
        \item $\beta$ can be extended slightly to cross $W$,
        \item $\beta$ has just crossed $W$,
        \item the last straight segment of $\beta$ is parallel to $W$, i.e., perpendicular to $Q^{\bullet}v=\Lambda P v$.
    \end{enumerate}
    In any case, we have
    \begin{align*}
     m_{t+\epsilon}=
     m_{t-\epsilon}+ ck\sigma(m_{t-\epsilon}\cdot Q^{\bullet}v)Pv
    \end{align*}
    where $\sigma =  \sign(\Gamma(t-\epsilon)\cdot Q^{\bullet} v)$ and $0\leq k\leq K$.  We will modify $\beta$, with final attached exponent $b_{t-\epsilon}$, to obtain a new broken line ending at $\Gamma(t+\epsilon)$ with final attached exponent $b_{t+\epsilon}$ satisfying $\Lambda^{\top} b_{t+\epsilon} \cdot m_{t+\epsilon} \geq \Lambda^{\top}b_{t-\epsilon} \cdot m_{t-\epsilon}$.  Note that $D_{p^-} (\vartheta_{m,p^-}^{\trop}\circ \Lambda) = \Lambda^{\top} b_{t-\epsilon}$.
    
    For Case (3), we use the same argument as in the proof of Lemma \ref{lem:L-const}, using the consistency of the scattering diagram and applying Lemma \ref{lem:CPS} to produce a broken line on the other side of $W$ with the same final exponent $b_{t+\epsilon}=b_{t-\epsilon}.$  Since $b_{t\pm \epsilon}\cdot \Lambda Pv=0$ here, we have $b_{t+\epsilon}\cdot \Lambda m_{t+\epsilon} = b_{t-\epsilon} \cdot \Lambda m_{t-\epsilon}$, or equivalently, $\Lambda^{\top} b_{t+\epsilon}\cdot  m_{t+\epsilon} = \Lambda^{\top} b_{t-\epsilon} \cdot m_{t-\epsilon}$.
    
    For Cases (1) and (2), by extending $\beta$ across $W$ in Case (1) or truncating $\beta$ before $W$ in Case (2), we obtain $\beta'$ with final exponent $b_{t+\epsilon}$ satisfying $b_{t+\epsilon}=b_{t-\epsilon}+ck'\sigma (b_{t-\epsilon}\cdot Q^{\bullet} v)Pv$ for some $0\leq k'\leq K$.  Then
    \begin{align*}
        \Lambda^{\top} b_{t+\epsilon}\cdot m_{t+\epsilon} &= (b_{t-\epsilon}+ck'\sigma (b_{t-\epsilon}\cdot Q^{\bullet} v)Pv) \cdot \Lambda (m_{t-\epsilon}+ ck\sigma(m_{t-\epsilon}\cdot Q^{\bullet}v)Pv) \\
        &=b_{t-\epsilon}\cdot \Lambda m_{t-\epsilon} + c(k-k')\sigma (m_{t-\epsilon}\cdot Q^{\bullet} v)(b_{t-\epsilon}\cdot Q^{\bullet}v). 
    \end{align*}
    For convenience, let $\alpha \coloneqq (m_{t-\epsilon}\cdot Q^{\bullet} v)(b_{t-\epsilon}\cdot Q^{\bullet}v)$, so the above equation can be written as
    \begin{align*}%\label{eq:bte}
        b_{t+\epsilon}\cdot \Lambda m_{t+\epsilon} = b_{t-\epsilon}\cdot \Lambda m_{t-\epsilon} + c(k-k')\sigma\alpha.
    \end{align*}
Thus, our goal is to choose $k'$ so that $c(k-k')\sigma \alpha \geq 0$.

This is simple in Case (1) since we have control over $k'$.  In this case, $b_{t- \epsilon}$ and $m_{t- \epsilon}$ lie on the same side of $(Q^{\bullet}v)^{\perp}$, so $\alpha>0$.  So if $\sigma=+1$, we can choose $k'=0$, and for $\sigma=-1$, we choose any $k'\geq k$.

    Next consider Case (2).  Here we have $\alpha < 0$.  
    Suppose $\sigma=+1$.  Since $\beta$ is assumed to be $\Lambda^{\top}$-taut,  we must have $k'=K\geq k$, hence $c (k-k')\sigma \alpha \geq 0$.  Similarly, if $\sigma=-1$, then $\beta$ being $\Lambda^{\top}$-taut implies $k'=0\leq k$, and again $c (k-k')\sigma \alpha \geq 0$.
    
    Thus, given that $\beta$ is $\Lambda p^-$-minimizing, there always exists a broken line $\beta'$ ending at $\Gamma(t+\epsilon)$ such that the final attached exponent $b_{t+\epsilon}$ satisfies $\Lambda^{\top} b_{t+\epsilon} \cdot m_{t+\epsilon} \geq \Lambda^{\top} b_{t-\epsilon} \cdot m_{t-\epsilon}$.  If the infimum determining $\vartheta_{m,p^+}^{\trop}(\Lambda p^+)$ is attained for this $\beta'$, then we are done.  If not, there must be a different broken line $\beta''$ ending at $p^+$ which is $\Lambda p^+$-minimizing.  So the final attached exponent $b'_{t+\epsilon}$ of $\beta''$ yields a smaller value than $b_{t+\epsilon}$ when paired with $\Lambda p^+$.  Then $\Lambda^{\top}(b'_{t+\epsilon}-b_{t+\epsilon})$ will be a linear function which is $0$ along $W$ (since both $\Lambda^{\top}b'_{t+\epsilon}$ and $\Lambda^{\top} b_{t+\epsilon}$ agree with $\Lambda^{\top}b_{t-\epsilon}$ on $W$) and negative on the new side of the wall, hence positive at $p^-$.  Since $p^-$ and $m_{t+\epsilon}$ lie on the same side of $W$, this means that the actual value $\Lambda^{\top} b'_{t+\epsilon}\cdot m_{t+\epsilon}$ of $[D_{p^+}(\vartheta_{m}^{\trop}\circ \Lambda)]\cdot m_{t+\epsilon}$ is \textit{greater than} the value computed using $b_{t+\epsilon}$, so the desired non-decreasing property holds.
\end{proof}

The following can be viewed as compactifying the BL-convexity in Proposition \ref{prop: BL} to include a convexity condition at $t=-\infty$.

\begin{cor}\label{lem:BL-convex-infty}
    Let $\Gamma$ be a broken line with ends $(m,p)$. 
 Let $b\in M$, and assume $\Lambda m$ lies in the interior of the cone where $\vartheta^{\trop}_{b,+}$ is finite.  Further assume that $\Gamma$ does not pass through any points where $\vartheta_{b,+}^{\trop}\circ \Lambda$ equals $-\infty$.  Then $$\vartheta_{b,+}^{\trop}(\Lambda m)\leq [D_p(\vartheta_b^{\trop}\circ \Lambda)](m_{\Gamma}).$$
\end{cor}
\begin{proof}
 By Lemma \ref{lem:Lambda-p-val}, $\vartheta_{b,+}^{\trop}(\Lambda m)=(\lim_{\epsilon \rar 0^+} \vartheta_{b,m+\epsilon \mu})^{\trop}(\Lambda m)$ for any generic $\mu\in M_{\R}$. We choose $\mu=\Gamma(t)-m$ for some $t\ll 0$.  For convenience, let $\Theta\subset M$ denote the set of exponents appearing in $\lim_{\epsilon \rar 0^+} \vartheta_{b,m+\epsilon \mu}$, so we have
    \begin{align}\label{eq:theta-inf}
        \vartheta_{b,+}^{\trop}(\Lambda m) = \inf_{u\in \Theta} u\cdot \Lambda m.
    \end{align}
    By Lemma \ref{lem:inf=min}, $\inf_{u\in \Theta} u\cdot \Lambda m$ is attained for some $u_0\in \Theta$.  For generic $t$, let $\Theta_t\subset M$ be the set of exponents appearing in $\vartheta_{b,\Gamma(t)}$.  It follows from the definition of $\lim_{\epsilon \rar 0^+} \vartheta_{b,m+\epsilon \mu}$ (cf. \S \ref{sub:lim}) that
    \begin{align}\label{eq:bigcup-bigcap}
        \Theta = \bigcup_{t_0\leq 0} \bigcap_{t\leq t_0} \Theta_t.
    \end{align}

    By assumption, $\vartheta_{b,+}^{\trop}(\Lambda \Gamma(t))$ is finite for all $t\ll 0$, and $m$ will lie in the same domain of linearity of  $\vartheta_{b,+}^{\trop}\circ \Lambda$ as $\Gamma(t)$ for $t\ll 0$.  Now Lemma \ref{lem:Lambda-p-val} implies that we have $[D_{\Gamma(t)}(\vartheta_b^{\trop}\circ \Lambda)]=\Lambda^{\top} u_0$ for $t\ll 0$, hence \begin{align*}
    \vartheta_{b,+}^{\trop}(\Lambda m) = u_0\cdot \Lambda m &= [D_{\Gamma(t)}(\vartheta_b^{\trop}\circ \Lambda)](m_{t}) \qquad \text{(for $t\ll 0$ so $m_t=m$)} \\ & \leq [D_p(\vartheta_b^{\trop}\circ \Lambda)](m_{\Gamma})    
    \end{align*}
    where the last inequality follows from the convexity along broken lines of $\vartheta_b^{\trop}\circ \Lambda$ (Proposition \ref{prop: BL}).
\end{proof}

\subsubsection{Convexity on the infinite locus}\label{subsub: infinite}

In the proof of Theorem \ref{thm:taut-pairing}, one might hope to apply Corollary \ref{lem:BL-convex-infty} with $b=m_2$, $m=m_1$, $p\in m_1+C^+$ very near $m_1$, and $\Gamma_+$ the $\val_{\Lambda m_1}$-minimizing broken line with ends $(m_2,p)$ and final exponent $m_+$ to deduce $$\Theta_{m_2,+}^{\trop}(\Lambda m_1)\leq [D_p(\Theta_{m_2}^{\trop}\circ \Lambda)](m_{\Gamma})=m_{\Gamma}\cdot \Lambda^{\top} m_+ $$
hence
$$\Theta_{m_2,+}^{\trop}(\Lambda m_1)\leq \inf_{\Ends(\Gamma)=(m_1,p)} m_{\Gamma} \cdot \Lambda^{\top}m_+ = \Theta_{m_1,p}^{\trop}(\Lambda^{\top}m_+).$$
This would yield the inequality $\Theta_{m_2,+}^{\trop}(\Lambda m_1) \leq \Theta_{m_1,-}^{\trop}(\Lambda^{\top} m_2)$ without the need for assuming that $\Theta_{m_1,-}^{\trop}(\Lambda^{\top} m_2)$ is finite.

The flaw with this argument is that, a priori, the broken lines $\Gamma$ with ends $(m_1,p)$ could pass through the locus where $\Theta_{m_2,+}^{\trop}=-\infty$, so we lose the notion of convexity along broken lines.  We have attempted to extend the definition of convexity along broken lines to these cases but have been unable to prove that the corresponding convexity claims are satisfied. 
E.g., it would suffice to prove that the locus $S\coloneqq \{\Theta_{m_2,+}^{\trop}>-\infty\}$ is broken line convex (i.e., any broken line segment with endpoints in $S$ is entirely contained in $S$), but we could not find a proof despite expecting this to be true.

We note that an extension of Claim \ref{conj: theta1} to drop the assumption that both sides are finite would be a useful result.  In particular, one is often interested in knowing, for a seed datum $(P,Q)$, whether or not $A^{\midd}_{(P,Q)}$ is equal to $A^{\can}_{(P,Q)}$.  This equality is equivalent to $\val_n(\vartheta_{m,+})$ being finite for all $m$ and $n$, so theta reciprocity (without the finiteness hypothesis) would imply that it suffices to check the chiral (Langlands) dual seed datum.  

In particular, recall from Example \ref{eg:clusterA} that $(P,Q)$ determines corresponding ``$\s{A}$-type'' and ``$\s{X}$-type'' seed data $(B,\Id)$ and $(\Id,B^\top)$, respectively.  We shall write $A^{\midd}_{+,(B,\Id)}$ as $\s{A}^{\midd}_+$ and $A^{\midd}_{+,(\Id,B^{\top})}$ as $\s{X}_+^{\midd}$.  One analogously defines $\s{A}_+^{\can}$ and $\s{X}_+^{\can}$.  For the original seed $(P,Q)$, we write $A_{+,(P,Q)}^{\midd}$ and $A_{+,(P,Q)}^{\can}$ as $\s{U}_+^{\midd}$ and $\s{U}_+^{\can}$, respectively.

By Theorem \ref{thm: linearmorphism}, the linear morphisms of \eqref{eq:ensemble} induce homomorphisms
\begin{align}\label{eq:rhoPQ}
    \s{X}_+^{\can}\stackrel{\rho_P}{\longrightarrow} \s{U}_+^{\can} \stackrel{\rho_{Q^{\top}}}{\longrightarrow} \s{A}_+^{\can},
\end{align}
and these homomorphisms take theta functions to theta functions.  As a consequence, if $\s{U}_+^{\midd}=\s{U}_+^{\can}$, it follows that $\s{X}_+^{\midd}=\s{X}_+^{\can}$.

Now suppose Claim \ref{conj: theta1}, hence Proposition \ref{lem:D-move}, hold without the finiteness assumptions. Then under the integrality assumptions on $D$ and $B^{\bullet}$, $A^{\midd}_{+,(Q,P)}=A^{\can}_{+,(Q,P)}$.  The corresponding $\s{X}$-type seed datum for $(Q,P)$ is $(\Id,B)$, and the chiral-Langlands dual is the $\s{A}$-type seed datum $(B,\Id)$, so it would follow that $\s{A}_{(P,Q)}^{\midd}=\s{A}_{(P,Q)}^{\can}$ as well.

Similar manipulations of theta reciprocity without the finiteness hypothesis would imply that if $\s{A}_+^{\midd}=\s{A}_+^{\can}$ for $(P,Q)$, then $\s{A}_-^{\midd}=\s{A}_-^{\can}$ for the Langlands dual seed datum $(Q,P)$.

\subsection{Example: Log Calabi-Yau surfaces}\label{sub:Looijenga-Pair}

Here we consider the special case of a seed datum $(P,Q)$ where $\rank(M)=\rank(N)=2$.  We further assume that $(P,Q)$ is a skew-symmetric seed datum, that $Pe_i$ and $Qe_i$ are primitive for each $i$, and that there exists a $\Lambda$-structure $\Lambda$ which is given with respect to some fixed basis for $M$ and dual basis for $N$ by the matrix $\left(\begin{matrix}
    0 & 1 \\ -1 & 0
\end{matrix}\right)$.\footnote{These further assumptions do not result in a significant loss of generality.  Given $k\in \Z_{\geq 1}$, a wall $(W,f,kn)$ is equivalent (up to identifying coefficient variables) to $k$ walls $(W,f,n)$.  We may apply this to always assume the $Qe_i$'s are primitive.  Additionally, $(W,f,n)$ is equivalent to $(W,f,-n)$ because the sign $s$ of a wall-crossing as in \eqref{eq:Egamma} also changes when we change the sign of $n$ (only the side of the wall considered to be positive changes).  The exponents $Pv$ can also be rescaled using some equivalence of scattering diagrams and specialization of coefficients like in the proof of Proposition \ref{lem:D-move}.  With these modifications, we can assume that each $Pe_i$ is primitive and that $Q=\Lambda P$ for $\Lambda$ as specified.}  The resulting scattering diagrams $\f{D}=\f{D}(P,Q)$ in $M_{\R}$ are essentially those associated to log Calabi-Yau surfaces in \cite{GHK1} (up to ``moving worms'' as in \cite[\S 3.2-3.3]{GHK1} and an identification of the coefficients as in \cite[Thm. 5.5]{GHK3}).  Let $m\in M$.  In the next several paragraphs, we will explicitly describe  the following two functions:
\begin{center}
\hfill\begin{minipage}{.4\textwidth}
\eqn{\vartheta_{m,+}^{\trop}\circ \Lambda:M_{\R}&\to \R^{\min}\\
 u&\mapsto \vartheta_{m,+}^{\trop}(\Lambda u)}
\end{minipage}\hfill and \hfill
\begin{minipage}{.4\textwidth}
\eqn{\val_{\Lambda m}:M_{\R}&\to \R^{\min} \\
   u&\mapsto \val_{\Lambda m}(\vartheta_{u,+}) .}
   \end{minipage}\hfill\null
\end{center}

\null

Let us first consider $\vartheta_{m,+}^{\trop}(\Lambda u) = \val_{\Lambda u}(\vartheta_{m,+})$.  By Lemma \ref{lem:Lambda-p-val}, this is equal to $\inf_{\Gamma} m_{\Gamma} \cdot \Lambda u$, where the infimum is over all broken lines $\Gamma$ with ends $(m,p)$ for $p$ very near the ray through $u$.  Note in general that $m_t\cdot \Lambda \Gamma(t)$ is constant for $t\in (-\infty,0]$; cf. \cite[Lem. 5.3]{CGMMRSW}.  We call this the $\Lambda$-\textbf{momentum} of $\Gamma$, denoted $\mu(\Gamma)$.

For $m_t=(a,b)$ and $\Gamma(t)=(c,d)$, the $\Lambda$-momentum of $\Gamma$ will be $\mu(\Gamma)=m_t\cdot \Lambda \Gamma(t)=(a,b)\cdot (d,-c) = ad-bc$.  It follows $\mu(\Gamma)$ is positive when $\Gamma$ is wrapping counterclockwise around the origin and negative when wrapping clockwise.

Furthermore, if $\Gamma(0)=p\in \R_{\geq 0} u$, then for $\kappa \in \R_{>0}$ such that $u=\kappa p$, we have $m_{\Gamma}\cdot \Lambda u=\kappa \mu(\Gamma)$.  Thus, for $\mu(\Gamma)>0$ (so $\Gamma$ wrapping counterclockwise), minimizing $m_{\Gamma}\cdot \Lambda u$ will mean always bending away from $0$ as much as possible in order to minimize $\kappa $ (so bending maximally on outgoing walls, but not bending on the incoming parts of incoming walls), while for $\mu(\Gamma)<0$ (so $\Gamma$ wrapping clockwise), it will mean always bending towards $0$ as much as possible in order to maximize $\kappa $ (so bending maximally on the incoming parts of incoming walls but otherwise not bending).

Indeed, consider a wall $(W,f\in 1+x^{Pv}y^v\Z\llb x^{Pv}y^v\rrb,cQv)$, $c\in \Q_{>0}$.  Since $Q=\Lambda P$ (and $\Lambda$ can be viewed as $90^{\circ}$-clockwise rotation), we see that, for an outgoing wall, $\Gamma$ passes from the positive side of the wall to the negative side exactly if it is wrapping clockwise around the origin, i.e., when $\mu(\Gamma)$ is negative, so the condition of taking no bend on outgoing walls here agrees with the $\Lambda^{\top}$-tautness condition of Lemma \ref{lem:LambdaT-taut}.  Analogous analyses apply to incoming rays and to broken lines wrapping counterclockwise (i.e., with $\mu(\Gamma)>0$) to see that the above conditions again agree with $\Lambda^{\top}$-tautness.

We note that, unlike higher-dimensional cases, it is clear in the present setting that if a broken line intersects a given ray $\rho$ at least $k$ times, then it will continue to intersect $\rho$ at least $k$ times even when modified to bend more towards the origin.  It follows that, when considering $\val_{\Lambda u}(m)=\inf_{\Gamma} m_{\Gamma} \cdot \Lambda u$, it suffices to restrict to $\Lambda^{\top}$-taut broken lines even when the valuation ends up being $-\infty$.

Now let us consider $\val_{\Lambda m}(\vartheta_{u,+})$.  By Claim \ref{conj: theta2}, this is equal to $\val_{\Lambda^{\top} u} (\vartheta_{m,-})$.  Then by Lemma \ref{lemma: swap}, $\vartheta^{(P,Q)}_{m,-}=\vartheta^{(P,-Q)}_{m,+}$.  Note that $\Lambda$ being a $\Lambda$-structure for $(P,Q)$ is equivalent to $\Lambda^{\top}$ being a $\Lambda$-structure for $(P,-Q)$, and recall that $\f{D}(P,Q)=\f{D}(P,-Q)$ except for reversing our identification of which side of a wall is considered positive.  So the above analysis of $\val_{\Lambda u}(\vartheta_{m,+}^{(P,Q)})$ may be similarly applied to describe $\val_{\Lambda^{\top} u}(\vartheta_{m,+}^{(P,-Q)})=\val_{\Lambda m}(\vartheta_{u,+})$ by simply replacing $\Lambda$ with $\Lambda^{\top}$.  We define $\mu^{\top}(\Gamma)=m_t\cdot \Lambda^{\top} \Gamma(t)$ the be the $\Lambda^{\top}$-\textbf{momentum} of $\Gamma$.

This is all summarized in the following proposition:

\begin{prop}
    Fix $(P,Q)$ and $\Lambda$ as above.  Let us modify $\f{D}(P,Q)$ by breaking each initial wall apart into an incoming ray and an outgoing ray.  Fix $m,u\in M$.

    Then $\vartheta_{m,+}^{\trop}(\Lambda u)=\inf_{\Gamma} \mu(\Gamma)$, where the infimum is over all $\Lambda^{\top}$-taut broken lines $\Gamma$ with ends $(m,u)$, along with broken lines with initial exponent $m$ whose final direction $\Gamma'(0)$ is parallel to $u$.\footnote{Broken lines with initial exponent $m$ and $\Gamma'(0)$ parallel to $u$, if they exist, contribute $0$ to the infimum.}  Here, a $\Lambda^{\top}$-taut broken line $\Gamma$ with $\mu(\Gamma)>0$ will wrap counterclockwise around $0$ and will always bend away from $0$ as much as possible, i.e., bending maximally on outgoing rays while not bending on incoming rays.  A $\Lambda^{\top}$-taut broken line $\Gamma$ with $\mu(\Gamma)<0$ will wrap clockwise around $0$ and will always bend towards $0$ as much as possible, i.e., bending maximally on incoming rays while not bending on outgoing rays.

    Similarly, $\val_{\Lambda m}(u)=\inf_{\Gamma} \mu^{\top}(\Gamma)$, where the infimum is over all $\Lambda$-taut broken lines $\Gamma$ with ends $(m,u)$, along with broken lines with initial exponent $m$ whose final direction $\Gamma'(0)$ is parallel to $u$.  A $\Lambda$-taut broken line $\Gamma$ with $\mu^{\top}(\Gamma)>0$ will wrap clockwise around $0$ and will always bend away from $0$ as much as possible, i.e., bending maximally on outgoing rays while not bending on incoming rays.  A $\Lambda$-taut broken line $\Gamma$ with $\mu^{\top}(\Gamma)<0$ will wrap counterclockwise around $0$ and will always bend towards $0$ as much as possible, i.e., bending maximally on incoming rays while not bending on outgoing rays.
\end{prop}

In particular, the negative fibers of $\val_{\Lambda m}$ and $\vartheta_{m,+}^{\trop}$ are broken lines which always bend as closely to $0$ as possible.  As explained in \cite[\S 3.5]{Man1}, these broken exactly coincide with the straight lines for $U^{\trop}$, where by $U^{\trop}$ we mean $M_{\R}$ equipped with the canonical integral linear structure defined in \cite{GHK1}.  In \cite[\S 3.7]{Man1} and \cite[\S 4]{LaiZhou}, these straight lines are used to construct partial compactifications of the mirror (the former via local charts, the latter via a $\Proj$ construction).  If $L_m$ is a counterclockwise-wrapping $\Lambda$-taut broken line with initial exponent $m$, then the boundary of the associated partial compactification in \cite{Man1,LaiZhou} will consist of a single divisor $D_{\Lambda m}$ whose corresponding valuation agrees with our $\val_{\Lambda m}$.

\bibliographystyle{alpha}  % Here the bibliography 		     %
\bibliography{biblio.bib}        % is inserted.			     %
\index{Bibliography@\emph{Bibliography}}%

\end{document}